\documentclass[10pt,oneside,english]{amsart}
\usepackage[T1]{fontenc}
\usepackage{color}
\usepackage{amsthm}
\usepackage{bm}
\usepackage{amstext}
\usepackage{amssymb}

\makeatletter
\numberwithin{equation}{section} 
\numberwithin{figure}{section} 

\usepackage{amsmath}
\usepackage{amsfonts}
\usepackage{amsthm}
\usepackage{amssymb}
\usepackage{amscd}
\usepackage[all]{xy}
\usepackage{enumerate}
\usepackage{mathrsfs}
\usepackage{graphicx, bm, color, ulem}
\usepackage{accents}

\makeatletter
\def\widebar{\accentset{{\cc@style\underline{\mskip15mu}}}}
\makeatother

\def\ft#1{{\mathsf #1}}
\def\chow{{\mathscr X}}
\def\hcoY{{\mathscr Y}}
\def\Hes{{\mathscr H}}
\def\UU{{\mathscr U}}
\def\Zpq{{\mathscr Z}}

\def\Cd{{\mathscr C}}
\def\Vs{{\mathscr V}}
\def\hchow{{\mathaccent20\chow}}
\def\Prt{{\mathscr P}}
\def\Lrho{\text{\large{\mbox{$\rho\hskip-0pt$}}}}
\def\tLrho{\text{\large{\mbox{$\tilde\rho\hskip-0pt$}}}}

\def\Lpi{\text{\large{\mbox{$\pi\hskip-0pt$}}}}
\def\cLpi{\text{\large{\mbox{$\check\pi\hskip-0pt$}}}}

\def\Lwedge{\text{\tiny{\mbox{$\bigwedge\hskip-1pt$}}}}

\font\eu=eusm10 at 10pt

\def\eF{\text{\eu F}}
\def\eG{\text{\eu G}}
\def\eQ{\text{\eu W}}
\def\eS{\text{\eu U}}

\newcommand{\vcorr}[3][1]{%
  \begingroup
    \tabcolsep=.5\tabcolsep
    \sbox0{%
      \begin{tabular}[b]{@{}l}%
        #3%
         \tabularnewline
      \end{tabular}%
    }%
    \settoheight{\dimen0 }{%
      \rotatebox{#2}{%
        \copy0 %
        \kern-\tabcolsep
      }%
    }%
    \rule{0pt}{#1\dimen0}%
    \setlength{\wd0 }{1em}%
    \setlength{\ht0 }{1em}%
    \rotatebox{#2}{\usebox{0}}%
  \endgroup
}

\newenvironment{fcaption}{\begin{list}{}{
\setlength{\leftmargin}{35pt}
\setlength{\rightmargin}{35pt}
\setlength{\labelsep}{5pt}
}}{\end{list}}

\newenvironment{myitem}{\begin{list}{}{
\setlength{\leftmargin}{25pt}
\setlength{\labelsep}{5pt}
}}{\end{list}}
\newenvironment{myitem2}{\begin{list}{}{
\setlength{\leftmargin}{0.6cm}
\setlength{\itemindent}{-0.3cm}
\setlength{\itemsep}{0cm}
}}{\end{list}}

\newtheorem{thm}{Theorem}[subsection]
\newtheorem{prop}[thm]{Proposition}

\newtheorem{lem}[thm]{Lemma}
\newtheorem{cla}[thm]{Claim}
\theoremstyle{definition}
\newtheorem{defn}[thm]{Definition}

\theoremstyle{remark}
\newtheorem*{rem}{Remark}

\makeatother

\usepackage{babel}

\begin{document}
\global\long\def\sA{\mathcal{A}}
 \global\long\def\sB{\mathcal{B}}
 \global\long\def\sC{\mathcal{C}}
 \global\long\def\sD{\mathcal{D}}
 \global\long\def\sE{\mathcal{E}}
 \global\long\def\sF{\eF}
 \global\long\def\sG{\eG}
 \global\long\def\sH{\mathcal{H}}
 \global\long\def\sI{\mathcal{I}}
 \global\long\def\sJ{\mathcal{J}}
 \global\long\def\sK{\mathcal{K}}
 \global\long\def\sL{\mathcal{L}}
 \global\long\def\sN{\mathcal{N}}
 \global\long\def\sM{\mathcal{M}}
 \global\long\def\sO{\mathcal{O}}
 \global\long\def\sP{\mathcal{P}}
 \global\long\def\sS{\mathcal{S}}
 \global\long\def\sR{\mathcal{R}}
 \global\long\def\sQ{\mathcal{Q}}
 \global\long\def\sT{\mathcal{T}}
 \global\long\def\sU{\mathcal{U}}
 \global\long\def\sV{\mathcal{V}}
 \global\long\def\sW{\mathcal{W}}
 \global\long\def\sX{\mathcal{X}}
 \global\long\def\sY{\mathcal{Y}}
 \global\long\def\sZ{\mathcal{Z}}
 \global\long\def\tA{{\widetilde{A}}}
 \global\long\def\mA{\mathbb{A}}
 \global\long\def\mC{\mathbb{C}}
 \global\long\def\mF{\mathbb{F}}
 \global\long\def\mG{\mathbb{G}}
 \global\long\def\G{{\bf G}}
 \global\long\def\mN{\mathbb{N}}
 \global\long\def\mP{\mathbb{P}}
 \global\long\def\mQ{\mathbb{Q}}
 \global\long\def\mZ{\mathbb{Z}}
 \global\long\def\mW{\mathbb{W}}
 \global\long\def\Ima{\mathrm{Im}\,}
 \global\long\def\Ker{\mathrm{Ker}\,}
 \global\long\def\Alb{\mathrm{Alb}\,}
 \global\long\def\ap{\mathrm{ap}}
 \global\long\def\Bs{\mathrm{Bs}\,}
 \global\long\def\Chow{\mathrm{Chow}}
 \global\long\def\CP{\mathrm{CP}}
 \global\long\def\Div{\mathrm{Div}\,}
 \global\long\def\divi{\mathrm{div}\,}
 \global\long\def\expdim{\mathrm{expdim}\,}
 \global\long\def\ord{\mathrm{ord}\,}
 \global\long\def\Aut{\mathrm{Aut}\,}
 \global\long\def\Hilb{\mathrm{Hilb}}
 \global\long\def\Hom{\mathrm{Hom}}
 \global\long\def\id{\mathrm{id}}
 \global\long\def\Ext{\mathrm{Ext}}
 \global\long\def\sHom{\mathcal{H}{\!}om\,}
 \global\long\def\Lie{\mathrm{Lie}\,}
 \global\long\def\mult{\mathrm{mult}}
 \global\long\def\opp{\mathrm{opp}}
 \global\long\def\Pic{\mathrm{Pic}\,}
 \global\long\def\Pf{{\bf Pf}}
 \global\long\def\Sec{\mathrm{Sec}}
 \global\long\def\Spec{\mathrm{Spec}\,}
 \global\long\def\Sym{\mathrm{Sym}}
 \global\long\def\sQpec{\mathcal{S}{\!}pec\,}
 \global\long\def\Proj{\mathrm{Proj}\,}
 \global\long\def\Rhom{{\mathbb{R}\mathcal{H}{\!}om}\,}
 \global\long\def\aw{\mathrm{aw}}
 \global\long\def\exc{\mathrm{exc}\,}
 \global\long\def\emb{\mathrm{emb\text{-}dim}}
 \global\long\def\codim{\mathrm{codim}\,}
 \global\long\def\OG{\mathrm{OG}}
 \global\long\def\pr{\mathrm{pr}}
 \global\long\def\Sing{\mathrm{Sing}\,}
 \global\long\def\Supp{\mathrm{Supp}\,}
 \global\long\def\SL{\mathrm{SL}\,}
 \global\long\def\Reg{\mathrm{Reg}\,}
 \global\long\def\rank{\mathrm{rank}\,}
 \global\long\def\VSP{\mathrm{VSP}\,}
 \global\long\def\B{B}
 \global\long\def\Q{Q}
 \global\long\def\rG{\mathrm{G}}

\def\FigI{\resizebox{10cm}{!}{\includegraphics{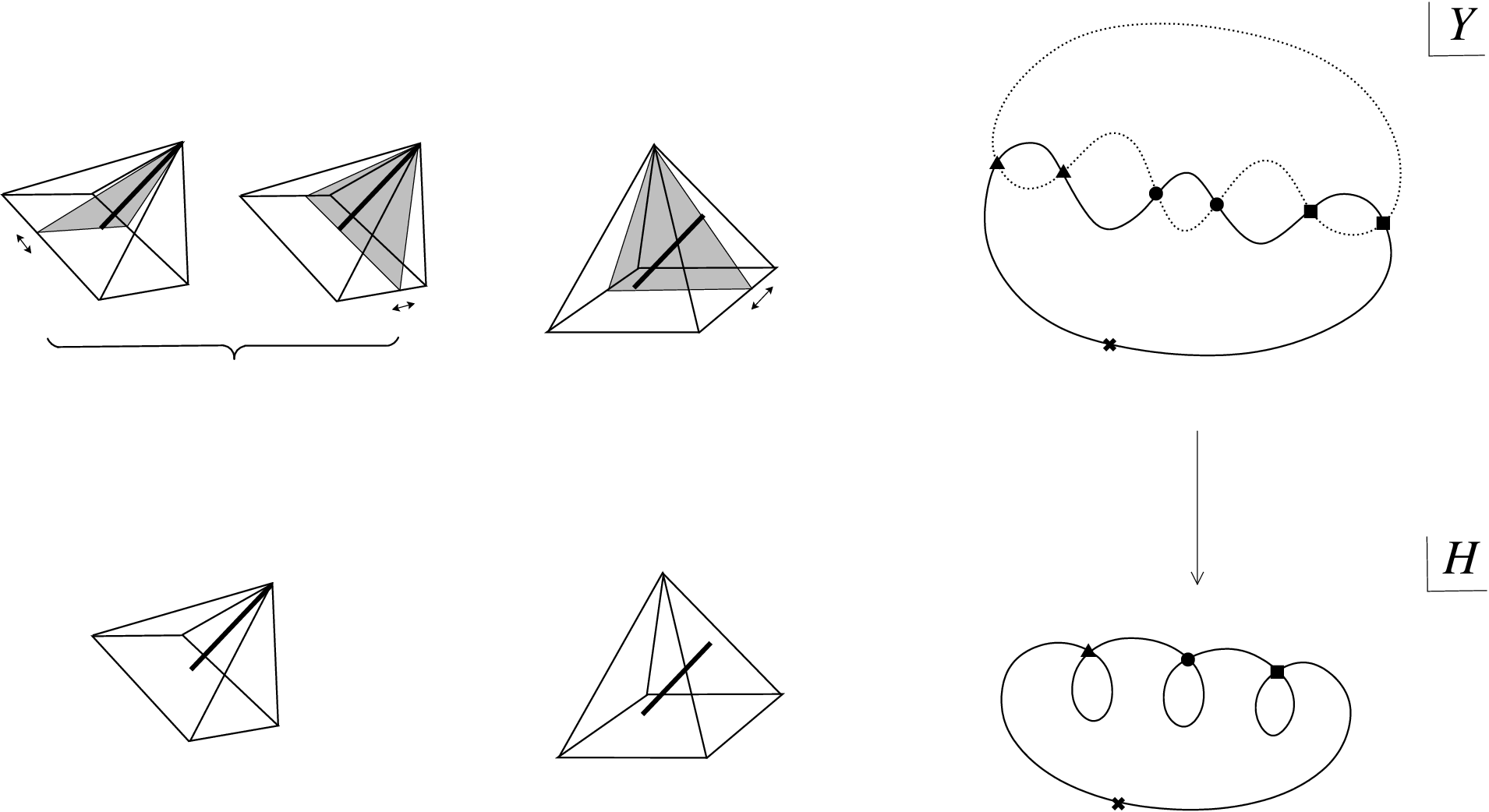}}}
\def\xyFigI{
\begin{xy}
(0,0)*{\FigI},
(-5,22)*{([Q],q)},
(-40,22)*{([Q],q_1)},
(-23,22)*{([Q],q_2)},
(-3,-20)*{l_x},
(-5,-8)*{[Q]},
(-33,-8)*{[Q]},
(-34,-18)*{l_x},
(14,22)*{C_x'},
(14,8)*{C_x},
(14,-17)*{\overline{\gamma}_x},
(35,-6)*{2:1},
\end{xy}
}

\def\xyResolY{
\begin{matrix}
\begin{xy}
(40,0)*+{\Zpq}="Z",  (55,0)*+{\mathrm{G}(3,V)}="G",
(40,-12)*+{\hcoY}="Y",
(40,-24)*+{\Hes}="H",
(14,0)*+{\hcoY_2}="Yii",
(1,-12)*+{\hcoY_3}="Yiii", (5,-12)*+{}="yiii", (22,-12)*+{}="ty",
(27,-12)*+{\widetilde{\hcoY}}="tY",
(14,-24)*+{\overline{\hcoY}}="bY", (19,-24)*+{}="by", (35,-24)*+{}="h",
(34,-10)*+{\,_{\Lrho_{\widetilde{\hcoY}}}},
\ar^{} "Z";"Y"
\ar^{\Lrho_\hcoY\;\;} "Y";"H"
\ar^{} "Z";"G"
\ar "Yii";"tY"
\ar "Yii";"Yiii"
\ar "Yiii";"bY"
\ar  "tY";"Y"
\ar "tY";"bY"
\ar @{-->}^{\,_\text{(anti-)flip}} "yiii";"ty"
\end{xy}
\end{matrix}
}

\def\FigYsRed{\resizebox{10cm}{!}{\includegraphics{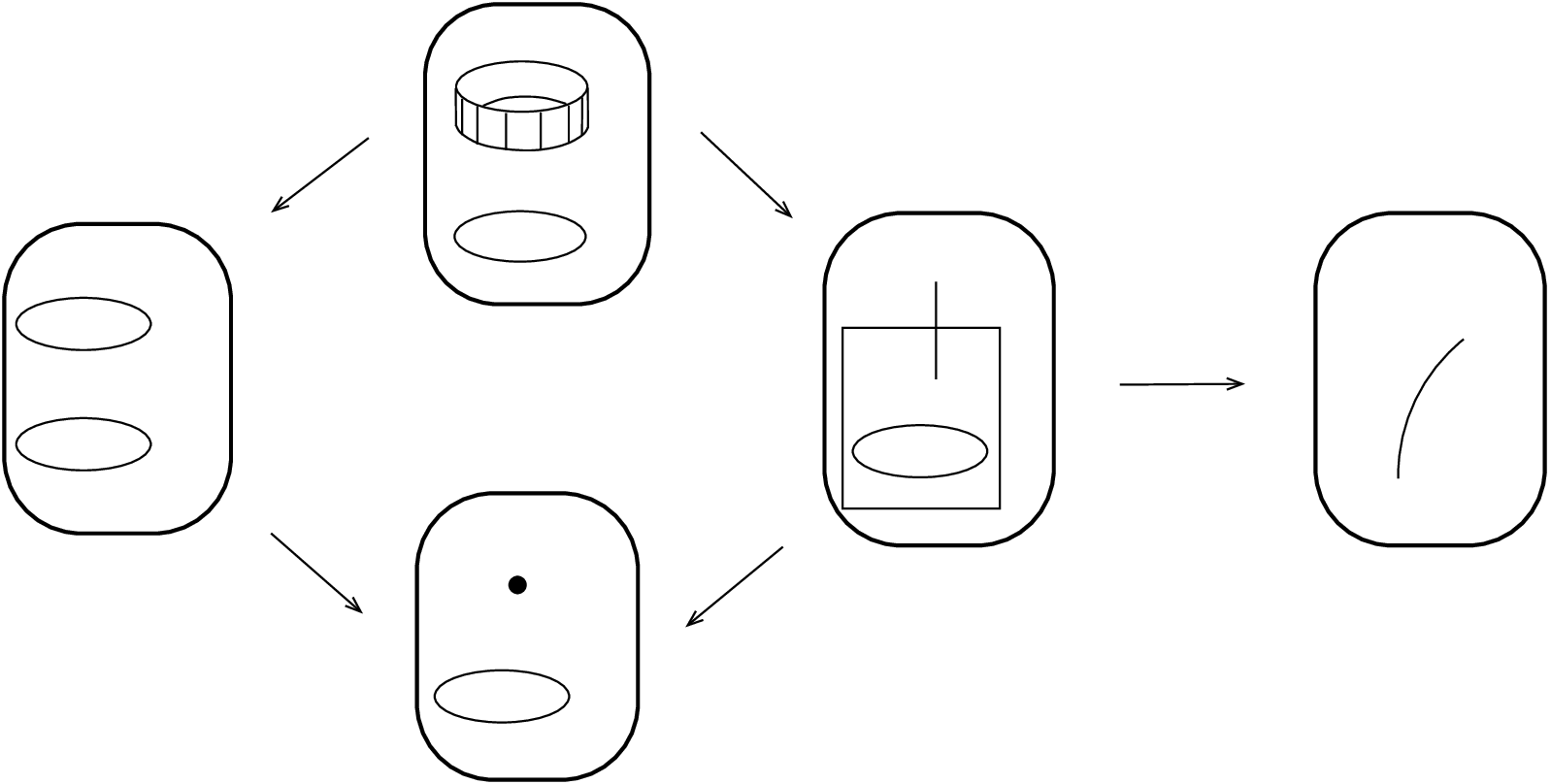}}}
\def\xyFigYs{
\begin{xy}
(0,0)*{\FigYsRed},
(-54,5.5)*{\hcoY_3},
(-38.5,3.5)*{\Prt_\rho},
(-38.5,-5)*{\Prt_\sigma},
(-28,12)*{\,_{\Lrho_{2}}},
{(-41.5,-12) \ar (-41.5,-20)^{\Lpi_{3}}_{\overset{\rG(3,6)\text{-}}{\text{bundle}}}},
(-41.5,-24)*{\mP(V)},
{(-31,0) \ar @{-->} (-2,0)^{\text{(anti-)flip}}},
(-25,24)*{\hcoY_2},
(-10, 16.5)*{F_\rho},
(-11,9)*{\Prt_\sigma},
(-2,11)*{\,_{\tLrho_{2}}},
(10,15)*{\widetilde\hcoY},
(7,1.7)*{F_{\widetilde\hcoY}},
(13,7.5)*{G_\rho},
(15,-5)*{\Prt_\sigma},
(26,-3)*{\,_{\Lrho_{\widetilde\hcoY}}},
(-16,-3.8)*{\overline{\hcoY}},
(-28,-8.7)*{\,_{\Lrho_{3}}},
(-16,-15)*{\rG(2,V)},
(-11.5,-20.5)*{\overline\Prt_\sigma},
(41.5,14)*{\hcoY},
(44.5,-3)*{G_\hcoY},
\end{xy}
}

\def\xyZYmain{
\begin{xy}
(0,40)*+{\Zpq_3}="Ziii",
(20,40)*+{\Zpq_2}="Zii",
(40,40)*+{\Zpq_2^o}="Ziio",
(-15,25)*+{\rG(2,T(-1))}="G",
(0,25)*+{\hcoY_3}="Yiii",
(20,25)*+{\hcoY_2}="Yii",
(40,25)*+{\hcoY_2^o}="Yiio",
(-30,10)*+{G(3,V)}="Gv",
( 0,10)*+{\mP(V)}="Pv",
(20,10)*+{\mP(V)}="Pvii",
(35,10)*+{\widetilde\hcoY}="tY",
(55,10)*+{\widetilde\hcoY^o}="tYo",
\ar^{\iota_{2'}}  @{_{(}->} "Ziio";"Zii",
\ar^{\Lrho_{2'}} "Zii";"Ziii",
\ar^{\Lpi_{G'}} "Ziii";"G",
\ar^{\Lrho_G} "G";"Gv",
\ar_{\Lpi_G} "G";"Pv",
\ar_{\iota_2} @{_{(}->} "Yiio";"Yii",
\ar^{\Lrho_2} "Yii";"Yiii",
\ar_{\Lpi_2} "Yii";"Pvii",
\ar^{\Lpi_3} "Yiii";"Pv",
\ar^{\tLrho_2} "Yii";"tY",
\ar_{\tilde\iota} @{_{(}->} "tYo";"tY",
\ar^{\Lpi_{3'}} "Ziii";"Yiii",
\ar_{\Lpi_{2'}} "Zii";"Yii",
\ar^{\Lpi_{2'}^o} "Ziio";"Yiio",
\ar^{\tLrho_2^o} "Yiio";"tYo",
\ar @{=} (5,10);(15,10)
\end{xy}
}

\def\xysarkisov{
\begin{xy} 
(0,0)*+{\Prt_\rho\;\subset}="A",  
(10,0)*+{\hcoY_3}="B",  
(40,0)*+{\widetilde\hcoY}="C",  
(50,0)*+{\supset \; G_\rho}="D",  
(10,-15)*+{\mP(V)}="E",  
(25,-15)*+{\overline{\hcoY}}="F",  
(40,-15)*+{\hcoY}="G",  
\ar @{-->}^{\text{(anti-)flip}} "B";"C"
\ar_{\mathrm{G}(3,6)\text{-bundle}} "B";"E"
\ar "B";"F",
\ar "C";"F",
\ar^{\text{div. cont.}} "C";"G"
\end{xy}
}

\def\xydiagII{
\begin{xy}
(45,0)*+{\;\;\;\;\;\mathrm{G}(2,T(-1))}="G2T", 
(25,-13)*+{\mathrm{G}(3,V)}="G3V",
(65,-13)*+{\mP(V).}="PV" 
\ar^{\Lrho_G} "G2T";"G3V"
\ar_{\Lpi_G} "G2T";"PV"
\end{xy}
}

\def\xyGiso{
\begin{xy}
(0,0)*+{\,_{([\overline{V}_2],[V_1]) \;\in \;}},
  (18,0)*+{\mathrm{G}(2,T(-1))}="A", 
  (18,-15)*+{\mathrm{G}(3,V)}="B", 
(3,-15)*+{\,_{[\Lpi_{V_1}^{\;-1}(\overline{V}_2)] \; \in \;}},
(67,0)*+{\,_{\ni \; ( [(\overline{V}_2)^\perp],[V_1]) }}, 
  (45,0)*+{\mathrm{G}(2,\Omega(1))}="C",
  (45,-15)*+{\mathrm{G}(2,V^*)}="D",
(62,-15)*+{\,_{\ni \; [\Lpi_{V_1}^{\;-1}(\overline{V}_2)^\perp],}} 
\ar^{\simeq} "A";"C"
\ar^{\simeq} "B";"D"
\ar^{\Lrho_G} "A";"B"
\ar^{} "C";"D"
\end{xy}
}

\def\xydiagIII{
\begin{xy}
(30,0)*+{\Zpq^u_3}="Z3p",
(0,-12)*+{\mP(T(-1)^{\wedge2})}="Pw2T",
(60,-12)*+{\qquad\qquad\qquad\hcoY_3 = \mathrm{G}(3,T(-1)^{\wedge2})}="Y3",
(30,-24)*+{\mP(V).}="PV", 
\ar_{\Lrho_{\Zpq^u_3}} "Z3p";"Pw2T"
\ar^{\Lpi_{\Zpq^u_3}} "Z3p";"Y3"
\ar_{\Lpi_{\;\mP}} "Pw2T";"PV"
\ar^{\Lpi_{\hcoY_3}} "Y3";"PV"
\end{xy}
}

\def\xydiagIV{
\begin{xy}
(30,0)*+{\Zpq_3}="Z3",
(0,-12)*+{\mathrm{G}(2,T(-1))}="G2T",
(60,-12)*+{\qquad\qquad\qquad\hcoY_3 = \mathrm{G}(3,T(-1)^{\wedge2})}="Y3",
(30,-24)*+{\mP(V).}="PV"
\ar_{\Lpi_{G'}} "Z3";"G2T"
\ar^{\Lpi_{3'}} "Z3";"Y3"
\ar_{\Lpi_{G}} "G2T";"PV"
\ar^{\Lpi_{3}} "Y3";"PV"
\end{xy}
}

\def\xydiagV{
\begin{xy}
(30,0)*+{\Zpq_\rho}="Zrho",
(0,-12)*+{\mathrm{G}(2,T(-1))}="G2T",
(60,-12)*+{\Prt_\rho}="Prho",
(30,-24)*+{\mP(V),}="PV", 
\ar_{\Lrho_{\Zpq_\rho}} "Zrho";"G2T"
\ar^{\Lpi_{\Zpq_\rho}} "Zrho";"Prho"
\ar_{\Lpi_{G}} "G2T";"PV"
\ar^{\Lpi_{\Prt_\rho}} "Prho";"PV"
\end{xy}
}

\def\xyZYZY{
\begin{xy}
(0,0)*++={\Zpq_2}="A",
(15,0)*++={\widetilde{\Zpq}}="B", 
(0,-12)*++={\hcoY_2}="C", 
(15,-12)*++={\widetilde{\hcoY}.}="D",
\ar "A";"B"
\ar "A";"C"
\ar "B";"D"
\ar "C";"D"
\end{xy} 
}

\def\xyquiverI{
\begin{xy}
(5,-12)*+{\circ}="c3",
(20,-12)*+{\circ}="c2",
(35,0)*+{\circ}="c1a",
(35,-24)*+{\circ}="c1b",
(0,-12)*{\mathcal{E}_3},
(20,-15)*{\mathcal{E}_2},
(38,3)*{\mathcal{E}_{1a}},
(38,-27)*{\mathcal{E}_{1b}},
(18,1)*{\Lwedge^2V},
(18,-25)*{{\ft S}^2V},
(30,-8)*{V},
(15,-9)*{V},
(30,-16)*{V},
\ar "c3";"c2"
\ar "c2";"c1a"
\ar "c2";"c1b"
\ar @/^1pc/@{->} "c3";"c1a"
\ar @/_1pc/@{->} "c3";"c1b"
\end{xy}
}

\def\xyquiverII{
\begin{xy}
(5,-12)*+{\circ}="c3",
(20,-12)*+{\circ}="c2",
(35,0)*+{\circ}="c1a",
(35,-24)*+{\circ}="c1b",
(0,-12)*{\mathcal{F}_3},
(20,-15)*{\mathcal{F}_2},
(38,3)*{\mathcal{F}_{1a}},
(38,-27)*{\mathcal{F}_{1b}},
(18,1)*{{\ft S}^2V^*},
(18,-25)*{\Lwedge^2V^*},
(30,-8)*{V^*},
(15,-10)*{V^*},
(30,-16)*{V^*},
\ar "c3";"c2"
\ar "c2";"c1a"
\ar "c2";"c1b"
\ar @/^1pc/@{->} "c3";"c1a"
\ar @/_1pc/@{->} "c3";"c1b"
\end{xy}
}

\def\xydiagIIIp{
\begin{xy}
(-17,8)*+{\Zpq_3},
(-8.5,8)*+{\subset},
(  0,8)*+{\Zpq_3^u}="Ti",
( 30,8)*+{\Zpq_2^u}="Tii",
( 39,8)*+{\supset},
( 48,8)*+{\Zpq_2},
(-17,4)*+{\cup},
(  0,4)*+{\cup},
( 30,4)*+{\cup},
( 48,4)*+{\cup},
(-17,0)*+{\Zpq_{\rho}}="Ai",
(-9,0)*+{=},
(0,0)*+{\Zpq^u_{\rho}}="A",
(30,0)*+{E_\rho^u=\mP(\Lrho_{2}^{*}\sS|_{F_{\rho}}) 
\;\supset}="B",
(0,-15)*+{\Prt_{\rho}}="C", 
(30,-15)*+{F_{\rho}}="D",
(48,0)*+{E_{\rho}}="E",
\ar (25,8);(5,8)
\ar "B";"A"
\ar^{\Lrho_{F_{\rho}}} "D";"C"
\ar^{\Lpi_{3^u\vert_\rho}} "A";"C"
\ar_{\Lpi_{2^u\vert_{\rho}}} "B";"D"
\ar^{\Lpi_{2'\vert_{\rho}}} "E";"D"
\ar_{\Lpi_\rho=\Lpi_{3'\vert_{\rho}}} "Ai";"C"
\end{xy}
}


\title{Double Quintic Symmetroids, Reye Congruences, and Their Derived Equivalence}

\author{Shinobu Hosono and Hiromichi Takagi}

\dedicatory{Dedicated to Professor Yujiro Kawamata on the occasion of his 60th
birthday}
\begin{abstract}
Let $\hcoY$ be the double cover of the quintic symmetric determinantal
hypersurface in $\mP^{14}$. We consider Calabi-Yau threefolds $Y$
defined as smooth linear sections of $\hcoY$. In our previous works,
we have shown that these Calabi-Yau threefolds $Y$ are naturally
paired with Reye congruence Calabi-Yau threefolds $X$ by the projective
duality of $\hcoY$, and $X$ and $Y$ have several interesting properties
from the view point of mirror symmetry and projective geometry. In
this paper, we prove the derived equivalence between a linear section
$Y$ of $\hcoY$ and the corresponding Reye congruence Calabi-Yau
threefold $X$. 
\end{abstract}
\maketitle
\markboth{}{Double symmetroids, Reye congruences,  and derived equivalence}

\section{Introduction}

Non-birational smooth Calabi-Yau threefolds which have an equivalent
derived category are of considerable interest from the view point
of the homological mirror symmetry due to Kontsevich\,\cite{Ko}.
As such an example, it has been proved in \cite{BC} that smooth Calabi-Yau
threefolds which are given respective smooth linear sections of $\mathrm{G}(2,7)$
and $\mathrm{Pfaff}(7)$ (see \cite{Ro}) are derived equivalent.
To our best knowledge, this pair is the first example of derived equivalent,
but non-birational smooth Calabi-Yau threefolds with Picard number
one. In \cite{Ku2}, the derived equivalence has been understood as
a corollary of a more general statement that a non-commutative resolution
of $\mathrm{Pfaff}(7)$ is homologically projective dual to $\mathrm{G}(2,7)$.
The homological projective duality is a framework which has been proposed
in \cite{Ku1} to generalize the classical projective duality in the
theory of derived category. In the case of $\mathrm{G}(2,7)$ and
$\mathrm{Pfaff}(7)$, the classical projective geometries involved
are that of the projective space of skew symmetric matrices $\mP(\Lwedge^{2}\mC^{7})$
and its dual projective space $\mP(\Lwedge^{2}(\mC^{*})^{7})$.

In the previous work \cite{HoTa1}, by studying mirror symmetry of
the Calabi-Yau threefold $X$ of a Reye congruence, we encountered
a similar phenomenon as above within the projective space of symmetric
matrices $\mP({\ft S}^{2}\mC^{5})$ and its dual projective space
$\mP({\ft S}^{2}(\mC^{*})^{5})$. We have observed that $X$ should
be paired with another smooth Calabi-Yau threefold $Y$, which is
given as a linear section of the double quintic symmetroid in the
projective duality. We are also led to the prediction {[}ibid.~Conj.2{]}
of their derived equivalence. The main result of this paper is to
show this prediction affirmatively.

\vspace{0.2cm}
 Let $V=\mC^{5}$ and define $X$ as a smooth linear section of the
second symmetric product $\chow=\ft{S}^{2}\mP(V)$ in $\mP(\ft{S}^{2}V)$.
Then $Y$ is defined as the orthogonal linear section of the double
symmetroid $\hcoY$, which is the double cover of the determinantal
symmetroid $\Hes$ in $\mP(\ft{S}^{2}V^{*})$. $\chow$ has a natural
resolution $\hchow$ given by the Hilbert scheme of two points. As
for $\hcoY$, a nice desingularization $\widetilde{\hcoY}$ has been
obtained in our recent paper \cite{HoTa3}. Furthermore, it has been
found that a finite collection of sheaves $(\mathcal{F}_{i})_{i\in I}$
on $\hchow$ introduces a dual Lefschetz collection in the derived
categories $\sD^{b}(\hchow)$ {[}ibid.~Thm.3.4.4{]}, and correspondingly
a collection of sheaves $(\sE_{i})_{i\in I}$ on $\widetilde{\hcoY}$
defines a Lefschetz collection in $\sD^{b}(\widetilde{\hcoY})$ {[}ibid.~Thm.8.1.1{]}.
In this paper, we focus on a certain closed subscheme $\Delta$ in
$\widetilde{\hcoY}\times\hchow$ and construct a locally free resolution
of its ideal sheaf $\sI$ in terms of these sheaves (Theorem \ref{thm:resolY}).
Considering a Fourier-Mukai functor with its kernel $I$ being the
restriction of the sheaf $\sI$ to $Y\times X$ in $\widetilde{\hcoY}\times\hchow$,
we prove the derived equivalence $\sD^{b}(X)\simeq\sD^{b}(Y)$ (\textbf{Theorem
\ref{thm:main}}).

\vspace{0.2cm}
 We introduce the subscheme $\Delta$ from the flag variety $\Delta_{0}:=\mathrm{F}(2,3,V)$
in $\mathrm{G}(3,V)\times\mathrm{G}(2,V)$. As we summarize in Subsection
\ref{HilbChow}, $\hchow$ is a $\mP^{2}$-bundle over $\rG(2,V)$
and hence there is a morphism to the Grassmannian $\mathrm{G}(2,V)$.
In contrast to this, the geometry of the double quintic symmetroid
$\hcoY$ is more involved. It turns out that $\hcoY$ describes the
connected families of planes contained in singular quadrics, which,
in case of quadrics of rank $4$, are represented by conics on $\mathrm{G}(3,V)$
(Subsection \ref{section:Quintic} and Section \ref{section:birationalY}).
In particular, we see that there is a generically conic bundle $\Zpq\to\hcoY$,
where $\Zpq$ parameterizes pairs of singular quadrics and planes
therein, and hence there exists a natural morphism $\Zpq\to\mathrm{G}(3,V)$.
Roughly speaking, though not quite precise, the subscheme $\Delta$
is constructed by pulling back $\Delta_{0}$ by the morphism $\Zpq\times\hchow\to\mathrm{G}(3,V)\times\mathrm{G}(2,V)$,
pushing forward by the morphism $\Zpq\times\hchow\to\hcoY\times\hchow$
and taking the transform by the birational morphism $\widetilde{\hcoY}\times\hchow\to\hcoY\times\hchow$.
The observation that the families of planes in rank 4 quadrics are
represented by conics on $\mathrm{G}(3,V)$ implies that $\hcoY$
is birational to the Hilbert scheme of conics on $\mathrm{G}(3,V)$
studied by \cite{IM}. 

The choice of the subscheme $\Delta$ comes from our observation on
the so-called BPS numbers of $Y$ listed in \cite[Table 3]{HoTa1}.
Let $I_{x}$ be the restriction of the ideal sheaf $I$ to $Y\times\{x\}$.
We will show that $I_{x}$ defines a curve $C_{x}$ on $Y$ of arithmetic
genus 3 and degree 5 parameterized by $X$, and $C_{x}$ is smooth
if $X$ and $x$ are general (Propositions \ref{prop:genus3degree5}
and \ref{prop:Cflat}). Our observation/discovery about this family
of curves is that this can be identified in the table of BPS numbers
$n_{g}^{Y}(d)$ at genus~3 and degree~5 as \[
n_{3}^{Y}(5)=100=(-1)^{\dim X}e(X)\times2,\]
 where $e(X)=-50$ is the Euler number of $X$. We note that the signed
Euler number $(-1)^{\dim X}e(X)$ is in accord with the counting rule
of the BPS numbers for a family of curves\,\cite{GV}. The factor
$2$ indicates that there should be another family of curves parameterized
by $X$. Interestingly, we find that there is a {}``shadow'' curve,
of the same genus and degree for each $C_{x}$ (Fig.1 in Subsection
\ref{subsection:curves}), which explains this factor. The BPS numbers
are integers which we calculate by using mirror symmetry, and its
mathematical ground is still missing in general. We believe, however,
that our observation above may be justified by the Donaldson-Thomas\,(DT)
invariants or the Pandharipande-Thomas\,(PT) invariants associated
with a suitable moduli problem of ideal sheaves or stable pairs (see
\cite{PT} and reference therein). Namely we expect that the Calabi-Yau
threefold $X$ arises as a suitable moduli space of the ideal sheaf
of curves on $Y$ and the derived equivalence between the two is a
consequence of this.

\vspace{0.1cm}
 Here we should remark some similarity of our construction to that
of the Grassmann-Pfaffian case. The proof of derived equivalence due
to \cite{BC} (and also \cite{Ku2}) is based on a certain incidence
relation on $\mathrm{G}(2,7)\times\mathrm{Pfaff}(7)$ which gives
rise the kernel of a Fourier-Mukai functor. Our proof starting with
$\Delta$ is basically parallel to this, although the formulation
of our incidence relation $\Delta$ and the corresponding ideal sheaf
are much more involved and requires the desingularization $\widetilde{\hcoY}\to\hcoY$
\cite{HoTa3}. In the Grassmann-Pfaffian case, the restriction $I_{y}$
to $\{y\}\times X$ of the corresponding ideal sheaf $I$ defines
a generically smooth family of curves of genus 6 and degree 14 on
$X$. In this case, however, we read the corresponding BPS number
as $n_{6}^{X}(14)=123676$ (see \cite[(4.2)]{HoTa1}) and there seems
to be some complications in the possible moduli interpretation in
terms of DT/PT invariants.\vspace{0.1cm}

\textcolor{black}{Recently, the derived equvalence of the Grassmann-Pfaffian
Calabi-Yau threefolds has been formulated in the framework of categorical
geometric invariant theory \cite{BFK,Ha,Segal}. Our case has been
argued in \cite{Hor,Hori-Knapp} in the language of physics called
gauged linear sigma model, which is closely related to geometric invariant
theory. It seems interesting to see how the derived equivalence in
our case will be formulated in categorical geometric invariant theory. }

\vspace{0.3cm}
 Here we outline the present paper. In Section 2, we summarize some
basic results on which our arguments are based. In Section \ref{section:family},
we construct a family of curves on $Y$ parameterized by $X$ (Proposition
\ref{prop:genus3degree5}) and show that it comes with another family
of {}``shadow'' curves. Also we calculate the Brauer group of $Y$
as a corollary. In Section 4, based on the results \cite{HoTa3},
we summarize the birational geometry of $\hcoY$ and introduce the
desingularization $\widetilde{\hcoY}$. We will also introduce generically
conic bundles over the birationally related varieties to $\hcoY$.
In Section 5, we introduce the subscheme $\Delta$ representing the
incidence relation on $\widetilde{\hcoY}\times\hchow$ and its ideal
sheaf $\sI$. Dividing the process into four major steps, we obtain
a locally free resolution of $\sI$ in terms of the collections of
locally free sheaves $(\sE_{i})_{i\in I}$ and $(\mathcal{F})_{i\in I}$
(Theorem \ref{thm:resolY}). In section 6, we show that the subscheme
$\Delta$ is contained in (the pullback of) the universal family of
hyperplane sections $\sV$. In Section 7, restricting the ideal sheaf
to $Y\times X$, we show that this defines a family of curve on $Y$
which is flat over $X$, and coincides with the family obtained in
Section 3 (Proposition \ref{prop:Cflat}). In Section 8, using the
properties of the (dual) Lefschetz collections $(\sE_{i})_{i\in I}$
and $(\mathcal{F}_{i})_{i\in I}$ \cite[Thm.8.1.1]{HoTa3}, we show
the derived equivalence $\sD^{b}(X)\simeq\sD^{b}(Y)$. Some further
discussions are presented in Section 9.

\vspace{1cm}
 \textbf{Acknowledgements:} This paper is supported in part by Grant-in
Aid Scientific Research (C 18540014, S.H.) and Grant-in Aid for Young
Scientists (B 20740005, H.T.).

\noindent 

\newpage{} \global\long\def\Homega{H_{\mP(\Omega(1))}}
 \global\long\def\Hwomega{H_{\mP(\Omega(1)^{\wedge2})}}
 \global\long\def\Lwomega{L_{\mP(\Omega(1)^{\wedge2})}}
 \global\long\def\HGo{H_{\mathrm{G}(\Omega(1)^{\wedge2})}}
 \global\long\def\HGt{H_{\mathrm{G}(\wedge^{2} T(-1))}}
\global\long\def\Lt{L_{\mP(T(-1))}}
 \global\long\def\Lwt{L_{\mP(T(-1)^{\wedge2})}}
 \global\long\def\Ht{H_{\mP(T(-1))}}
 \global\long\def\Hwt{H_{\mP(T(-1)^{\wedge2})}}

\def\Notation{******\textbf{Glossary of notation.}

\begin{list}{}{\setlength{\leftmargin}{10pt} \setlength{\labelsep}{5pt}} 

\item We often abuse Cartier divisors and invertible sheaves. 

\item $\mP(W):$ the projectivization of a vector space $W$. 

\item $\sO_{X}(1):=\sO_{\mP(W)}(1)|_{X}$ if a variety $X$ is naturally
embedded in $\mP(W)$. 

\item $\mP(\sE):$ the projectivization of a locally free sheaf $\sE$
on a variety $X$. 

\item $H_{\mP(\sE)}:$ the tautological divisor of $\mP(\sE)$. 

\item $\sO_{\mP(\sE)}(1):$ the tautological invertible sheaf of
$\mP(\sE)$. 

\item $\mathrm{G}(r,\sE):$ the Grassmann bundle which parameterizes
$r-1$-dimensional linear subspaces in fibers of $\mP(\sE)$. 

\item $V$: a (fixed) $5$ dimensional complex vector space. $\mP^{4}:=\mP(V)$. 

\item $V_{i}$: an $i$-dimensional vector subspace of $V$. 

\item $\mathrm{F}(a,b,V):=\{(V_{a},V_{b})\mid V_{a}\subset V_{b}\subset V\}$
(the flag variety). \end{list}******}

$\;$

\def\NotationII{******

\begin{list}{}{\setlength{\leftmargin}{10pt} \setlength{\labelsep}{5pt}} 

\item We also use the following notation which simplifies lengthy
formulas: 

\item $\Omega(1):=\Omega_{\mP(V)}(1)$. 

\item $\Omega(1)^{\wedge i}:=\Lwedge^{i}(\Omega_{\mP(V)}(1))$ for
$i\geq2$. 

\item $T(-1):=T_{\mP(V)}(-1)$. 

\item $T(-1)^{\wedge i}:=\Lwedge^{i}(T(-1))$ for $i\geq2$. 

\item $\sO(i):=\sO_{\mP(V)}(i)$ for $i\in\mZ$. \end{list} *******}

\noindent \textbf{Notation:} Throughout the paper, we work over $\mC$,
the complex number field. We will use the following notation which
simplifies lengthy formulas: 

\begin{list}{}{\setlength{\leftmargin}{10pt} \setlength{\labelsep}{5pt}} 

\item $V$: a (fixed) $5$ dimensional complex vector space. $\mP^{4}:=\mP(V)$. 

\item $V_{i}$: an $i$-dimensional vector subspace of $V$. 

\item $\Omega(1):=\Omega_{\mP(V)}(1)$. 

\item $\Omega(1)^{\wedge i}:=\Lwedge^{i}(\Omega_{\mP(V)}(1))$ for
$i\geq2$. 

\item $T(-1):=T_{\mP(V)}(-1)$. 

\item $T(-1)^{\wedge i}:=\Lwedge^{i}(T(-1))$ for $i\geq2$. 

\item $\sO(i):=\sO_{\mP(V)}(i)$ for $i\in\mZ$. \end{list} 

\vspace{1cm}

\section{Preliminaries}

Here we summarize the basic results on which our proof of the derived
equivalence $\sD^{b}(X)\simeq\sD^{b}(Y)$ rely. We also summarize
the construction of the Calabi-Yau threefolds $X$ and $Y$ which
has been described in \cite{HoTa1,HoTa3}. 

\vspace{0.5cm}

\subsection{Basic general results}

For the computations of cohomology groups which appear in this paper,
we use the Bott theorem about the cohomology groups of Grassmann bundles
below extensively.

\vspace{0.3cm}

For a locally free sheaf $\mathcal{E}$ of rank $r$ on a variety
and a nonincreasing sequence $\beta=(\beta_{1},\beta_{2},\dots,\beta_{r})$
of integers, we denote by $\ft{\Sigma}^{\beta}\mathcal{E}$ the associated
locally free sheaf with the Schur functor $\ft{\Sigma}^{\beta}$.

\begin{thm}{\rm {\bf (Bott Theorem)}}\label{thm:Bott} Let $\pi\colon\mathrm{G}(r,\sA)\to X$
be a Grassmann bundle for a locally free sheaf $\sA$ on a variety
$X$ of rank $n$ and $0\to\sS\to\sA\to\sQ\to0$ the universal exact
sequence. For $\beta:=(\alpha_{1},\dots,\alpha_{r})\in\mZ^{r}$ $(\alpha_{1}\geq\dots\geq\alpha_{r})$
and $\gamma:=(\alpha_{r+1},\dots,\alpha_{n})\in\mZ^{n-r}$ $(\alpha_{r+1}\geq\dots\geq\alpha_{n})$,
we set $\alpha:=(\beta,\gamma)$ and $\sV(\alpha):=\ft{\Sigma}^{\beta}\sS^{*}\otimes\ft{\Sigma}^{\gamma}\sQ^{*}$.
Finally, let $\rho:=(n,\dots,1)$, and, for an element $\sigma$ of
the $n$-th symmetric group $\mathfrak{S}_{n}$, we set $\sigma^{\bullet}(\alpha):=\sigma(\alpha+\rho)-\rho$. 

\begin{myitem2}

\item[\rm{(1)}] If $\sigma(\alpha+\rho)$ contains two equal integers,
then $R^{i}\pi_{*}\sV(\alpha)=0$ for any $i\geq0$. 

\item[\rm{(2)}]If there exists an element $\sigma\in\mathfrak{S}_{n}$
such that $\sigma(\alpha+\rho)$ is strictly decreasing, then $R^{i}\pi_{*}\sV(\alpha)=0$
for any $i\geq0$ except $R^{l(\sigma)}\pi_{*}\sV(\alpha)=\ft{\Sigma}^{\sigma^{\bullet}(\alpha)}\sA^{*}$,
where $l(\sigma)$ represents the length of $\sigma\in\mathfrak{S}_{n}$.

\end{myitem2}

\end{thm}

\begin{proof} See \cite{Bo}, \cite{D}, or \cite[(4.19) Corollary]{W}.
\end{proof}

\begin{thm}{\rm {\bf (Grothendieck-Verdier duality)}}\label{cla:duality}
Let $f\colon X\to Y$ be a proper morphism of smooth varieties $X$~and~$Y$.
Set $n:=\dim X-\dim Y$. We have the following functorial isomorphism:

For $\mathcal{F}^{\bullet}\in\sD^{b}(X)$ and $\mathcal{E}^{\bullet}\in\sD^{b}(Y)$,
\[
Rf_{*}R\sH om(\mathcal{F}^{\bullet},Lf^{*}\sE^{\bullet}\otimes\omega_{X/Y}[n])\simeq R\sH om(Rf_{*}\mathcal{F}^{\bullet},\mathcal{E}).\]
 In particular, if $\mathcal{E}^{\bullet}$ and $\mathcal{F}^{\bullet}$
are locally free $($and then we write them simply $\mathcal{E}$
and~$\mathcal{F})$ and if $R^{\bullet}f_{*}\mathcal{F}=0$, then
\[
R^{\bullet+n}f_{*}(\mathcal{F}^{*}\otimes\sE\otimes\omega_{X/Y})\simeq\sE xt^{\bullet}(f_{*}\mathcal{F},\mathcal{E}).\]

\end{thm}

\begin{proof} See \cite[Theorem 3.34]{Huy}. \end{proof}

For our proof of the derived equivalence, we use the following theorem:

\begin{thm} \label{thm:D} Let $X$ and $Y$ be smooth projective
varieties and $\Prt$ a coherent sheaf on $X\times Y$ flat over $X$.
Then the Fourier-Mukai transform $\Phi_{\sP}\colon\sD^{b}(X)\to\sD^{b}(Y)$
$(\sP$ is called the \textit{kernel} for $\Phi_{\sP})$ is fully
faithful if and only if the following two conditions are satisfied:
\global\long\def\labelenumi{\textup{(\roman{enumi})}}
 
\begin{enumerate}
\item For any point $x\in X$, it holds $\Hom(\sP_{x},\sP_{x})\simeq\mC$,
and 
\item if $x_{1}\not=x_{2}$, then $\Ext^{i}(\sP_{x_{1}},\sP_{x_{2}})=0$
for any $i$. 
\end{enumerate}
Moreover, under these conditions, $\Phi_{\sP}$ is an equivalence
of triangulated categories if and only if $\dim X=\dim Y$ and $\sP\otimes\pr_{1}^{*}\omega_{X}\simeq\sP\otimes\pr_{2}^{*}\omega_{Y}$.

In particular, if $\dim X=\dim Y$, $\omega_{X}\simeq\sO_{X}$ and
$\omega_{Y}\simeq\sO_{Y}$, then $\Phi_{\sP}$ is fully faithful if
and only if it is an equivalence. \end{thm}

\begin{proof} See \cite[Theorem 1.1]{BO}, \cite[Theorem 1.1]{B},
\cite[Corollary 7.5 and Proposition 7.6]{Huy}. \end{proof}

\vspace{0.3cm}
 In this paper, we adopt the following definition of Calabi-Yau variety
and also Calabi-Yau manifold.

\begin{defn} We say a normal projective variety $X$ \textit{a Calabi-Yau
variety} if $X$ has only Gorenstein canonical singularities, the
canonical bundle of $X$ is trivial, and $h^{i}(\sO_{X})=0$ for $0<i<\dim X$.
If $X$ is smooth, then $X$ is called \textit{a Calabi-Yau manifold}.
A smooth Calabi-Yau threefold is abbreviated as a Calabi-Yau threefold.~{\hfill{}{[}{]}}
\end{defn}

\vspace{0.5cm}

\subsection{The Hilbert scheme $\hchow$ of two points on $\mP(V)$}

\label{HilbChow}

Let $\chow$ be the Chow variety of two points on $\mP(V)$ embedded
by the Chow form into $\mP({\ft S}^{2}V)$. Denote by $\chow_{0}$
the second Veronese variety $v_{2}(\mP(V))$. It is a well-known fact
that $\chow_{0}=\Sing\chow$ and $\chow$ is the secant variety of
$\chow_{0}$. If we take a coordinate of ${\ft S}^{2}V$ so that it
represents a generic $5\times5$ symmetric matrix, then $\chow_{0}$
(resp.~$\chow$) is characterized as the locus of rank $1$ (resp.~rank
$\leq2$) symmetric matrices.

Let $\hchow$ be the Hilbert scheme of 0-dimensional subschemes of
length two in $\mP(V)$, which will be called the Hilbert scheme of
two points in $\mP(V)$ hereafter. A $0$-dimensional subscheme of
length two may be determined from the corresponding $0$-cycle $\eta$
of length two on $\mP(V)$ and a line $l\subset\mP(V)$ containing
$\eta$, and vice versa. Hence, we have an isomorphism ${\hchow}\simeq\mP(\ft{S}^{2}\sF)$,
where $\sF$ is the universal subbundle of rank two on $\mathrm{G}(2,V)$.
Let \begin{equation}
0\rightarrow\sF\rightarrow V\otimes\sO_{\mathrm{G}(2,V)}\rightarrow\sG\rightarrow0\label{eq:univG25}\end{equation}
 be the universal exact sequence on $\mathrm{G}(2,V)$. By the induced
injection ${\ft S}^{2}\sF\hookrightarrow{\ft S}^{2}V\otimes\sO_{\mathrm{G}(2,V)}$,
we obtain a morphism $\hchow\to\mP({\ft S}^{2}V)\times\mathrm{G}(2,V)\to\mP({\ft S}^{2}V)$.
Then the tautological divisor of $\mP({\ft S}^{2}\sF)$ is the pull-back
of $\sO_{\mP({\ft S}^{2}V)}(1)$. The image of this morphism is nothing
but $\chow$ and the induced morphism $f\colon\hchow\to\chow$ coincides
with the Hilbert-Chow morphism. Moreover, $f$ is the blow-up along
$\chow_{0}$. \[
\xymatrix{ & {\hchow}\ar[dl]_{f}\ar[dr]^{g}\\
\chow &  & \mathrm{G}(2,V).}
\]
 We define the following divisors on $\hchow$: \[
\text{\ensuremath{H_{{\hchow}}=f^{*}\sO_{\chow}(1)}and \ensuremath{L_{{\hchow}}=g^{*}\sO_{\mathrm{G}(2,V)}(1)}.}\]

We also denote by $E_{f}$ the $f$-exceptional divisor.

\vspace{0.5cm}

\subsection{The double quintic symmetroid $\hcoY$}

\label{section:Quintic}

Hereafter we assume $V\simeq\mC^{5}$ and write by $V^{*}$ the dual
vector space of $V$. Let $\Hes\subset\mP({\ft S}^{2}V^{*})$ be the
locus of singular quadrics in $\mP(V)$, which will be called the
(generic) \textit{quintic symmetroid} since it is the hypersurface
defined by the determinant of the (generic) $5\times5$ symmetric
matrix. It is the locus of $5\times5$ symmetric matrices of rank
$\leq4$.

Let $\UU:=\{(t,[Q])\mid t\in\Sing Q\}\subset\mP(V)\times\mP({\ft S}^{2}V^{*})$.
Then the natural morphism $p\colon\UU\to\Hes$ is a desingularization
of $\Hes$ (see \cite[Subsect.4.1]{HoTa3}).

To construct the double cover $\hcoY$ of $\Hes$ branched along the
locus of symmetric matrices of rank $\leq3$, we introduce the variety
$\Zpq$ which parameterizes the pair of quadrics $Q$ in $\mP(V)$
and planes $\mP(\Pi)$ such that $\mP(\Pi)\subset Q$. To describe
$\Zpq$ more explicitly, consider the universal exact sequence on
$\mathrm{G}(3,V)$; \begin{equation}
0\rightarrow\eS\rightarrow V\otimes\sO_{\mathrm{G}(3,V)}\rightarrow\eQ\rightarrow0.\label{eq:univG35}\end{equation}
 The surjection $\ft{S}^{2}V^{*}\otimes\sO_{\mathrm{G}(3,V)}\to\ft{S}^{2}\eS^{*}$
follows from the dual sequence. Then we define a locally free sheaf
$\sE^{*}$ on $\mathrm{G}(3,V)$ by \begin{equation}
0\to\sE^{*}\to\ft{S}^{2}V^{*}\otimes\sO_{\mathrm{G}(3,V)}\to\ft{S}^{2}\eS^{*}\to0.\label{eq:sE0}\end{equation}
 Then, since the fiber of $\mP(\ft{S}^{2}\eS^{*})$ over a point $[\Pi]\in\mathrm{G}(3,V)$
may be identified with the linear system of quadrics on $\mP(\Pi)$,
the fiber of $\mP(\sE^{*})$ represents the quadrics which contain
the plane $\mP(\Pi)$. Namely we have \[
\Zpq=\mP(\sE^{*})\subset\mathrm{G}(3,V)\times\mP({\ft S}^{2}V^{*}).\]
 Note that the image of the naturally induced morphism $\Zpq\to\mP({\ft S}^{2}V^{*})$
coincides with the singular quadrics $\Hes$, since a smooth quadric
does not contain a plane.

Let $\Zpq\overset{\Lpi_{\Zpq}}{\to}\hcoY\overset{\Lrho_{\hcoY}}{\to}\Hes$
be the Stein factorization of the natural morphism $\Zpq\to\Hes$.
Then $\Lrho_{\hcoY}\colon\hcoY\to\Hes$ is the finite double covering
branched along the locus of quadrics of rank less than or equal to
three. $\hcoY$ is called the (generic) \textit{double quintic symmetroid}.
We say that $y\in\hcoY$ is a \textit{rank $i$ point} if $\Lrho_{\hcoY}(y)\in\Hes$
corresponds to a quadric of rank $i$. $G_{\hcoY}:=\Sing\hcoY$ is
the subset consisting of rank $1$ and $2$ points (\cite[Prop.~5.7.2]{HoTa3}).
We introduce divisors on $\hcoY$ and $\Zpq$, respectively, by \[
M_{\hcoY}=\Lrho_{\hcoY}^{\;*}\sO_{\Hes}(1)\text{ and }M_{\Zpq}=\Lpi_{\Zpq}^{\;*}\circ\Lrho_{\hcoY}^{\;*}\sO_{\Hes}(1),\]
 where $\sO_{\Hes}(1):=\sO_{\mP({\ft S}^{2}V^{*})}(1)\vert_{\Hes}$.

Consider the fiber $\Zpq_{[Q]}$ of the morphism $\Zpq\to\Hes$ over
a quadric $[Q]\in\Hes$. If ${\rm rank}\, Q=4$, then $Q$ is a cone
over the smooth quadric ($\simeq\mP^{1}\times\mP^{1}$) in $\mP(V/V_{1})$
with the vertex $[V_{1}]$, and the planes in $Q$ consist of two
different $\mP^{1}$-families which correspond to the two rulings
of $\mP^{1}\times\mP^{1}$. If ${\rm rank}\, Q=3$, then $Q$ is a
cone over the smooth quadric in $\mP(V/V_{2})$ with the vertex $\mP(V_{2})\simeq\mP^{1}$,
and in this case there is only one $\mP^{1}$-family of planes in
$Q$. Thus the morphism $\Lpi_{\Zpq}\colon\Zpq\to\hcoY$ of the Stein
factorization is a generically conic bundle; generic points of $\hcoY$
are represented by pairs $(Q,q)$ of quadrics $Q$ of rank $4$ (or
$3$) and connected families $q$ of planes contained in $Q$, where
$q$ represent conics in $\mathrm{G}(3,V)$ (\cite[Prop.~4.2.5]{HoTa3}).
It turns out that several birational models of $\Zpq\to\hcoY$ play
crucial roles to construct the kernel of a Fourier-Mukai functor giving
an equivalence of $\sD^{b}(X)$ and $\sD^{b}(Y)$.

\vspace{0.3cm}
 The following computations of the Chern classes of the locally free
sheaf $\sE$ on $\Zpq$ will be used in the next section.

\begin{lem} \label{cla:c1c2} $c_{1}(\sE)=c_{1}(\sO_{\mathrm{G}(3,V)}(4))$
and $c_{2}(\sE)=5c_{2}(\eQ)+6c_{1}(\sO_{\mathrm{G}(3,V)}(1))^{2}$.
\end{lem}

\begin{proof} This follows from standard computations of Chern classes
by using the exact sequences (\ref{eq:univG35}) and (\ref{eq:sE0}).
Note that $c_{1}(\sO_{\mathrm{G}(3,V)}(1))$ is given by the Schubert
cycle $\sigma_{1}$, which is $c_{1}(\eS^{*})$ in our notation. Since
${\rm rk}\,\eS=3$, the first relation follows from $c_{1}(\sE)=c_{1}(\ft{S}^{2}\eS^{*})=4c_{1}(\eS^{*})$.
For the second relation, we derive $c_{2}(\ft{S}^{2}\eS^{*})=5c_{1}(\eS^{*})^{2}+5c_{2}(\eS^{*})$
and use $c_{2}(\eQ)=c_{1}(\eS^{*})^{2}-c_{2}(\eS^{*})$, $c_{2}(\sE)=c_{1}(\ft{S}^{2}\eS^{*})^{2}-c_{2}(\ft{S}^{2}\eS^{*})$.
\end{proof}

\vspace{0.5cm}

\subsection{Calabi-Yau threefolds $X$ and $Y$}

\label{section:XY}

Let us identify $\mP({\ft S}^{2}V^{*})$ with the space of quadrics
in $\mP(V)$. Then a $4$-plane $P\subset\mP({\ft S}^{2}V^{*})$ may
be regarded as a $4$-dimensional linear system of quadrics $P=|Q_{1},Q_{2},\cdots,Q_{5}|$
with the quadrics $Q_{i}$ defined by the corresponding $5\times5$
symmetric matrices $A_{i}$. Let $P^{\perp}\subset\mP(\ft{S}^{2}V)$
be the orthogonal space of $P$ with respect to the dual pairing between
${\ft S}^{2}V$ and ${\ft S}^{2}V^{*}$, and define $X=\chow\cap P^{\perp}$.
Dually, we may construct also $Y$ in $\hcoY$ as the pull-back of
the quintic symmetroid $\Hes\cap P$. The linear system $P$ is called
\textit{regular} if i) it is base point free and ii) any line $l\subset{\rm Sing}\, Q$
for some $Q\in P$ is not contained in a linear subsystem of dimension
$\geq2$. $X$ is smooth if and only if $P$ is regular \cite[Prop.2.1]{HoTa1}.
It has been shown that $Y$ is also smooth when $P$ is regular {[}ibid.~Prop.3.11{]}.
We say that $X$ and $Y$ defined for such a choice of $P$ are \textit{orthogonal}
to each other.

\begin{prop} \label{prop:topXY} $X$ and $Y$ constructed as above
are Calabi-Yau threefolds with the following invariants$:$ 

\noindent $1)$ \hfil $\deg(X)=35,\;\; c_{2}.D=50,\;\; h^{2,1}(X)=26,\; h^{1,1}(X)=1,$
\hfil 

\noindent where $D$ is the restriction to $X$ of a hyperplane section
of $\chow$, $\deg(X):=D^{3}$ and $c_{2}$ is the second Chern class
of $X$. 

\noindent $2)$ \hfil $\deg(Y)=10,\;\; c_{2}.M=40,\;\; h^{2,1}(Y)=26,\; h^{1,1}(Y)=1,$
\hfil 

\noindent where $M$ is the restriction of $M_{\hcoY}$ to $Y$, $\deg(Y):=M^{3}$,
and $c_{2}$ is the second Chern class of $Y$. \end{prop}

\begin{proof} The invariants of $X$ are easy to be determined, see
\cite[Prop.2.1]{HoTa1}. The invariants of $Y$ are determined in
\cite[Prop.4.3.4]{HoTa3} (see also \cite[Prop.3.11 and 3.12]{HoTa1}).
\end{proof}

Hereafter we consider the Calabi-Yau threefolds $X$ and $Y$ which
are orthogonal to each other.

Since $X$ is smooth, $X$ is disjoint from $\Sing\chow$, and hence
we can consider $X$ to be contained in $\hchow$. Moreover, $X$
is mapped by $g\colon\hchow\to\mathrm{G}(2,V)$ onto its image isomorphically,
and hence we can also consider $X$ to be contained in $\mathrm{G}(2,V)$.
By the existence of this embedding into $\mathrm{G}(2,V)$, $X$ is
called \textit{a $($generalized\,$)$ Reye congruence}. As a subvariety
of $\mathrm{G}(2,V)$, $X$ is characterized as the subset of lines
$l$ in $\mP(V)$ such that quadrics which contain $l$ form a $2$-dimensional
linear system (net) in $P$ (\cite[Prop.3.5.2]{HoTa3}).

$Y$ will be called the \textit{$(3$-dimensional\,$)$ double quintic
symmetroid orthogonal to $X$.} Since $Y$ is smooth, $Y$ is disjoint
from $\Sing\hcoY=G_{\hcoY}$.

\vspace{1cm}

\section{A family of curves on $Y$ parameterized by $X$}

\label{section:family}

In this section, using the generically conic bundle $\Lpi_{\Zpq}\colon\Zpq\to\hcoY$,
we construct a family of curves on $Y$ of degree $5$ parameterized
by $X$, and show that its general member is a smooth curve of genus
$3$ for a general $P$.

Later in Section~\ref{section:flat} (see also Section~\ref{subsection:onX}),
we will show that this family is flat and explains the BPS number
of curves of genus 3 and degree 5 on $Y$. The ideal sheaf of this
family of curves in $Y\times X$ will be related with the birational
model $\mathrm{G}(3,T(-1)^{\wedge2})$ of $\hcoY$ and will give the
kernel of the Fourier-Mukai functor which shows that $\sD^{b}(X)\simeq\sD^{b}(Y)$.

\vspace{0.5cm}

\subsection{Constructing the family of curves}

\label{subsection:curves}

Recall our definition of basic morphisms; \[
\xymatrix{\Hes & \ar[l]_{\;\Lrho_{\hcoY}}\hcoY & \ar[l]_{\;\Lpi_{\Zpq}}\Zpq\ar[r]^{\Lrho_{\Zpq}\;\;\;\;} & \mathrm{G}(3,V),}
\]
 where $\Lrho_{\Zpq}:\Zpq\to\mathrm{G}(3,V)$ is a $\mP^{8}$-bundle
since the fiber over a point $[\Pi]$ consists of quadrics which contain
the plane $\mP(\Pi)$.

We consider $X$ in ${\rm G}(2,V)$, and denote by $l_{x}$ the line
in $\mP(V)$ which corresponds to a point $x\in X$. We set $P_{x}=\{[Q]\in P\mid l_{x}\subset Q\}$,
the linear subsystem of $P$ consisting of quadrics which contain
the line $l_{x}$. Then $\dim P_{x}=2$ holds \cite[Prop.3.5.2]{HoTa3}).

\begin{lem} \label{lem:PxH} For any $x\in X$, the plane $P_{x}$
is not contained in the quintic symmetroid $H:=\Hes\cap P$. Moreover,
for a general regular $P$ and a general $x\in X$, the curve $H\cap P_{x}$
is a plane quintic with only three nodes. \end{lem}

\begin{proof} If $P_{x}\subset H$, then it is a divisor on $H$
and $\Lrho_{\hcoY}^{\;-1}(P_{x})=aM$ with some integer $a$, where
$M$ is the generator of ${\rm Pic}(Y)$ and satisfies $M^{3}=10$.
We set $M_{H}:=\sO_{P}(1)|_{H}$. By pulling back the intersection
relation $1=M_{H}\cdot M_{H}\cdot P_{x}$ to $Y$, we have $2=M\cdot M\cdot(aM)=10a$,
which is absurd.

Note that $H$ is a quintic hypersurface while $P_{x}\simeq\mP^{2}$
is a linear subspace of $P$. Therefore $H\cap P_{x}$ is a plane
quintic curve in $P_{x}$. The final part follows from an explicit
calculation of the plane curve $H\cap P_{x}$ by \textsl{Macaulay2}.
We verify in the example below that, for a general $P$ and a general
$x$, the curve $H\cap P_{x}$ has three nodes as singularities. \end{proof}

\noindent \textbf{Example.} (Nodal quintic curve $H\cap P_{x}$) We
fix a generic (regular) linear system of quadrics $P=|Q_{1},Q_{2},...,Q_{5}|$
giving the quadratic forms $q_{i}(\bm{z})=\,^{t}\bm{z}A_{i}\bm{z}$
on $\mP(V)$ by $5\times5$ symmetric matrices. Explicitly we give
them by \[
A_{\lambda}:=\sum_{i=1}^{5}\lambda_{i}A_{i}=\left(\begin{smallmatrix}\lambda_{1} & \lambda_{4} & \lambda_{3} & \lambda_{5} & \lambda_{2}\\
\lambda_{4} & -\lambda_{3} & \lambda_{2}-\lambda_{5} & \lambda_{2} & \lambda_{4}\\
\lambda_{3} & \lambda_{2}-\lambda_{5} & \lambda_{2} & \lambda_{4} & \lambda_{1}+2\lambda_{2}\\
\lambda_{5} & \lambda_{2} & \lambda_{4} & \lambda_{1} & \lambda_{4}\\
\lambda_{2} & \lambda_{4} & \lambda_{1}+2\lambda_{2} & \lambda_{4} & \lambda_{1}+\lambda_{2}\end{smallmatrix}\right).\]
 We identify $[A_{\lambda}]$ with the corresponding point $[\lambda]=[\sum_{i}\lambda_{i}Q_{i}]$
in $P$. Then it easy to verify that $[\bm{z}]=[-1,0,0,1,2]$ and
$\bm{[}w]=[-1,2,0,-1,0]$ satisfies $\,^{t}\bm{z}A_{\lambda}\bm{w}=0$
for any $[\lambda]\in P$, hence defines a point $x$ in $X$ and
also the corresponding line $l_{x}=\langle\bm{z},\bm{w}\rangle$.
The plane $P_{x}$ is determined by the linear equations $\,^{t}\bm{z}A_{\lambda}\bm{z}=\,^{t}\bm{w}A_{\lambda}\bm{w}=0$
as $P_{x}=\{3\lambda_{1}+2\lambda_{4}-\lambda_{5}=2\lambda_{1}-\lambda_{2}-\lambda_{3}=0\}\subset P$.
Then the curve $H\cap P_{x}$ is given by the quintic equation representing
$P_{x}\cap\{\det\, A_{\lambda}=0\}$. By calculating the Jacobian,
it is straightforward to see that $H\cap P_{x}$ has three singularities
at $[\lambda]=[1,\frac{2}{9}(2\alpha^{2}+3\alpha+1),-\frac{2}{9}(2\alpha^{2}+3\alpha-8),\alpha,3+2\alpha]$
for each root $\alpha$ of the cubic $4x^{3}-x^{2}-13x-26=0$ which
is nondegenerate. By writing the local equation of the curve, we verify
that all these singularities are nodal.

The symmetroid $H=\Hes\cap P$ is written by $\{\det\, A_{\lambda}=0\}\subset P$.
By using \textsl{Macaulay2}, we verify that $\Sing~H$ is a smooth
curve of genus $26$ and degree $20$ as noted in \cite{HoTa1}. We
also verify that $\Sing H\cap(H\cap P_{x})=\emptyset$.

Finally, consider a set $\{[\lambda]\in H\cap P_{x}\mid A_{\lambda}(a\bm{z}+b\bm{w})=\bm{0},\exists[a\bm{z}+b\bm{w}]\in l_{x}\}$,
which represents quadrics which contain $l_{x}$ with a point on $l_{x}$
passing through their vertices. Note that by the regularity condition
ii), there is no quadric that contains $l_{x}$ in its vertex. We
can verify that the three nodes on $H\cap P_{x}$ exactly correspond
to this set.~\hfill{}{[}{]} \vspace{0.5cm}

By this example, we see that the normalization of $H\cap P_{x}$ is
a smooth curve of genus three for a general $P$ and a general $x\in X$.
We show that the normalization exists as curves on $Z$ and $Y$.

To state the result precisely, we begin with a preliminary construction.
We define \[
G_{x}:=\{[\Pi]\in\mathrm{G}(3,V)\mid l_{x}\subset\mP(\Pi)\}\]
 and also \[
\Zpq_{x}:=\{([\Pi],[Q])\mid l_{x}\subset\mP(\Pi)\subset Q\}=\Lrho_{\Zpq}^{-1}(G_{x})\subset\Zpq.\]
 $G_{x}$ is a plane in ${\rm G}(3,V)$ and $\Zpq_{x}$ is a $\mP^{8}$-bundle
over $G_{x}$ under the natural projection $\Zpq_{x}\to G_{x}$. Set
\[
\gamma_{x}:=\Zpq_{x}\cap\Lpi_{\Zpq}^{\;-1}(Y)=\{([\Pi],[Q])\mid l_{x}\subset\mP(\Pi)\subset Q,[Q]\in P\}\]
 and denote by $C_{x}$ its image on $Y$. We show

\begin{prop} \label{prop:genus3degree5} For smooth Calabi-Yau threefolds
$X$ and $Y$ which are orthogonal to each other, $\{\gamma_{x}\}_{x\in X}$
is a family of curves of arithmetic genus $3$ and of degree $5$
with respect to $M_{\Zpq}$, and its images $\{C_{x}\}_{x\in X}$
on $Y$ is a family of curves of degree $5$ with respect to $M$.

Moreover, if $X$ and $Y$ are general, then a general member $C_{x}$
is a smooth curve of genus $3$. \end{prop}

\begin{proof} Consider the projections $\overline{\gamma}_{x}:=\Lrho_{\hcoY}\circ\Lpi_{\Zpq}(\gamma_{x})$
and $\Hes_{x}:=\Lrho_{\hcoY}\circ\Lpi_{\Zpq}(\Zpq_{x})$. We define
$\mP_{x}:=\{[Q]\in\mP({\ft S}^{2}V^{*})\mid l_{x}\subset Q\}$. If
we write $x=w_{\bm{x}\bm{y}}\in X$ with $\bm{x},\bm{y}\in\mP(V)$
as a point of ${\ft S}^{2}\mP(V)$, then $\mP_{x}=\{[Q]\in\mP({\ft S}^{2}V^{*})\mid{\empty^{t}\bm{x}}A_{Q}\bm{x}={\empty^{t}\bm{y}}A_{Q}\bm{y}={\empty^{t}\bm{x}}A_{Q}\bm{y}=0\}$,
where $A_{Q}$ is a $5\times5$ symmetric matrix defining the quadric
$Q$. In particular, $\mP_{x}$ is isomorphic to $\mP^{11}$. Then
we have $\Hes_{x}=\{[Q]\mid l_{x}\subset\,^{\exists}\mP(\Pi)\subset Q\}$
and $\overline{\gamma}_{x}=\Hes_{x}\cap P=\Hes_{x}\cap P_{x}$. If
a singular quadric $[Q]\in\Hes$ contains a line $l$, then there
always exists at least one plane $\mP(\Pi)$ such that $l\subset\mP(\Pi)\subset Q$.
Hence we have $\Hes_{x}=\Hes\cap\mP_{x}$ and $\overline{\gamma}_{x}=\Hes_{x}\cap P_{x}=\Hes\cap P_{x}\subset\Hes\cap P$.
Since $H=\Hes\cap P$, we have $\overline{\gamma}_{x}=H\cap P_{x}$,
which is a plane quintic curve by Lemma \ref{lem:PxH}.

Now we note that if a line $l$ is contained in a singular quadric
$Q$ but $l\not\subset\Sing Q$, then there are at most two planes
satisfying $l\subset\mP(\Pi)\subset Q$. For the lines $l_{x}$ of
$x\in X$, we have $\dim P_{x}=2$ as we see above. By the condition
ii) of the regularity of $P$ (see the beginning of Section \ref{section:XY}),
there is no quadric $[Q]\in P$ which contains the line $l_{x}$ in
$\Sing Q$. Therefore $\gamma_{x}\to\overline{\gamma}_{x}$ is finite
of degree at most two. In particular, $\gamma_{x}$ is a curve.

By \cite[Prop.3.5.2]{HoTa3}, $P_{x}$ is a plane in $\mP_{x}$ and
hence $P_{x}$ is of codimension $9$ in $\mP_{x}$. Therefore, since
$\overline{\gamma}_{x}=\Hes_{x}\cap P_{x}\subset\Hes_{x}\cap\mP_{x}=\Hes_{x}$,
we see that $\overline{\gamma}_{x}$ is also a complete intersection
in $\Hes_{x}=\Lrho_{\hcoY}\circ\Lpi_{\Zpq}(\Zpq)$ by 9 hyperplane
sections. Corresponding to this, we also see that $\gamma_{x}$ is
a complete intersection of 9 elements of $|M_{\Zpq_{x}}|:=|M_{\Zpq}\vert_{\Zpq_{x}}|$
in $\Zpq_{x}$ since $\gamma_{x}$ is the pull-back of $\overline{\gamma}_{x}$.

By this fact, we can compute the degree and the arithmetic genus of
$\gamma_{x}$. The degree of $\gamma_{x}$ with respect to $M_{\Zpq}$
is evaluated by using $\Zpq_{x}=\Lrho_{\Zpq}^{\;-1}(G_{x})$ and the
Segre class of the projective bundle $\Zpq=\mP(\sE^{*})\to{\rm G}(3,V)$
as \[
M_{\Zpq}\cdot(\Zpq_{x}\cdot M_{\Zpq}^{9})=M_{\Zpq}^{10}\cdot\Zpq_{x}=s_{2}(\sE|_{G_{x}})=(c_{1}(\sE)^{2}-c_{2}(\sE))G_{x},\]
 which is equal to $(c_{1}(\sE)^{2}-c_{2}(\sE))G_{x}=(10c_{1}(\sO_{\mathrm{G}(3,V)}(1))^{2}-5c_{2}(\eQ))G_{x}$
by Lemma \ref{cla:c1c2}. Since $G_{x}$ is a plane, we have $c_{1}(\sO_{\mathrm{G}(3,V)}(1))^{2}G_{x}=1$.
We note that, by definition, $c_{2}(\eQ)=\sigma_{2}$ which represents
the $4$-cycle $\{[\Pi]\mid t\in\mP(\Pi)\}\subset\mathrm{G}(3,V)$,
parameterizing $2$-planes containing a fixed point $t$ of $\mP^{4}$.
Therefore choosing such a point $t\in\mP^{4}$ so that $t\not\in l_{x}$,
we see $c_{2}(\eQ)G_{x}=1$. Hence we have $M_{\Zpq}^{10}\cdot\Zpq_{x}=M_{\Zpq}\cdot\gamma_{x}=5$.
Since $\deg\gamma_{x}=\deg\overline{\gamma}_{x}=5$, we see that $\gamma_{x}\to\overline{\gamma}_{x}$
is birational. The canonical divisor of $\gamma_{x}$ is the restriction
of $K_{\Zpq_{x}}+9M_{\Zpq_{x}}$. From the relative Euler sequence
of the projective bundle $\Zpq_{x}=\mP(\sE^{*}|_{G_{x}})$ over $G_{x}\simeq\mP^{2}$
and $c_{1}(\sE)=c_{1}(\sO_{\mathrm{G}(3,V)}(4))$, we have $K_{\Zpq_{x}}=(-9M_{\Zpq}+N_{\Zpq})|_{\Zpq_{x}}$,
where $N_{\Zpq}:=\Lrho_{\Zpq}^{\;*}\sO_{{\rm G}(3,V)}(1)$. Thus $K_{\gamma_{x}}={N_{\Zpq}|_{\gamma_{x}}}$.
Using the Segre class again, we evaluate \[
{\Lrho_{\Zpq}}_{*}(M_{\Zpq}^{9}\cdot\Zpq_{x})=s_{1}(\sE|_{G_{x}})=c_{1}(\sE|_{G_{x}})=4N_{\Zpq}|_{G_{x}},\]
 and obtain $\deg K_{\gamma_{x}}=N_{\Zpq}M_{\Zpq}^{9}\cdot\Zpq_{x}=4(N_{\Zpq})^{2}|_{G_{x}}=4$.
Therefore the arithmetic genus of $\gamma_{x}$ is $3$.

Now we consider the image $C_{x}$ on $Y$ of $\gamma_{x}$. Note
that a point $([\Pi],[Q])$ of $\gamma_{x}$ satisfying $l_{x}\subset\mP(\Pi)\subset Q\;([Q]\in P)$
is mapped to a point $([Q],q)$ in $Y$, where $q$ represents a connected
family of planes contained in $Q$. Then $\gamma_{x}\to C_{x}$ is
injective since once we fix a quadric $Q$ of rank $3$ or $4$ and
a connected family $q$ therein, there exists at most one point $([\Pi],[Q])$
which satisfies $[\Pi]\in q$ and $l_{x}\subset\mP(\Pi)$. In particular,
the degree of $C_{x}$ is $5$ with respect to $M$.

Now we assume that $X$ and $Y$ are general. By Lemma \ref{lem:PxH},
the geometric genus of $\overline{\gamma}_{x}$ is three for a general
$x$. Since the arithmetic genus of $\gamma_{x}$ is three, $\gamma_{x}$
is the normalization of $\overline{\gamma}_{x}$. Since $\overline{\gamma}_{x}$
has only nodes as its singularities and $\gamma_{x}\to C_{x}$ is
injective, we conclude $C_{x}\simeq\gamma_{x}$ and $C_{x}$ is a
smooth curve of genus 3 for general $x\in X$. \end{proof}

\vspace{0.5cm}
 \[
\xyFigI\]
 \vspace{0.3cm}
 \begin{fcaption} 

\item \textbf{Fig.1. Curve $C_{x}$ and its {}``shadow'' $C_{x}'$.}
The line $l_{x}$ and quadrics which contain $l_{x}$. If $l_{x}$
passes through the vertex of $Q$, then two points $([Q],q_{1}),([Q],q_{2})$
map to $[Q]$. Otherwise, $([Q],q)\in Y$ is uniquely determined by
$[Q]$. \end{fcaption} \vspace{0.2cm}

Assume that $X$ and $Y$ are general. For a general point $x$ in
$X$, we can verify in the above example that the plane quintic curve
$\overline{\gamma}_{x}$ has only three nodes and does not intersect
with the singular locus $\Sing\, H$. Since $Y\to H$ is a double
cover branched along $\Sing H$ and also the smooth curve $C_{x}$
of degree 5 covers $\overline{\gamma}_{x}$, the inverse image $q^{-1}(\overline{\gamma}_{x})$
is the union of $C_{x}$ and another curve ${C}'_{x}$, which we called
{}``shadow'' curve of $C_{x}$ in Introduction (see Fig.1). As shown
in Fig.1, we note that the shadow curve is also a smooth curve of
genus $3$ and degree $5$ for a general $x$ and intersects at 6
points with $C_{x}$. These 6 points are inverse images of three nodal
points on $\overline{\gamma}_{x}$.

\subsection{The Brauer group of $Y$}

As an interesting corollary to the existence of the curves $\gamma_{x}$
on $Z$, we show that $Y$ has non-trivial Brauer group. Let $N_{Z}:=N_{\Zpq}|_{Z}$
for $N_{\Zpq}=\Lrho_{\Zpq}^{\;*}\sO_{\mathrm{G}(3,V)}(1)$.

\begin{prop} \label{prop:Brauer} The $\mP^{1}$-fibration $Z\to Y$
is not associated to a locally free sheaf of rank two on $Y$. In
particular, the Brauer group of $Y$ contains a non-trivial $2$-torsion
element. \end{prop}

\begin{proof} Assume by contradiction that $Z=\mP(\sA)$ for some
locally free sheaf $\sA$ of rank two on $Y$. Since a fiber of $Z\to Y$
is of degree two with respect to $N_{{Z}}$, we can write $N_{Z}\equiv2H_{\mP(\sA)}+a{M}_{Y}$,
where $a$ is an integer (here we use $\rho(Y)=1$). Since $N_{{Z}}\cdot\gamma_{x}=4$
and ${M}_{Y}\cdot\gamma_{x}=5$ by the proof of Proposition \ref{prop:genus3degree5},
we have $4=2H_{\mP(\sA)}\cdot\gamma'_{x}+5a$. Thus $a$ is even and
$\frac{1}{2}N_{{Z}}$ is numerically equivalent to the Cartier divisor
$H_{\mP(\sA)}+\frac{1}{2}a{M}_{Y}$. Note that \[
(N_{{Z}})^{4}=N_{\Zpq}^{4}M_{\Zpq}^{10}=s_{2}(\sE)c_{1}(\sO_{\mathrm{G}(3,V)}(1))^{4}=(c_{1}(\sE)^{2}-c_{2}(\sE))c_{1}(\sO_{\mathrm{G}(3,V)}(1))^{4}.\]
 By Lemma \ref{cla:c1c2}, we have $(N_{{Z}})^{4}=40$. Then $(\frac{1}{2}N_{{Z}})^{4}$
is not an integer, a contradiction. \end{proof}

We present a further discussion on the Brauer groups of $X$ and $Y$
in Subsection \ref{subsection:fund}.

\vspace{1cm}
\vspace{1cm}

\section{Birational geometry of $\hcoY$ and generically conic bundles}

\label{section:birationalY}

Let us consider $\hcoY_{3}:=\mathrm{G}(3,T(-1)^{\wedge2})$, which
is a $\mathrm{G}(3,6)$-bundle over $\mP(V)$. The fiber of $\hcoY_{3}\to\mP(V)$
over a point $[V_{1}]\in\mP(V)$ parameterizes planes in $\mP(\Lwedge^{2}(V/V_{1}))$.
$\hcoY_{3}$ appeared naturally in the construction of a resolution
$\widetilde{\hcoY}\to\hcoY$ and played important roles to describe
a Lefschetz collection in $\sD^{b}(\widetilde{\hcoY})$ \cite{HoTa3}.

\subsection{Birational geometry of $\hcoY$\label{sub:Birat-Y}}

Here we briefly describe the construction of the resolution $\widetilde{\hcoY}\to\hcoY$
which can be summarized in the diagram:\begin{equation}
\xyResolY\label{eq:ResolY}\end{equation}
Note that $\hcoY_{3}$ defined above is equivalently described by\[
\hcoY_{3}=\left\{ ([U],[V_{1}])\in\rG(3,\wedge^{3}V)\times\mP(V)\mid U\wedge V_{1}=0\right\} ,\]
since $U\wedge V_{1}=0$ implies $[U]=[\bar{U}\wedge V_{1}]$ for
some $[\bar{U}]\in\rG(3,\wedge^{2}(V/V_{1}))$ and there is a bijective
correspondence between $([U],[V_{1}])$ and $([\bar{U}],[V_{1}])\in\rG(3,T(-1)^{\wedge2})$.
With this definition of $\hcoY_{3}$, $\overline{\hcoY}$ is defined
by the projection to the first factor with the reduced structure \cite[Def. 5.3.1]{HoTa3}.
$\hcoY_{2}$ and $\widetilde{\hcoY}$ are described explicitly as
blow-ups of $\hcoY_{3}$ and $\overline{\hcoY}$, and the fibers of
the resolution $\widetilde{\hcoY}\to\hcoY$ are described in detail
{[}ibid, Sect.5.6{]}. 

We can see the birational correspondence between $\hcoY_{3}$ and
$\hcoY$ as follows: As described in Subsection \ref{section:Quintic},
the fiber of $\Zpq\to\hcoY$ over $y\in\hcoY\setminus G_{\hcoY}$
is a smooth conic in $\rG(3,V)$ which parametrizes a family of planes
contained in the corresponding quadric $\Lrho_{\hcoY}(y)=[Q_{y}]\in\Hes$.
Observing this, we wrote each point $y\in\hcoY\setminus G_{\hcoY}$
by the pair $([Q_{y}],q_{y})$ in the proof of Proposition \ref{prop:genus3degree5}.
Consider the Pl\"ucker embedding $\rG(3,V)\hookrightarrow\mP(\wedge^{3}V)$.
In $\mP(\wedge^{3}V)$, a conic $q$ on $\rG(3,V)$ determines the
corresponding plane $\mP_{q}^{2}$ in $\mP(\wedge^{3}V)$. If $\rank y=4$
(i.e., $\rank Q_{y}=4$), then $\mP_{q_{y}}^{2}$ takes the form \[
[U_{y}]=[\bar{U}_{y}\wedge V_{1}]\;\text{with some }[\bar{U}_{y}]\in\rG(3,\wedge^{2}(V/V_{1})),\]
since the planes in $Q_{y}$ parametrized by $q_{y}$ contain the
vertex $[V_{1}]$ of $Q_{y}$ in common. In this case, the intersection
$\mP(U_{y})\cap\rG(3,V)$ in $\mP(\wedge^{3}V)$ recovers the conic
$q_{y}$, since \begin{equation}
\mP(U_{y})\cap\rG(3,V)\text{ in }\mP(\wedge^{3}V)\simeq\mP(\bar{U}_{y})\cap\rG(2,V/V_{1})\text{ in }\mP(\wedge^{2}(V/V_{1}))\label{eq:plane-to-conic}\end{equation}
and $\mP(\bar{U}_{y})\not\subset\rG(2,V/V_{1})$ holds. Namely the
intersection with the quadric $\rG(2,V/V_{1})$ determines a conic
$\bar{q}_{y}$ in $\rG(2,V/V_{1})$ and also $q_{y}$ in $\rG(3,V)$,
and in turn the quadric $Q_{y}$, hence $y\in\hcoY$ \cite[Sect.~5.1,~5.2]{HoTa3}.
Using the decomposition as $sl(V/V_{1})$-module $\wedge^{3}(\wedge^{2}(V/V_{1}))=\ft{S}^{2}(V/V_{1})\oplus\ft{S}^{2}(V/V_{1})^{*}$,
it has been shown that generic $([U],[V_{1}])\in\hcoY_{3}$ determines
$y\in\hcoY$ in this way {[}ibid, Sect.$\,$5.5{]}. The birational
map $\hcoY_{3}\dashrightarrow\hcoY$ is based on this correspondence.\[
\;\]

\[
\xyFigYs\]
 \begin{fcaption} 

\item \textbf{Fig.2. Birational geometries of $\hcoY$.} The singular
locus $\overline{\Prt}_{\rho}\simeq\rG(2,V)$ of $\overline{\hcoY}$
and also $\Lrho_{3}$-exceptional set $\Prt_{\rho}=\mathrm{F}(1,2,V)$
in $\hcoY_{3}$ and the exceptional set $G_{\rho}$ of $\widetilde{\hcoY}\to\overline{\hcoY}$
are depicted. $F_{\rho}$ and $F_{\widetilde{\hcoY}}$ are exceptional
divisors which are contracted to $G_{\rho}$ and $G_{\hcoY}$, respectively.
$G_{\hcoY}$ is the singular locus of $\hcoY$. \end{fcaption} \vspace{0.5cm}

The above correspondence of the planes $\mP(U)\subset\mP(\wedge^{3}V)$
to conics does not work obviously when $\mP(U)\subset\rG(3,V)$. There
are two types of planes contained in $\rG(3,V).$ The first one is
\[
\mathrm{P}_{V_{2}}=\left\{ [\Pi]\in\rG(3,V)\mid V_{2}\subset\Pi\right\} \simeq\mP^{2}\]
with some $V_{2}$, and the second is\[
\mathrm{P}_{V_{1}V_{4}}=\left\{ [\Pi]\in\rG(3,V)\mid V_{1}\subset\Pi\subset V_{4}\right\} \simeq\mP^{2}\]
with some $V_{1}\subset V_{4}$. They are called \textit{$\rho$-plane}
and \textit{$\sigma$-plane}, respectively, and conics on $\rG(3,V)$
contained in these planes are called \textit{$\rho$-conic} and \textit{$\sigma$-conic}.
Other types of conics on $\rG(3,V)$ are called \textit{$\tau$-conic},
and they are determined by the intersection $\mP(U)\cap\rG(3,V)$
with $[U]\in\overline{\hcoY}\setminus(\overline{\Prt}_{\rho}\sqcup\overline{\Prt}_{\sigma})$.

The above two types of planes in $\rG(3,V)$ determine the corresponding
planes in $\mP(\wedge^{3}V)$ under the Pl\"ucker embedding $\rG(3,V)\hookrightarrow\mP(\wedge^{3}V$)
and determine the corresponding loci $\overline{\Prt}_{\rho}$ and
$\overline{\Prt}_{\sigma}$ in $\overline{\hcoY}$, where \[
\overline{\Prt}_{\rho}\simeq\rG(2,V),\;\;\overline{\Prt}_{\sigma}\simeq F(1,4,V).\]
It has been shown that the singular locus in $\overline{\hcoY}$ is
exactly along $\overline{\Prt}_{\rho}$ with the singularity being
given by the affine cone over the Segre embedding $\mP^{1}\times\mP^{5}$
for each point. The resolution $\hcoY_{3}\to\overline{\hcoY}$ is
given by the blow-up along $\overline{\Prt}_{\rho}$ in one direction
giving the exceptional set $\Prt_{\rho}=F(1,2,V)$. The blow-up along
$\overline{\Prt}_{\rho}$ in the other direction gives the resolution
$\widetilde{\hcoY}\to\overline{\hcoY},$ where the exceptional set
$G_{\rho}$ is described by $\mP(\ft{S}^{2}\eG^{*})$ with the universal
quotient bundle $\eG$ of $\rG(2,V)$ (cf. \cite[(5.11)]{HoTa3}).
$\hcoY_{2}$ is obtained by further blow-up of $\hcoY_{3}$ (or $\widetilde{\hcoY}$)
with the exceptional divisor $F_{\rho}$ (see Fig. 2). 

Planes contained in a quadric $\rank Q=3$ determine a (smooth)$\rho$-conic,
since they contain the vertex $V_{2}$ of $Q$ in common. Hence points
$y=([Q],q)\in\hcoY$ with $\rank Q$=3 correspond to generic points
in the exceptional set $G_{\rho}=\mP(\ft{S}^{2}\eG^{*}).$ The resolution
$\widetilde{\hcoY}\to\hcoY$ contracts a prime divisor $F_{\widetilde{\hcoY}}$
to $G_{\hcoY}$ which correspond to quadrics $[Q]\in\Hes$ with $\rank Q\leq2$.
The divisor consists of $\tau$- and $\rho$-conics in $\rG(3,V)$
with rank $\leq2$ and also the $\sigma$-planes (see \cite[Sect.~5.6]{HoTa3}
for the complete description). In particular, the fiber of $\Lrho_{\widetilde{\hcoY}}$
over a point $[Q]\in G_{\hcoY}$ with $\rank Q=2$ is isomorphic to
$\mP^{2}\times\mP^{2}$, which generically parametrizes $\tau$-conics
of rank two, and $\rho$-conics of rank two appear on the diagonal. 

\vspace{0.5cm}
 Now we list some divisors which will be used later. In that follows,
we will use the following conventions without mentioning at each time:

\begin{myitem} 

\item[$K_{\Sigma}$:] canonical divisor on a normal variety $\Sigma$
.

\item[$L_{\Sigma}$:] pull back on a variety $\Sigma$ of $\sO(1)$
if there is a morphism $\Sigma\to\mP(V)$. 

\item[$M_{\Sigma}$:] pull back on a variety $\Sigma$ of $\sO_{\Hes}(1)$
if there is a morphism $\Sigma\to\Hes$. 

\item[$N_{\Sigma}$:] pull back on a variety $\Sigma$ of $\sO_{\mathrm{G}(3,V)}(1)$
if there is a morphism $\Sigma\to\mathrm{G}(3,V)$.

\item[ $H_{\mP(\sE)}:$] tautological divisor of the projective
bundle $\pi:\mP(\sE)\to X$, namely $H_{\mP(\sE)}=\sO_{\mP(\sE)}(1)$
with the property $\pi_{*}\sO_{\mP(\sE)}(1)=\sE^{*}$. 

\end{myitem} We denote the universal exact sequence on $\hcoY_{3}=\mathrm{G}(3,T(-1)^{\wedge2})$
by \begin{equation}
0\to\sS\to\Lpi_{3}^{\;*}(T(-1)^{\wedge2})\to\sQ\to0,\label{eq:univ}\end{equation}
where $\sS$ is the relative universal subbundle of rank three and
$\sQ$ is the relative universal quotient bundle of rank three. We
note that $\Prt_{\rho}=F(1,2,V)\simeq\mP(T(-1))$.

\begin{prop} \label{prop:div-relations}The following relations among
divisors hold$:$

\begin{myitem2}

\item[\;\;$(1)$] $\det\sQ=\det\sS^{*}+3L_{\hcoY_{3}}=\det\{\sS^{*}(L_{\hcoY_{3}})\},$

\item[\;\;$(2)$] $K_{\hcoY_{3}}=-6\det\sQ+4L_{\hcoY_{3}},$

\item[\;\;$(3)$] $\sQ|_{\Prt_{\rho}}\simeq\sS^{*}(L_{\hcoY_{3}})|_{\Prt_{\rho}}$, 

\item[\;\;$(4)$] $\det\sQ|_{\Prt_{\rho}}=2(H_{\Prt_{\rho}}+L_{\Prt_{\rho}})$, 

\item[\;\;$(5)$]  $M_{\hcoY_{2}}=\Lrho_{_{2}}^{*}(\det\sQ)-L_{\hcoY_{2}}-F_{\rho}$.

\end{myitem2}

\end{prop}

\begin{proof} (1) is immediate from the exact sequence (\ref{eq:univ}).
(2) follows from $T_{\hcoY_{3}/\mP(V)}=\sS^{*}\otimes\sQ$ and $K_{\hcoY_{3}}=-\det T_{\hcoY_{3}/\mP(V)}+5L_{\hcoY_{3}}$.
(3) and (4) are obtained in \cite[Props. 6.3.1, 6.3.2]{HoTa3}. See
{[}ibid, Prop. 6.4.1{]} for (5). 

\end{proof}

$\;$

\subsection{Generically conic bundle $\Zpq_{3}\to\hcoY_{3}$}

\label{Z3Y3}

\textcolor{black}{Consider the universal subbundle $\sS$ and define
the projective bundle $\Zpq_{3}^{u}:=\mP(\sS)\subset\mP(T(-1)^{\wedge2})\times_{\mP(V)}\hcoY_{3}$
over $\hcoY_{3}=\rG(3,T(-1)^{\wedge2})$, where the superscript $u$
stands for universal. Since the fiber of $\mP(\sS$) over $([\bar{U}],[V_{1}])$
is the plane $\mP(\bar{U})\subset\mP(\wedge^{2}(V/V_{1}))$, the intersection
with $\rG(2,V/V_{1})$ in (\ref{eq:plane-to-conic}) can be described
by \[
\Zpq_{3}^{u}\cap\left(\rG(2,T(-1))\times_{\mP(V)}\hcoY_{3}\right),\]
using the inclusion $\rG(2,T(-1))\times_{\mP(V)}\hcoY_{3}\hookrightarrow\mP(T(-1)^{\wedge2})\times_{\mP(V)}\hcoY_{3}$.
We denote this restriction by $\Zpq_{3}$, and note that it can be
written explicitly by} \[
\Zpq_{3}=\left\{ \left([\bar{a}\wedge\bar{b}],([\bar{U}],[V_{1}])\right)\mid[\bar{a}\wedge\bar{b}]\in\mP(\bar{U})\cap\rG(2,V/V_{1}),\;([\bar{U}],[V_{1}])\in\hcoY_{3}\right\} .\]
By definition, $\Zpq_{3}$ fits into the following diagram with the
natural morphisms:

\begin{equation}
\begin{matrix}\xydiagIV\end{matrix}\label{eq:diag4}\end{equation}
 Note that $\Lpi_{G'}$ maps the point $\left([\bar{a}\wedge\bar{b}],([\bar{U}],[V_{1}])\right)$
to $([\bar{a}\wedge\bar{b}],[V_{1}])\in\rG(2,T(-1)),$ which may be
considered as a point $([a\wedge b\wedge v_{1}],[V_{1}])$ with $V_{1}=\mC v_{1}$
in $\rG(3,V)\times\mP(V)$. Hence we naturally have \[
\Zpq_{3}\subset\rG(2,T(-1))\times_{\mP(V)}\hcoY_{3}\,\subset\,\rG(3,V)\times\hcoY_{3},\]
where we use $\rG(2,T(-1))\subset\rG(3,V)\times\mP(V)$ (actually,
$\rG(2,T(-1))\simeq F(1,3,V)$ is the universal family of planes in
$\mP(V)$ parameterized by $\mathrm{G}(3,V)$). As explained in (\ref{eq:plane-to-conic}),
the intersection $\mP(\bar{U})\cap\rG(2,V/V_{1})$ generically defines
a conic on $\rG(2,V/V_{1})$ or $\rG(3,V)$. Therefore we have;

\begin{prop} \label{prop:Z3} $\Lpi_{3}\colon\Zpq_{3}\to\hcoY_{3}$
is a generically conic bundle. More precisely, fibers over $\hcoY_{3}\setminus(\Prt_{\rho}\sqcup\Prt_{\sigma})$
are conics in $\rG(3,V).$ 

\end{prop}

\vspace{0.3cm}
Below are some properties of $\Zpq_{3}$ which will be used in later
sections.

\begin{prop} \label{cla:G25} $\Lpi_{G'}\colon\Zpq_{3}\to\mathrm{G}(2,T(-1))$
is a $\mathrm{G}(2,5)$-bundle. In particular, $\Zpq_{3}$ is smooth.
\end{prop}

\begin{proof} By definition, the fiber over $([\bar{a}\wedge\bar{b}],[V_{1}])\in\rG(2,T(-1))$
consists of $\bar{U}\simeq\mC^{3}$ (or points $([\bar{a}\wedge\bar{b}],([\bar{U}],[V_{1}]))\in\Zpq_{3}$)
satisfying \[
\bar{a}\wedge\bar{b}\in\bar{U}\subset\wedge^{2}(V/V_{1}).\]
The claim is immediate if we write the above condition as $\mC\subset\bar{U}\subset\mC^{6}$.
\end{proof}

We define \[
\Zpq_{\rho}:=\Lpi_{3'}^{-1}(\Prt_{\rho})=\mP(\sS|_{\Prt_{\rho}})\text{ and \ensuremath{\Zpq_{\sigma}}:=\ensuremath{\Lpi_{3'}^{-1}}(\ensuremath{\Prt_{\sigma}})=\ensuremath{\mP}(\ensuremath{\sS|_{\Prt_{\sigma}}}).}\]
Then the fibers of $\Zpq_{\rho}\to\Prt_{\rho}$ and $\Zpq_{\sigma}\to\Prt_{\sigma}$
are the family of planes parameterized by $\Prt_{\rho}$ and $\Prt_{\sigma}$
respectively. The diagram (\ref{eq:diag4}) naturally restricts to
$\Prt_{\rho}$ and $\Prt_{\sigma}$, respectively.

\begin{prop} \label{cla:P1bdl} $\Lpi_{G'}\vert_{\Zpq_{\rho}}\colon\Zpq_{\rho}\to\mathrm{G}(2,T(-1))$
is a $\mP^{1}$-bundle. If we consider this as a sub-bundle of the
$\mathrm{G}(2,5)$-bundle $\Lpi_{G'}\colon\Zpq_{3}\to\mathrm{G}(2,T(-1))$,
then the fibers of $\Lpi_{G'}\vert_{\Zpq_{\rho}}$ are conics in $\mathrm{G}(2,5)$.
Similar properties hold also for $\Lpi_{G'}|_{\Zpq_{\sigma}}\colon\Zpq_{\sigma}\to\mathrm{G}(2,T(-1))$.
\end{prop}

\begin{proof} Consider a point $([\bar{a}\wedge\bar{b}],[V_{1}])\in\mathrm{G}(2,T(-1))$.
We set $V_{3}=\langle a,b,v_{1}\rangle$ with $V_{1}=\langle v_{1}\rangle$
and $\bar{a}=a\text{ mod }V_{1}$, $\bar{b}=b\text{ mod }V_{1}$.
By definition, the fiber over the point consists of $[\bar{U}]\subset\rG(3,\wedge^{2}(V/V_{1}))$
satisfying \[
\bar{a}\wedge\bar{b}\in\bar{U}\subset\wedge^{2}\left(V/V_{1}\right)\text{ and }[U]=[\bar{U}\wedge V_{1}]\text{ is a }\rho\text{-plane}.\]
We note that the former condition is rephrased by $\wedge^{3}V_{3}\subset U$,
and the latter by $[U]=\mathrm{P}_{V_{2}}=[(V/V_{2})\wedge(\wedge^{2}V_{2})]$
for some $V_{2}$ with $V_{1}\subset V_{2}$. \textcolor{black}{From
these, we see that the fiber is described by $\left\{ [\bar{U}]=[(V/V_{2})\wedge(V_{2}/V_{1})]\mid V_{1}\subset V_{2}\subset V_{3}\right\} \simeq\mP(V_{3}/V_{1})$. }

\textcolor{black}{Similarly, we described the fiber $\Lpi_{G'}^{-1}([\bar{a}\wedge\bar{b}],[V_{1}])$
by $\left\{ [\bar{U}]\mid\bar{a}\wedge\bar{b}\in\bar{U}\subset\wedge^{2}(V/V_{1})\right\} $
$\simeq\rG(2,5)$. Let us consider the Pl\"ucker embedding $\bar{[U}]\in\rG(3,\wedge^{2}(V/V_{1}))\mapsto[\wedge^{3}\bar{U}]\in\mP(\wedge^{3}(\wedge^{2}(V/V_{1})))$
and also the decomposition $\wedge^{3}(\wedge^{2}(V/V_{1}))=\ft{S}^{2}(V/V_{1})\oplus\ft{S}^{2}(V/V_{1})^{*}$
for the projective space $\mP(\wedge^{3}(\wedge^{2}(V/V_{1})))$.
Obviously, the fiber $\Lpi_{G'}^{-1}([\bar{a}\wedge\bar{b}],[V_{1}])$
is mapped to a Grassmann $\rG(2,5)$ in $\rG(3,\wedge^{2}(V/V_{1}))$.
For the other fiber, using the {}``double spin coordinate'' of $\rG(3,6)$
given in \cite[Appendix A]{HoTa3}, it is straightforward to see that
$\Lpi_{G'}\vert_{\Zpq_{\rho}}^{-1}([\bar{a}\wedge\bar{b}],[V_{1}])\simeq\mP(V_{3}/V_{1})$
is mapped into the linear subspace $\mP(\ft{S}^{2}(V_{3}/V_{1}))\subset\mP(\ft{S}^{2}(V/V_{1}))$
by the second Veronese embedding, hence its image is given by a conic. }

The corresponding properties for $\Lpi_{G'}\vert_{\Zpq_{\sigma}}$
follow in a similar way.\end{proof} 

Finally, we observe a fact about the relative Euler sequence for the
projective bundle $\Lpi_{\rho}\colon\Zpq_{\rho}=\mP(\sS|_{\Prt_{\rho}})\to\Prt_{\rho}$;
\begin{equation}
0\to\sO_{\Zpq_{\rho}}(-1)\to\Lpi_{\rho}^{\;*}(\sS|_{\Prt_{\rho}})\to\sR_{\rho}\to0,\label{eq:EulerS}\end{equation}
 where we set $\sR_{\rho}:=T_{\Zpq_{\rho}/\Prt_{\rho}}\otimes\sO_{\Zpq_{\rho}}(-1)$.

\begin{lem} \label{cla:WZp} Let $\eQ_{\rho}$ be the pull-back of
the universal quotient bundle $\eQ$ on $\mathrm{G}(3,V)$ by the
composition $\Lrho_{G}\circ\Lpi_{G'}\vert_{\Zpq_{\rho}}:\Zpq_{\rho}\to\mathrm{G}(2,T(-1))\to\mathrm{G}(3,V)$.
Then \[
\eQ_{\rho}\simeq\sR_{\rho}\otimes{\Lpi_{\rho}}^{*}\sO_{\mP(T(-1))}(1),\]
 where $\sO_{\mP(T(-1))}(1)$ is the tautological sheaf on $\mP(T(-1))=F(1,2,V)=\Prt_{\rho}$.
\end{lem}

\begin{proof} Let $([\bar{U}],[V_{1}])$ with $[\bar{U}]=[(V/V_{2})\wedge(V_{2}/V_{1})]$
be a point of $\Prt_{\rho}$. Then a point $z_{\rho}=([\bar{U}_{1}],([\bar{U}],[V_{1}]))$
in the fiber $\Lpi_{\rho}^{-1}(\Prt_{\rho})\subset\Zpq_{\rho}$ may
be written by $\bar{U}_{1}=(V_{3}/V_{2})\wedge(V_{2}/V_{1})$ with
$V_{2}\subset V_{3}$. Since the composition $\Lrho_{G}\circ\Lpi_{G'}\vert_{\Zpq_{\rho}}$
sends $z_{\rho}$ to $\wedge^{3}V_{3}$, we have $\eQ_{\rho}\vert_{z_{\rho}}=V/V_{3}$.
On the other hand, the Euler sequence (\ref{eq:EulerS}) restricts
at $z_{\rho}$ to $0\to\bar{U}_{1}\to\bar{U}\to\bar{U}/\bar{U}_{1}\to0.$
Hence we have \[
\sR_{\rho}\otimes\Lpi_{\rho}^{\;*}\sO_{\mP(T(-1))}(1)\vert_{z_{\rho}}=\bar{U}/\bar{U}_{1}\otimes(V_{2}/V_{1})^{*}\simeq V/V_{3},\]
 which shows the claim. \end{proof}

$\;$

\subsection{Generically conic bundles $\Lpi_{\Zpq_{2}}\colon\Zpq_{2}\to\hcoY_{2}$
\label{Z2Y2}}

Let us recall that $\Zpq_{3}^{u}=\mP(\sS$) is the projective bundle
over $\hcoY_{3}$ and denote the natural projection by $\Lpi_{3^{u}}:\Zpq_{3}^{u}\to\hcoY_{3}$.
The generically conic bundle $\Lpi_{3'}:\Zpq_{3}\to\hcoY_{3}$ is
the fiber-wise intersection of $\Zpq_{3}^{u}$ with the Grassmann
bundle $\rG(2,T(-1))\times_{\mP(V)}\hcoY_{3}$. Let $\Zpq_{\rho}$
be the inverse image $\Lpi_{3'}^{-1}(\Prt_{\rho})$ of $\Prt_{\rho}\simeq F(1,2,V)$.
$\hcoY_{2}$ is the blow-up of $\hcoY_{3}$ along $\Prt_{\rho}$ with
the exceptional divisor $F_{\rho}$. We will consider the blow-up
$\rho_{2'}:\Zpq_{2}\to\Zpq_{3}$ of $\Zpq_{3}$ along $\Zpq_{\rho}$
and denote by $E_{\rho}$ its exceptional divisor. Then there is a
projection $\Lpi_{2'}:\Zpq_{2}\to\hcoY_{2}$ with the following commutative
diagram:\[
\begin{matrix}\xymatrix{(\Zpq_{\rho}\subset\Zpq_{3})\ar[d]^{\Lpi_{3'}} & \ar[l]^{\rho_{2'}}(E_{\rho}\subset\Zpq_{2})\ar[d]^{\Lpi_{2'}}\\
(\Prt_{\rho}\subset\hcoY_{3}) & \ar[l]^{\rho_{2}}(F_{\rho}\subset\hcoY_{2})}
\end{matrix}\]
By construction, the fibers of $\Lpi_{2'}$ over $\hcoY_{2}\setminus(F_{\rho}\cup\Lrho_{2}^{-1}(\Prt_{\sigma}))$
are $\tau$-conics in the fibers of the Grassmann bundle $\rG(2,T(-1))$.
We will show that the fibers of $\Lpi_{2'}$ over $F_{\rho}$ are
$\rho$-conics.

We first describe the exceptional set $F_{\rho}$ (see also \cite[Sect.~5.4]{HoTa3}).

\begin{lem} \label{cla:G(2T)} $\mathrm{G}(2,T(-1))\in|2\Hwt+\Lwt|$.
\end{lem} 

\begin{proof} Note that $\mathrm{G}(2,T(-1))$ is a divisor in $\mP(T(-1)^{\wedge2})$.
Fix a point $[V_{1}]\in\mP(V)$. The defining equation of $\rG(2,V/V_{1})$
in $\mP(\wedge^{2}(V/V_{1}))$ is the Pl\"ucker quadric, which defines
a symmetric bi-linear form $\wedge^{2}(V/V_{1})\times\wedge^{2}(V/V_{1})$
$\to\wedge^{4}(V/V_{1}).$ Since this globalizes to $\wedge^{2}T(-1)\times\wedge^{2}T(-1)\to\wedge^{4}T(-1)\simeq\sO(1)$,
the defining equation of $\rG(2,T(-1))$ in $\mP(T(-1)^{\wedge2})$
is an element of $H^{0}(\mP(V),\ft{S}^{2}(\Omega(1)^{\wedge2})\otimes\sO(1)).$
This proves the claim.\end{proof}

\begin{lem} \label{lem:Normal-Bundles}The normal bundles of $\Prt_{\rho}$
in $\hcoY_{3}$ and $\Prt_{\sigma}$ in $\hcoY_{3}$, respectively,
are given by \[
\sN_{\Prt_{\rho}/\hcoY_{3}}=\ft{S}^{2}\sS^{*}\otimes L_{\hcoY_{3}}\vert_{\Prt_{\rho}}\text{ and }\sN_{\Prt_{\sigma}/\hcoY_{3}}=\ft{S}^{2}\sS^{2}\otimes L_{\hcoY_{3}}\vert_{\Prt_{\sigma}}.\]

\end{lem}

\begin{proof} Recall that both $\Prt_{\rho}$ and $\Prt_{\sigma}$
consist of points $([\bar{U}],[V_{1}])\in\hcoY_{3}$ satisfying $\mP(\bar{U})\cap\rG(2,V/V_{1})=\mP(\bar{U})$
in $\mP(\wedge^{2}(V/V_{1}))$. As described in the above Lemma \ref{cla:G(2T)},
the defining equation of $\rG(2,T(-1))\subset\mP(T(-1)^{\wedge2})$
is given by a section of $\ft{S}^{2}(T(-1)^{\wedge2})^{*}\otimes\sO(1)$
over $\mP(V)$. Pulling back this by $\Lpi_{3}:\hcoY_{3}\to\mP(V)$
and composing with the surjection $\Lpi_{3}^{\;*}(T(-1)^{\wedge2})^{*}\to\sS^{*}$,
we obtain a section of $\ft{S}^{2}\sS^{*}\otimes L_{\hcoY_{3}}$ over
$\hcoY_{3}$. $\Prt_{\rho}\sqcup\Prt_{\sigma}$ is exactly the scheme
of the zeros of this section (which is isomorphic to the orthogonal
Grassmann bundle $\mathrm{OG}(3,T(-1)^{\wedge2})$). The claimed forms
of the normal bundles follow from this. \end{proof}

\begin{prop} $F_{\rho}=\mP(\ft{S}^{2}\sS^{*}\otimes L_{\hcoY_{3}}\vert_{\Prt_{\rho}})$.
The fibers of $F_{\rho}\to\Prt_{\rho}$ are conics in the $\rho$-planes
parametrized by $\Prt_{\rho}$. \end{prop}

\begin{proof} The first claim follows from $F_{\rho}=\mP(\sN_{\Prt_{\rho}/\hcoY_{3}})$
with the above Lemma \ref{lem:Normal-Bundles}. For the second claim,
let us recall that $\Prt_{\rho}\subset\Zpq_{3}$ consists of points
$([\bar{U}],[V_{1}])$ with $[\bar{U}]=[(V/V_{2})\wedge(V_{2}/V_{1})]$
parametrized by $[V_{1}\subset V_{2}]\in F(1,2,V)$. Then the fiber
of $F_{\rho}\to\Prt_{\rho}$ over a point $([\bar{U}],[V_{1}])$ is
given by \[
\mP(\ft{S}^{2}\sS^{*}\otimes L_{\hcoY_{3}}\vert_{([\bar{U}],[V_{1}])})=\mP(\ft{S}^{2}\bar{U}^{*}\otimes V_{1}^{*})\simeq\mP(\ft{S}^{2}\bar{U}^{*}),\]
which we can identify with conics on the $\rho$-plane $\mP(\bar{U})$
as claimed. \end{proof}

Under the composition map $\hcoY_{2}\to\widetilde{\hcoY}\to\hcoY$,
generic points on $F_{\rho}$ are mapped to points $([Q],q)\in\hcoY$
with $\rank Q=3$ and the corresponding (smooth) $\rho$-conic $q$
in $\rG(3,V)$, i.e., the $\mP^{1}$-family of planes contained in
the quadric $Q$. More precisely, the image of $F_{\rho}$ in $\widetilde{\hcoY}$
has a bijective correspondence to the pairs $([Q],q)$ with $\rank Q\leq3$
and a $\rho$-conic $(\rank q\leq3)$ which parametrizes planes contained
in the quadric $Q$. See \cite[Sect.~5.5, 5.6]{HoTa3} for more complete
descriptions. 

\begin{prop}\label{prop:gen-conic-bundle2} $E_{\rho}\to F_{\rho}$
is the universal family of conics parametrized by $F_{\rho}$. Hence
$\Zpq_{2}\to\hcoY_{2}$ is a generically conic bundle with fibers
over $\hcoY_{2}\setminus\Prt_{\sigma}$ being conics on $\rG(3,V)$
and the fibers over $\Prt_{\sigma}$ being $\sigma$-planes on $\rG(3,V)$.
$($For notational simplicity, we write the transform of $\Prt_{\sigma}\subset\hcoY_{3}$
on $\hcoY_{2}$ by the same $\Prt_{\sigma}$.$)$

\end{prop}

\begin{proof} We describe the blow-up $\Lrho_{2'}:\Zpq_{2}\to\Zpq_{3}$
along $\Zpq_{\rho}\subset\Zpq_{3}$. For this, we consider the subvariety
$\Zpq_{\rho}=\Lpi_{3'}^{-1}(\Prt_{\rho})=\mP(\sS\vert_{\Prt_{\rho}})$
in $\Zpq_{3}^{u}=\mP(\sS$) with the following normal bundle sequence:\[
0\to\sN_{\Zpq_{\rho}/\Zpq_{3}}\to\sN_{\Zpq_{\rho}/\Zpq_{3}^{u}}\to\sN_{\Zpq_{3}/\Zpq_{3}^{u}}\vert_{\Zpq_{\rho}}\to0.\]
We note that since $\Lpi_{3^{u}}:\Zpq_{3}^{u}\to\hcoY_{3}$ is a projective
bundle and hence is flat, we have \[
\sN_{\Zpq_{\rho}/\Zpq_{3}^{u}}=(\Lpi_{3^{u}}\vert_{\Zpq_{\rho}})^{*}\sN_{\Prt_{\rho}/\Zpq_{3}}=(\Lpi_{3^{u}}\vert_{\Zpq_{\rho}})^{*}\ft{S}^{2}\sS^{*}(L_{\hcoY_{3}})\vert_{\Prt_{\rho}},\]
where we use Lemma \ref{lem:Normal-Bundles} for the normal bundle
$\sN_{\Prt_{\rho}/\Zpq_{3}}$. From Lemma \ref{cla:G(2T)} and the
definition $\Zpq_{3}=\rG(2,T(-1))\times_{\mP(V)}\hcoY_{3}\subset\Zpq_{3}^{u}$,
we have\begin{equation}
0\to\sN_{\Zpq_{\rho}/\Zpq_{3}}\to(\Lpi_{3^{u}}\vert_{\Zpq_{\rho}})^{*}\ft{S}^{2}\sS^{*}(L_{\hcoY_{3}})\vert_{\Prt_{\rho}}\to\sO_{\Zpq_{3}^{u}}(2H_{\Zpq_{3}^{u}}+L_{\Zpq_{3}^{u}})\vert_{\Zpq_{\rho}}\to0,\label{eq:Normal-bundle-sequence}\end{equation}
where $H_{\Zpq_{3}^{u}}=\Lpi_{G'}^{*}\sO_{\rG(2,T(-1))}(1)$ and $L_{\Zpq_{3}^{u}}=(\Lpi_{3}\circ\Lpi_{3^{u}})^{*}\sO(1)$.
The exceptional set (divisor) $E_{\rho}$ of the blow-up is given
by $\mP(\sN_{\Zpq_{\rho}/\Zpq_{3}}).$ Now, take a point $([\bar{U}],[V_{1}])\in\Prt_{\rho}$
with $[\bar{U}]=[(V/V_{2})\wedge(V_{2}/V_{1})]$ and $[V_{1}\subset V_{2}]\in F(1,2,V)$.
Then the fiber of $\Zpq_{\rho}\to\Prt_{\rho}$ over the point $([\bar{U}],[V_{1}])$
is the plane $\mP(\bar{U})$ in $\rG(2,V/V_{1})\subset\mP(\wedge^{2}(V/V_{1}))$.
Restricting the sequence (\ref{eq:Normal-bundle-sequence}) to $\mP(\bar{U})$,
we obtain \[
0\to\sN_{\Zpq_{\rho}/\Zpq_{3}}\vert_{\mP(\bar{U})}\to\ft{S}^{2}\bar{U}^{*}\otimes V_{1}^{*}\otimes\sO_{\mP(\bar{U})}\to\sO_{\mP(\bar{U})}(2)\otimes V_{1}^{*}\to0.\]
If we further restrict this to a point $x\in\mP(\bar{U})$, we see
that the stalk $(\sN_{\Zpq_{\rho}/\Zpq_{3}})_{x}$ at $x$ is given
by the kernel of the map $\ft{S}^{2}\bar{U}^{*}\to\sO_{\mP(\bar{U})}(2)_{x}$,
which we identify with the conics in $\mP(\bar{U})$ passing through
$x$. Now we recall that the fiber of $F_{\rho}\to\Prt_{\rho}$ over
$([\bar{U}],[V_{1}])$ is given by $\mP(\ft{S}^{2}\bar{U})$ which
parametrizes the conics in $\mP(\bar{U})$. Therefore, we see that
$E_{\rho}=\mP(\sN_{\Zpq_{\rho}/\Zpq_{3}}\vert_{\mP(\bar{U})})$ describes
the conics in $\mP(\bar{U})$ which correspond to each point of $F_{\rho}$,
i.e., the universal family of conics in the $\rho$-planes as claimed. 

The rest of the claims are clear since $\Zpq_{3}\to\hcoY_{3}$ is
a generically conic bundle over $\hcoY_{3}\setminus\Prt_{\rho}\sqcup\Prt_{\sigma}$.
\end{proof}

\begin{defn} \label{def:Z2t}For later use, we introduce 

\begin{myitem2} 

\item[\;\;(i)] $\mathrm{G}(2,4)$-bundle $B(2,4,\hcoY_{2}):=\mathrm{G}(2,T(-1))\times_{\mP(V)}\hcoY_{2}$, 

\item[\;\;(ii)] $\mP^{2}$-bundle $\Zpq_{2}^{u}:=\Zpq_{3}^{u}\times_{\hcoY_{3}}\hcoY_{2}=\mP(\Lrho_{2}^{*}\sS)$
over $\hcoY_{2}$ and 

\item[\;\;(iii)] $\Zpq_{2}^{t}:=B(2,4,\hcoY_{2})\cap\Zpq_{2}^{u}$,
the intersection over $\hcoY_{2}$.

\end{myitem2} 

\end{defn} 

Since $\Lpi_{3^{u}}:\Zpq_{3}^{u}=\mP(\sS)\to\hcoY_{3}$ is a flat
fibration, the morphism $\Zpq_{3}^{u}\times_{\hcoY_{3}}\hcoY_{2}\to\Zpq_{3}^{u}$
is the blow-up along $\Lpi_{3^{u}}^{-1}(\Prt_{\rho})$ with its exceptional
divisor $\Zpq_{3}^{u}\times_{\hcoY_{3}}F_{\rho}=\mP(\Lrho_{2}^{*}\sS|_{F_{\rho}})$.
The definition $\Zpq_{2}^{t}$ corresponds to the intersection $\Zpq_{3}=(\mathrm{G}(2,T(-1))\times_{\mP(V)}\hcoY_{3})\cap\Zpq_{3}^{u}$.
We note that $\Zpq_{2}^{t}$ is the total transform of $\Zpq_{3}$
under the blow-up $\Zpq_{2}^{u}\to\Zpq_{3}^{u}$ since it contains
the exceptional divisor. The superscript $t$ has been used to remind
this. $\Zpq_{2}^{t}$ is reduced since $\Zpq_{3}$ is smooth by Proposition
\ref{cla:G25}. $\Zpq_{2}^{t}$ will play important roles in our proof
of Proposition \ref{prop:est} (see Lemma \ref{cla:est}).

$\;$

We now fit the generically conic bundles in the following diagam:\begin{equation}
\begin{matrix}\xyZYmain\end{matrix}\label{eq:ZYmainDiagram}\end{equation}
Here we set $\hcoY_{2}^{o}=\hcoY_{2}\setminus\Prt_{\sigma},$ $\widetilde{\hcoY^{o}}=\widetilde{\hcoY}\setminus\Prt_{\sigma}$
and $\Zpq_{2}^{o}=\Zpq_{2}\setminus\Lpi_{2'}^{-1}(\Prt_{\sigma})$
and define $\iota_{2},\tilde{\iota}$ and $\iota_{2'}$ to be the
respective inclusions as in the above diagram. This diagram will be
used extensively in our construction of the ideal sheaf $\sI$ on
$\widetilde{\hcoY}\times\hchow$ in the next section. In the rest
of this section, we will prepare some preliminary results related
to this diagram.

\subsection{The Grassmann bundle $\rG(2,T(-1))\to\mP(V)$}

Fix a point $[V_{1}]\in\mP(V)$. Then the Euler sequence $0\to\sO(-1)\to V\otimes\sO\to T(-1)\to0$
over $[V_{1}]$ is represented by $0\to V_{1}\to V\to V/V_{1}\to0$.
We denote the projection by $\Lpi_{V_{1}}:V\to V/V_{1}$. We consider
the dual vector space $V^{*}$ to $V$ and identify the dual $(V/V_{1})^{*}$
in $V^{*}$ as $\{\varphi\in V^{*}\mid\varphi|_{V_{1}}=0\}$. Then
it is easy to deduce the following isomorphisms: \begin{equation}
\begin{matrix}\xyGiso\end{matrix}\qquad\qquad\label{eq:Giso}\end{equation}
 where $\overline{V}_{2}$ is a two dimensional subspace in $V/V_{1}$.

Note that Grassmann bundles $\mathrm{G}(2,T(-1))$ and $\mathrm{G}(2,\Omega(1))$,
respectively, are embedded into the projective bundles $\mP(T(-1)^{\wedge2})$
and $\mP(\Omega(1)^{\wedge2})$, which have the following relations:

\begin{lem} \label{cla:compare} There is a natural isomorphism $\mP(T(-1)^{\wedge2})\simeq\mP(\Omega(1)^{\wedge2})$,
and under this isomorphism we have the following identifications$:$
\[
\Hwomega=\Hwt+\Lwt\text{ and }\Lwomega=\Lwt.\]
 \end{lem}

\begin{proof} By the natural isomorphism $T(-1)^{\wedge2}\simeq\Omega(1)^{\wedge2}\otimes\Lwedge^{4}T(-1)\simeq\Omega(1)^{\wedge2}\otimes\sO(1)$
and our definitions of $H_{\mP(\sE)}$ and $L_{\Sigma}$ introduced
above (\ref{eq:univ}), the assertion follows. \end{proof}

Since we have $H_{\mP(\Omega(1)^{\wedge2})}\vert_{\rG(2,\Omega(1))}=\Lrho^{\;*}\sO_{\rG(2,V^{*})}(1)$
and also we can identify $\sO_{\mathrm{G}(2,V^{*})}(1)=\sO_{\mathrm{G}(3,V)}(1)$
in (\ref{eq:Giso})$,$ we have $H_{\mP(\Omega(1)^{\wedge2})}\vert_{\rG(2,T(-1))}$
$=\Lrho_{G}^{\;*}\sO_{\rG(3,V)}(1).$ Hence, using the above lemma,
we have \begin{equation}
N_{\mathrm{G}(2,T(-1))}=\left(\Hwt+\Lwt\right)\big|_{\mathrm{G}(2,T(-1))},\label{eq:pullback}\end{equation}
where $N_{\mathrm{G}(2,T(-1))}:=\Lrho_{G}^{\;*}\sO_{\mathrm{G}(3,V)}(1)$
following the convention in Subsection \ref{sub:Birat-Y}.

$\;$

\subsection{Canonical divisors $K_{\Zpq_{i}/\hcoY_{i}}$}

We calculate $N_{\Zpq_{3}}=(\Lrho_{G}\circ\Lpi_{G'})^{*}\sO_{G(3,V)}(1)$
and the relative canonical divisors $K_{\Zpq_{i}/\hcoY_{i}}:=K_{\Zpq_{i}}-\Lpi_{i'}^{\;*}K_{\hcoY_{i}}$
$(i=2,3)$ .

\begin{prop} \label{cla:cano} 

\begin{myitem2}

\item[\;$(1)$] $N_{\Zpq_{3}}=H_{\mP(\sS)}|_{\Zpq_{3}}+L_{\Zpq_{3}}$,

\item[\;$(2)$] $K_{\Zpq_{3}/\hcoY_{3}}=\Lpi_{\Zpq_{3}}^{\;*}(\det\sQ-L_{\hcoY_{3}})-N_{\Zpq_{3}},$

\item[\;$(3)$] $K_{\Zpq_{2}/\hcoY_{2}}=M_{\Zpq_{2}}-N_{\Zpq_{2}}.$

\end{myitem2}

\end{prop}

\begin{proof} (1) follows from (\ref{eq:pullback}) and Lemma \ref{cla:compare}
since $H_{\mP(\sS)}|_{\Zpq_{3}}$ is given by the restriction of $\Hwt$
by (\ref{eq:univ}). For (2), recall $\Zpq_{3}^{u}=\mP(\sS)$ and
the projection $\Lpi_{3^{u}}:\Zpq_{3}^{u}\to\hcoY_{3}.$ Then we have
$K_{\Zpq_{3}^{u}/\hcoY_{3}}=-3H_{\mP(\sS)}+\Lpi_{3^{u}}^{*}\det\sS^{*}$
from the Euler sequence for $\Zpq_{3}^{u}=\mP(\sS)$ \[
0\to\sO_{\Zpq_{3}^{u}}(-1)\to\Lpi_{3^{u}}^{*}\sS\to T_{\Zpq_{3}^{u}/\hcoY_{3}}(-1)\to0.\]
Note the inclusion $i:\mathrm{G}(2,T(-1))\hookrightarrow\mP(T(-1)^{\wedge2})$
and the definition $\Zpq_{3}=\Zpq_{3}^{u}\cap i(\mathrm{G}(2,T(-1)))$.
Then, by Proposition \ref{cla:G(2T)}, we have $\Zpq_{3}\in|2H_{\mP(\sS)}+L_{\Zpq_{3}^{u}}|$.
Hence, by the adjunction formula and Proposition \ref{prop:div-relations}
(2), we have \[
K_{\Zpq_{3}/\hcoY_{3}}=\{K_{\Zpq_{3}^{u}/\hcoY_{3}}+(2H_{\mP(\sS)}+L_{\Zpq_{3}^{u}})\}\big\vert_{\Zpq_{3}}=-H_{\mP(\sS)}|_{\Zpq_{3}}+L_{\Zpq_{3}}+\Lpi_{3'}^{\;*}\det\sS^{*}.\]
Now by (1) and Proposition \ref{prop:div-relations} (1), we obtain
(2). For (3), recall that $E_{\rho}=\Lpi_{2'}^{\;-1}(F_{\rho})$ is
the exceptional divisor of $\Lrho_{2'}\colon{\Zpq}_{2}\to{\Zpq}_{3}$.
Since the codimension of $\Prt_{\rho}$ in ${\hcoY}_{3}$ is three
and that of $\Lpi_{3'}^{-1}(\Prt_{\rho})$ in ${{\Zpq}_{3}}$ is two,
we have \[
K_{{\Zpq}_{2}/{\hcoY}_{2}}=\Lrho_{2'}^{\;*}K_{{\Zpq}_{3}/{\hcoY}_{3}}-E_{\rho}.\]
 Using (2), we obtain $K_{\Zpq_{2}/\hcoY_{2}}=\Lpi_{2'}^{\;*}(\Lrho_{2}^{\;*}\det\sQ-L_{\hcoY_{2}}-F_{\rho})-{N}_{\Zpq_{2}}.$
Now, by Proposition \ref{prop:div-relations} (5), the claimed relation
follows. \end{proof}

$\;$

\subsection{The subvariety $\Zpq_{3}^{u}$ in $\mP(T(-1)^{\wedge2})\times_{\mP(V)}\hcoY_{3}$}

The univeral family of planes $\Zpq_{3}^{u}$ is contained in $\mP(T(-1)^{\wedge2})\times_{\mP(V)}\hcoY_{3}$,
which is a $\mP^{5}$-bundle over $\hcoY_{3}$. In later sections,
we will be required to describe the ideal sheaf of $\Zpq_{3}^{u}$
in $\mP(T(-1)^{2})\times_{\mP(V)}\hcoY_{3}$ as a projective subbundle.
Here we recall the following standard lemma:

\begin{lem} \label{lem:KoszulZ} Let $X$ be a variety, and $0\to\sA\to\sB\to\sC\to0$
a short exact sequence of locally free $\sO_{X}$-modules. Associated
to the surjection $\sB\to\sC$, we may regards $\mP(\sC^{*})$ as
a subbundle of $\mP(\sB^{*})$. As such, the subvariety $\mP(\sC^{*})$
is the complete intersection with respect to a section of $\pi^{*}\sA^{*}\otimes\sO_{\mP(\sB^{*})}(1)$,
where $\pi$ is the natural projection $\mP(\sB^{*})\to X$. \end{lem}

\begin{proof} Let $\sI$ be the ideal sheaf of $\mP(\sC^{*})$ in
$\mP(\sB^{*})$. We have a natural exact sequence $0\to\sI\otimes\sO_{\mP(\sB^{*})}(1)\to\sO_{\mP(\sB^{*})}(1)\to\sO_{\mP(\sC^{*})}(1)\to0$.
Pushing forward this on $X$, we have $0\to\pi_{*}(\sI\otimes\sO_{\mP(\sB^{*})}(1))\to\sB\to\sC\to0$.
Therefore $\sA\simeq\pi_{*}(\sI\otimes\sO_{\mP(\sB^{*})}(1))$. Moreover,
by a natural map $\pi^{*}\pi_{*}(\sI\otimes\sO_{\mP(\sB^{*})}(1))\to\sI\otimes\sO_{\mP(\sB^{*})}(1)$,
we obtain $\pi^{*}\sA\to\sI\otimes\sO_{\mP(\sB^{*})}(1)$. Investigating
this along fibers, we see that this is surjective. \end{proof}

\begin{prop} \label{cla:KoszulZ} The subvariety $\Zpq_{3}^{u}$
in $\mP(T(-1)^{2})\times_{\mP(V)}\hcoY_{3}$ is the complete intersection
with respect to a section of $\sO_{\mP(T(-1)^{2})}(1)\boxtimes\sQ$,
hence $\sO_{\Zpq_{3}^{u}}$ has the following Koszul resolution as
a $\sO_{\mP(T(-1)^{2})\times_{\mP(V)}\hcoY_{3}}$-module$:$ \begin{eqnarray*}
0\to\Lwedge^{3}\{\sO_{\mP(T(-1)^{2})}(-1)\boxtimes\sQ^{*}\}\to\Lwedge^{2}\{\sO_{\mP(T(-1)^{2})}(-1)\boxtimes\sQ^{*}\}\to\\
\sO_{\mP(T(-1)^{2})}(-1)\boxtimes\sQ^{*}\to\sO_{\mP(T(-1)^{2})\times_{\mP(V)}\hcoY_{3}}\to\sO_{\Zpq_{3}^{u}}\to0.\end{eqnarray*}
 \end{prop}

\begin{proof} Since $\Zpq_{3}^{u}=\mP(\sS)$, the assertion follows
by applying Lemma \ref{lem:KoszulZ} to the dual of (\ref{eq:univ}).
\end{proof}

\newpage{}

\section{The ideal sheaf $\sI$ of a closed subscheme $\Delta$ on $\widetilde{\hcoY}\times\hchow$}

\label{section:Ker} We formulate a natural incidence correspondence
between $\widetilde{\hcoY}$ and $\hchow$ and consider its ideal
sheaf $\sI$ on $\widetilde{\hcoY}\times\hchow$. We obtain a locally
free resolution of the ideal sheaf, which will play a central role
for our proof of the derived equivalence between $X$ and $Y$ in
Section \ref{section:BC}. 

$\;$

\subsection{The ideal sheaf $\sI$ and its locally free resolution}

Let us consider an incident relation in $\rG(3,V)\times\rG(2,V)$
by \[
\Delta_{0}=\left\{ ([V_{3}],[V_{2}])\mid V_{2}\subset V_{3}\right\} ,\]
which coincides with the flag variety $F(2,3,V)$. We note that $\rG(3,V)$
is connected to $\Zpq_{2},\Zpq_{3}$ and $\rG(2,T(-1))$ by the morphisms
in the diagram (\ref{eq:ZYmainDiagram}). Considering the products
of these morphisms with $g:\hchow\to\rG(2,V)$, we denote the pull-backs
of $\Delta_{0}\subset\rG(3,V)\times\rG(2,V)$ in $\Zpq_{2}\times\hchow,$
$\Zpq_{3}\times\hchow$ and $\rG(2,T(-1))\times\hchow$, respectively,
by \[
\Delta_{2},\;\Delta_{3},\;\Delta_{G},\]
 and their ideal sheaves by $\sI_{\Delta_{2}},\sI_{\Delta_{3}}$ and
$\sI_{\Delta_{G}}$. Recall that we have introduced the following
universal sheaves:\[
0\to\eS\to V\otimes\sO_{\rG(3,V)}\to\eQ\to0,\;\;0\to\eF\to V\otimes\sO_{\rG(2,V)}\to\eG\to0.\]

\begin{prop} \label{cla:Delta} The ideal sheaf $\sI_{\Delta_{0}}$
of $\Delta_{0}$ has the following Koszul resolution$:$ \begin{equation}
0\to\Lwedge^{4}({\eQ}^{*}\boxtimes{\sF})\to\Lwedge^{3}({\eQ}^{*}\boxtimes{\sF})\to\Lwedge^{2}({\eQ}^{*}\boxtimes{\sF})\to{\eQ}^{*}\boxtimes{\sF}\to\sI_{\Delta_{0}}\to0.\label{eq:delta0}\end{equation}
 \end{prop} 

\begin{proof} Tensoring the two surjections $V\otimes\sO_{\mathrm{G}(3,V)}\to\eQ$
and $V^{*}\otimes\sO_{\mathrm{G}(2,V)}\to\sF^{*}$, we obtain a map
$(V\otimes V^{*})\otimes\sO_{\mathrm{G}(3,V)\times\mathrm{G}(2,V)}\to\eQ\boxtimes\sF^{*}$.
Associated to the identity element in $\Hom(V,V)\simeq V\otimes V^{*}$,
we obtain a map $\sO_{\mathrm{G}(3,V)\times\mathrm{G}(2,V)}\to\eQ\boxtimes\sF^{*}$.
We show that $\Delta_{0}$ is the scheme of zeros of the section associated
to this map. Indeed, at a point $([V_{3}],[V_{2}])$ of $\mathrm{G}(3,V)\times\mathrm{G}(2,V)$,
the fiber of $\eQ$ is $V/V_{3}$ and the fiber of $\sF^{*}$ is $V_{2}^{*}$.
Then it is easy to see that the identity in $V\otimes V^{*}$ is contained
in the kernel of the natural map $V\otimes V^{*}\to(V/V_{3})\otimes V_{2}^{*}$
if and only if $V_{2}\subset V_{3}$, namely, $([V_{3}],[V_{2}])\in\Delta_{0}$.
\end{proof}

Denote by $\mu_{2}$ the composition $\Lrho_{G}\circ\Lpi_{G'}\circ\Lrho_{2'}:\Zpq_{2}\to G(3,V)$.
Pulling back the above resolution, we obtain the following locally
free resolution of $\sI_{\Delta_{2}}$: \begin{equation}
\begin{aligned}0\to\Lwedge^{4}(\mu_{2}^{*}{\eQ}^{*}\boxtimes g^{*}\sF)\to & \Lwedge^{3}(\mu_{2}^{*}{\eQ}^{*}\boxtimes g^{*}\sF)\to\hfill\\
 & \Lwedge^{2}(\mu_{2}^{*}{\eQ}^{*}\boxtimes g^{*}\sF)\to\mu_{2}^{*}{\eQ}^{*}\boxtimes g^{*}\sF\to\sI_{\Delta_{2}}\to0.\end{aligned}
\label{eqnarray:I''}\end{equation}
We extend all morphisms in the diagram (\ref{eq:ZYmainDiagram}) by
taking the product $\times\id_{\hchow}$ to the corresponding morphisms
between the products with $\hchow$ and indicate these by $\check{\;}$,
e.g., $\check{\Lpi}_{3'}=\Lpi_{3'}\times\id_{\hchow}$. We consider
the restriction $\sI_{\Delta_{2}}^{o}=\check{\iota}_{2}^{*}\sI_{\Delta_{2}}$
of the ideal sheaf $\sI_{\Delta_{2}}$ to $\Zpq_{2}^{o}\times\hchow$.
In what follows, for simplicity, we abuse $\cLpi_{2'*}^{o}$ for $\cLpi_{2'*}^{o}\circ\check{\iota}_{2}^{*}$
when it is clear from the context. 

\begin{prop} \label{prop:step1} Define $\sI_{2}^{o}:=\check{\Lpi}_{2'*}^{o}\sI_{\Delta_{2}}^{o}$,
which is an ideal sheaf on ${\hcoY}_{2}^{o}\times{\hchow}$. There
is an exact sequence on ${\hcoY}_{2}^{o}\times{\hchow}:$ \begin{equation}
\begin{aligned}0\to R^{1}\check{\Lpi}_{2'*}^{o}\Lwedge^{4}(\mu_{2}^{*}{\eQ}^{*}\boxtimes g^{*}\sF)\to & R^{1}\check{\Lpi}_{2'*}^{o}\Lwedge^{3}(\mu_{2}^{*}{\eQ}^{*}\boxtimes g^{*}\sF)\to\hfill\\
 & R^{1}\check{\Lpi}_{2'*}^{o}\Lwedge^{2}(\mu_{2}^{*}{\eQ}^{*}\boxtimes g^{*}\sF)\to\sI_{2}^{o}\to0.\end{aligned}
\label{eqnarray:I'}\end{equation}
 \end{prop}

We prove this proposition in Subsection \ref{subsection:LocY2}. Our
strategy to obtain the ideal sheaf $\sI$ on $\widetilde{\hcoY}\times\hchow$
consists of the following four steps:

\begin{myitem} 

\item[S1.] We evaluate each term of (\ref{eqnarray:I'}) by Grothendieck-Verdier
duality (Subsection \ref{sub:Evaluating5.3}).

\item[S2.] We further simplify the sheaves which result from (S1)
to obtain a locally free resolution of $\sI_{2}^{o}$ (Subsection
\ref{sub:Simplify-I2}).

\item[S3.] We define the push forward $\sI^{o}:=(\tLrho_{2}\times\id)_{*}\sI_{2}^{o}$
and obtain its locally free resolution on $\widetilde{\hcoY}^{o}\times\hchow$
(Subsection \ref{subsection:LocY}). 

\item[S4.] Finally, we define the ideal sheaf $\sI=(\tilde{\iota}\times\id)_{*}\sI^{o}$
on $\widetilde{\hcoY}\times\hchow$ with its locally free resolution
(Subsection \ref{sub:Step-4}).

\end{myitem} Our final results can be summerized in the following
form:

\begin{thm} \label{thm:resolY} The ideal sheaf $\sI$ has the following
$\mathrm{SL}(V)$-equivariant locally free resolution$:$ \begin{equation}
\begin{matrix}0\to\widetilde{\sS}_{L}\boxtimes\sO_{\hchow}\to\widetilde{\sT}^{*}\boxtimes g^{*}\sF^{*}\to\Big(\sO_{\widetilde{\hcoY}}\boxtimes g^{*}\ft{S}^{2}\sF^{*}\Big)\oplus\left(\widetilde{\sQ}^{*}(M_{\widetilde{\hcoY}})\boxtimes\sO_{\hchow}(L_{\hchow})\right)\\
\qquad\;\;\qquad\qquad\qquad\qquad\to\;\;\sI\otimes\Big(\sO_{\widetilde{\hcoY}}(M_{\widetilde{\hcoY}})\boxtimes\sO_{\hchow}(2L_{\hchow})\Big)\;\;\to\;\;0\end{matrix}\label{eqn:fin5}\end{equation}
with the locally free sheaves $\widetilde{\sS}_{L}=(\widetilde{\sS}_{L}^{*})^{*}$,
$\widetilde{\sQ}$ and $\widetilde{\sT}$ on $\widetilde{\hcoY}$
introduced in \cite[Subsect.~6.1,~6.2]{HoTa3}. Let $\Delta$ be the
closed subscheme defined by $\sI$. Then $\Delta$ is an $\SL(V)$-invariant,
normal, and Cohen-Macaulay variety. \end{thm}

For convenience, we recall that the locally free sheaves $\widetilde{\sS}_{L}=(\widetilde{\sS}_{L}^{*})^{*}$,
$\widetilde{\sQ}$ and $\widetilde{\sT}$ are defined by the properties;
\[
\Lrho_{2}^{\;*}\sS(-L_{\hcoY_{2}})=\tLrho_{2}^{\;*}\widetilde{\sS}_{L},\;\Lrho_{2}^{\;*}\sQ=\tLrho_{2}^{\;*}\widetilde{\sQ},\;\sT_{2}=\tLrho_{2}^{\;*}\widetilde{\sT},\]
where $\sS$ and $\sQ$ are the universal sheaves on $\hcoY_{3}$
and $\sT_{2}$ is defined by the exact sequence \begin{equation}
0\to\sT_{2}^{*}\to\Lpi_{2}^{\;*}\Omega(1)\xrightarrow{}(\Lrho_{2}|_{F_{\rho}})^{*}\sO_{\mP(T(-1))}(1)\to0.\label{eq:T*}\end{equation}

$\;$

\subsection{Relations to the (dual) Lefschetz collections\label{subsection:State} }

Let us introduce the following ordered collection: \begin{equation}
\begin{aligned} & \;\;\,(\sE_{3},\sE_{2},\sE_{1a},\sE_{1b})=((\widetilde{\sS}_{L}^{*})^{*},\;\widetilde{\sT}^{*},\;\sO_{\widetilde{\hcoY}},\;\widetilde{\sQ}^{*}(M_{\widetilde{\hcoY}})),\\
 & (\mathcal{F}_{3},\mathcal{F}_{2},\mathcal{F}'_{1a},\mathcal{F}_{1b})=(\,\sO_{\hchow},\; g^{*}\sF^{*},\; g^{*}\ft{S}^{2}\sF^{*},\;\sO_{\hchow}(L_{\hchow})\,),\end{aligned}
\label{eqn:ordered-EF}\end{equation}
and write the locally free resolution (\ref{eqn:fin5}) as \[
0\to\sE_{3}\boxtimes\mathcal{F}_{3}\to\sE_{2}\boxtimes\mathcal{F}_{2}\to\sE_{1a}\boxtimes\mathcal{F}'_{1a}\oplus\sE_{1b}\boxtimes\mathcal{F}_{1b}\to\sI\otimes\Big(\sO_{\widetilde{\hcoY}}(M_{\widetilde{\hcoY}})\boxtimes\sO_{\hchow}(2L_{{\hchow}})\Big)\to0.\]
The ordered collection $(\mathcal{F}_{i})_{i\in I}=(\mathcal{F}_{3},\mathcal{F}_{2},\mathcal{F}{}_{1a},\mathcal{F}_{1b})$
for the dual Lefschetz collection of $\hchow$ (\cite[Thm.~3.4.5]{HoTa3})
and $(\sE_{i})_{i\in I}=(\sE_{3},\sE_{2},\sE_{1a},\sE_{1b})$ for
the Lefschetz collection of $\widetilde{\hcoY}$ ({[}\textit{ibid},~Thm.~8.1.1{]})
are defined from (\ref{eqn:ordered-EF}) with a slight modification
$\mathcal{F}_{1a}:=\mathcal{F}'_{1a}/\sO_{\hchow}(-H_{\hchow}+2L_{\hchow})$.
The modification of $\mathcal{F}'_{1a}$ will be explained in Proposition
\ref{cla:Vs}, where we will consider the ideal sheaf $\sI$ in the
(pull-back of the) universal family of hyperplane sections in $\mP(\ft{S}^{2}V^{*})\times\mP(\ft{S}^{2}V)$.
Here we explain the obvious duality in the following the quiver diagrams
(determined in \cite{HoTa3}) which represents the non-vanishing $\Hom$'s
among the ordered collections $(\sE_{i})_{i\in I}$ and $(\mathcal{F}_{i})_{i\in I}$:
\vspace{0.5cm}
 \begin{equation}
\begin{matrix}\xyquiverI\\[3cm]
\Hom_{\sO_{\widetilde{\hcoY}}}(\sE_{i},\sE_{j})\end{matrix}\hspace{1cm}\begin{matrix}\xyquiverII\\[3cm]
\Hom_{\sO_{\hchow}}(\mathcal{F}_{i},\mathcal{F}_{j})\end{matrix}\quad\qquad\label{eq:quivers}\end{equation}

$\;$

The duality in the above diagrams is related to the $\mathrm{SL}(V)$-equivariance
of the resolution (\ref{eqn:fin5}) (or corresponding resolution (\ref{eqnarray:onVs}))
as follows: Let us consider two pairs $(\sA_{1},\sA_{2})$ and $(\sB_{1},\sB_{2})$
of $\mathrm{SL}(V)$-equivariant, locally free sheaves on an $\mathrm{SL}(V)$-variety
in general. By tensoring the evaluation maps $\Hom(\sA_{1},\sA_{2})\otimes\sA_{1}\to\sA_{2}$
and $\Hom(\sB_{1},\sB_{2})\otimes\sB_{1}\to\sB_{2}$, we have \[
(\Hom(\sA_{1},\sA_{2})\otimes\Hom(\sB_{1},\sB_{2}))\otimes(\sA_{1}\otimes\sB_{1})\to\sA_{2}\otimes\sB_{2}.\]
 Suppose that we have the following duality as $\mathrm{SL}(V)$-modules:
\begin{equation}
\Hom(\sA_{1},\sA_{2})\simeq\Hom(\sB_{1},\sB_{2})^{*}.\label{eq:assump}\end{equation}
Then corresponding to the identity element in \[
\Hom(\sA_{1},\sA_{2})\otimes\Hom(\sB_{1},\sB_{2})\simeq\Hom(\Hom(\sB_{1},\sB_{2}),\Hom(\sB_{1},\sB_{2})),\]
 there is a unique $\mathrm{SL}(V)$-equivariant morphism \begin{equation}
\sA_{1}\otimes\sB_{1}\to\sA_{2}\otimes\sB_{2}.\label{eq:map}\end{equation}

Now we can see each morphism in the resolution from the above general
framework; (a) $\sE_{3}\boxtimes\mathcal{F}_{3}\to\sE_{2}\boxtimes\mathcal{F}_{2}$,
(b) $\sE_{2}\boxtimes\mathcal{F}_{2}\to\sE_{1a}\boxtimes\mathcal{F}'_{1a}$,
(c) $\sE_{2}\boxtimes\mathcal{F}_{2}\to\sE_{1b}\boxtimes\mathcal{F}_{1b}$,
(d) $\sE_{1a}\boxtimes\mathcal{F}'_{1a}\to\sO_{\widetilde{\hcoY}}(M_{\widetilde{\hcoY}})\boxtimes\sO_{\hchow}(2L_{\hchow})$
and (e) $\sE_{1b}\boxtimes\mathcal{F}_{1b}\to\sO_{\widetilde{\hcoY}}(M_{\widetilde{\hcoY}})\boxtimes\sO_{\hchow}(2L_{\hchow})$.
For (a) and (c), we can verify the relation (\ref{eq:assump}) directly
from the diagrams. For (b), we read the isomorphism $\Hom(\sE_{2},\sE_{1a})\simeq V$
from (\ref{eq:quivers}). To compute $\Hom(\mathcal{F}_{2},\mathcal{F}'_{1a})$,
we use the exact sequence $0\to\sO_{\hchow}(-H_{\hchow}+2L_{\hchow})\to\mathcal{F}'_{1a}\to\mathcal{F}_{1a}\to0$.
The vanishing $H^{\bullet}(\hchow,g^{*}\sF\otimes\sO_{\hchow}(-H_{\hchow}+2L_{\hchow}))=0$
follows from the fact that $g:\hchow\to{\rm G}(2,V)$ is a $\mP^{2}$-bundle.
Therefore we read $\Hom(\mathcal{F}_{2},\mathcal{F}'_{1a})\simeq\Hom(\mathcal{F}_{2},\mathcal{F}_{1a})\simeq V^{*}$
from (\ref{eq:quivers}). For (d) and (e), respectively, we use the
following calculations: \[
\begin{aligned} & \Hom(\sE_{1a},\sO_{\widetilde{\hcoY}}(M_{\widetilde{\hcoY}}))=H^{0}(\widetilde{\hcoY},\sO_{\widetilde{\hcoY}}(M_{\widetilde{\hcoY}}))\simeq\ft{S}^{2}V,\\
 & \Hom(\mathcal{F}'_{1a},\sO_{\hchow}(2L_{\hchow}))=H^{0}(\mathrm{G}(2,V),\ft{S}^{2}\eF\otimes\sO(2L_{\hchow}))\simeq\ft{S}^{2}V^{*},\end{aligned}
\]
 and \[
\begin{aligned} & \Hom(\sE_{1b},\sO_{\widetilde{\hcoY}}(M_{\widetilde{\hcoY}}))=H^{0}(\hcoY_{3},\sQ)=H^{0}(\mP(V),T(-1)^{2})\simeq\Lwedge^{2}V^{*},\\
 & \Hom(\mathcal{F}_{1b},\sO_{\hchow}(2L_{\hchow}))=H^{0}(\mathrm{G}(2,V),\sO_{\hchow}(L_{\hchow}))\simeq\Lwedge^{2}V.\end{aligned}
\]

$\;$

\subsection{Closed subschemes $\Delta_{3}\subset\Zpq_{3}\times\hchow$ and $\Delta_{2}\subset\Zpq_{2}\times\hchow$}

\label{Closed}

For a point $x\in\hchow$, we denote by $\Delta_{3,x}$ the fiber
of $\Delta_{3}\to\hchow$ over $x$. Let $l_{x}$ be the line in $\mP(V)$
corresponding to $g(x)\in\mathrm{G}(2,V)$.

\begin{prop} \label{cla:delta3} For any point $x\in\hchow$, the
variety $\Delta_{3,x}$ has a structure of a $\mathrm{G}(2,5)$-bundle
over the blow-up of $\mP(V)$ along $l_{x}$. \end{prop}

\begin{proof} Recall that $\Delta_{3}$ is the pull-back of $\Delta_{G}$
by $\check{\Lpi}_{G'}:\Zpq_{3}\times\hchow\to\rG(2,T(-1))\times\hchow$.
Let $\Delta_{G,x}$ be the fiber of ${\Delta}_{G}\to\hchow$ over
$x$. By Proposition \ref{cla:G25}, $\Delta_{3,x}$ is a $\mathrm{G}(2,5)$-bundle
over $\Delta_{G,x}$. We show that $\Delta_{G,x}$ is isomorphic to
the blow-up of $\mP(V)$ along $l_{x}$. We describe $\mathrm{G}(2,T(-1))$
as the universal family of planes on $\mP(V)$, namely, \[
\mathrm{G}(2,T(-1))=\{([V_{3}],[V_{1}])\mid V_{1}\subset V_{3}\}\subset\mathrm{G}(3,V)\times\mP(V).\]
 Then it holds that \[
\Delta_{G,x}=\{([V_{3}],[V_{1}])\mid V_{1}\subset V_{3},l_{x}\subset\mP(V_{3})\}\subset\mathrm{G}(2,T(-1)).\]
We see that the natural projection morphism from $\Delta_{G,x}$ to
$\mP(V)$ sending $([V_{3}],[V_{1}])$ to $[V_{1}]$ is the blow-up
of $\mP(V)$ along the line $l_{x}$. \end{proof}

$\;$ 

Define $\Delta_{\hcoY_{3}}:=\cLpi_{3'}(\Delta_{3})$ in $\hcoY_{3}\times\hchow$
with its reduced structure. 

\begin{prop} \label{cla:descrdel3} Let $[V_{1}]$ be a point of
$\mP(V)$. If $[V_{1}]\not\in l_{x}$, then the fiber of $\Delta_{\hcoY_{3},x}\to\mP(V)$
is isomorphic to $\mathrm{G}(2,5)$. If $[V_{1}]\in l_{x}$, then
the fiber of $\Delta_{\hcoY_{3},x}\to\mP(V)$ is isomorphic to the
$8$-dimensional Schubert cycle $\{[\bar{U}]\mid\mP(\bar{U})\cap\bar{P}_{x}\not=\emptyset\}\subset\mathrm{G}(3,\Lwedge^{2}V/V_{1})$,
where $\bar{P}_{x}$ is a fixed plane of $\mP(\Lwedge^{2}V/V_{1})$.
In particular, $\Delta_{3,x}\to\Delta_{\hcoY_{3},x}$ is and hence
$\Delta_{3}\to\Delta_{\hcoY_{3}}$ is birational. \end{prop}

\begin{proof} We can write \begin{equation}
\Delta_{3,x}=\left\{ ([\bar{a}\wedge\bar{b}],([\bar{U}],[V_{1}]))\in\Zpq_{3}\mid l_{x}\subset\mP(V_{3}),V_{3}=\langle a,b,v\rangle,V_{1}=\langle v\rangle\right\} ,\label{3x}\end{equation}
where $a\in V$ is determined by $\bar{a}=a\text{ mod }V_{1}$ and
similarly for $b$. Also, we have \begin{equation}
\Delta_{\hcoY_{3},x}=\left\{ ([\bar{U}],[V_{1}])\in\hcoY_{3}\mid\,^{\exists}V_{3}\text{ s.t. }\wedge^{2}(V_{3}/V_{1})\subset\bar{U}\text{ and }l_{x}\subset\mP(V_{3})\right\} .\label{3x2}\end{equation}
 If $[V_{1}]\not\in l_{x}$, then $V_{3}$ in (\ref{3x2}) is unique
by the condition $l_{x}\subset\mP(V_{3})$, i.e., $V_{3}=V_{1}\oplus\mC l_{x}$.
Then the three dimensional subspaces $\bar{U}$ satisfying $\wedge^{2}(V_{3}/V_{1})\subset\bar{U}\subset\wedge^{2}(V/V_{1})$
determines the fiber of $\Delta_{\hcoY_{3},x}\to\mP(V)$ over $[V_{1}]$,
which is isomorphic to $\mathrm{G}(2,5)$. In particular, identifying
$[\bar{a}\wedge\bar{b}]$ with $[\wedge^{2}(V_{3}/V_{1})]$, we see
that the morphism $\Delta_{3,x}\to\Delta_{\hcoY_{3},x}$ is birational,
and so is $\Delta_{3}\to\Delta_{\hcoY_{3}}$. Now, assume that $[V_{1}]\in l_{x}$.
Then $V_{3}$ satisfying $l_{x}\subset\mP(V_{3})$ form a $\rho$-plane
${\rm \mathrm{P}}_{V_{2,x}}=\left\{ [\Pi]\mid V_{2,x}\subset\Pi\right\} \subset\rG(3,V)$,
where $V_{2,x}$ is determined by $l_{x}=\mP(V_{2,x})$. Writing the
corresponding plane $\bar{P}_{x}=\left\{ [\wedge^{2}(\Pi/V_{1})]\mid[\Pi]\in\mathrm{P}_{V_{2,x}}\right\} \simeq\mP^{2}$
in $\mP(\wedge^{2}(V/V_{1}))$, we can rephrase the condition $^{\exists}V_{3}\text{ s.t. }\wedge^{2}(V_{3}/V_{1})\subset\bar{U}$
by $\bar{P}_{x}\cap\mP(\bar{U})\not=\emptyset$. Then the fiber of
$\Delta_{\hcoY_{3},x}\to\mP(V)$ over $[V_{1}]$ can be identified
with the Schubert cycle, $\{[\bar{U}]\mid\mP(\bar{U})\cap\bar{P}_{x}\not=\emptyset\}\subset\mathrm{G}(3,\Lwedge^{2}V/V_{1})$.
\end{proof}

$\;$

\begin{rem} The Schubert cycle $\{[\bar{U}]\mid\mP(\bar{U})\cap\bar{P}_{x}\not=\emptyset\}$
has a natural resolution of singularities; $\{([\wedge^{2}(V_{3}/V_{1})],[\bar{U}])\mid[\wedge^{2}(V_{3}/V_{1})]\in\mP(\bar{U}),l_{x}\subset\mP(V_{3})\}$,
which is contained in $\Delta_{3,x}$ and has a $\mathrm{G}(2,5)$-bundle
structure over $\{[V_{3}]\mid l_{x}\subset\mP(V_{3})\}\simeq\mP^{2}$.
\end{rem}

$\;$

\begin{prop} \label{cla:delta2} $\Delta_{2,x}$ is the blow-up of
$\Delta_{3,x}$ along $\Zpq_{\rho}\cap\Delta_{3,x}$. In particular,
$\Delta_{2,x}$ is and hence $\Delta_{2}$ is smooth. \end{prop} 

\begin{proof} Recall that $\Zpq_{2}\to\Zpq_{3}$ is the blow-up along
$\Zpq_{\rho}=\Lpi_{3'}^{-1}(\Prt_{\rho})\subset\Zpq_{3}$. By Proposition
\ref{cla:P1bdl}, we see that $\Zpq_{\rho}\cap\Delta_{3,x}$ is a
$\mP^{1}$-bundle over the smooth variety $\Delta_{G,x}\subset\rG(2,T(-1))$
(see Proposition \ref{cla:delta3}), and hence $\Zpq_{\rho}\cap\Delta_{3,x}$
is smooth. Therefore, $\Delta_{2,x}$ is smooth and so is $\Delta_{2}$.
\end{proof}

$\;$

\subsection{Locally free resolution (\ref{eqnarray:I'}) of $\sI_{2}^{o}$ on
$\hcoY_{2}^{o}\times{\hchow}$\label{subsection:LocY2}}

We start preparing the following two lemmas: 

\begin{lem} \label{cla:rhoplane} Let ${\rm {P}={\rm {P}_{V_{2}}}}$
be the $\rho$-plane in $\mathrm{G}(3,V)$ associated to some two
dimensional vector space $V_{2}$ in $V$ $($cf.~Subsection  $\ref{sub:Birat-Y})$.
Then $\eQ|_{{\rm {P}}}\simeq T_{{\rm {P}}}(-1)$. \end{lem} 

\begin{proof} From the natural surjection $V\otimes\sO_{\mathrm{G}(3,V)}\to\eQ$,
we obtain the surjection $V/V_{2}\otimes\sO_{{\rm {P}}}\to\eQ|_{{\rm {P}}}$.
Since ${\rm {P}\simeq\mP(V/V_{2})}$, this surjection can be identified
in the Euler sequence of ${\rm {P}}$. Therefore $\eQ|_{{\rm {P}}}\simeq T_{{\rm {P}}}(-1)$.
\end{proof}

\begin{lem} \label{cla:conic} Let $q$ be a $\tau$- or $\rho$-conic
on $\mathrm{G}(3,V)$. Then $H^{\bullet}({\eQ}^{*}|_{q})=0$. Moreover,
if $q$ is smooth, then ${\eQ}|_{q}\simeq\sO_{\mP^{1}}(1)^{\oplus2}$.
\end{lem}

\begin{proof} We denote by $\mP_{q}^{2}$ the plane spanned by $q$.
Let us first assume that $q$ is a $\tau$-conic. As we can see in
(\ref{eq:plane-to-conic}), there exists an $S\simeq\mathrm{G}(2,4)$
in $\rG(3,V)$ such that $q\subset S$. The conic $q$ is a complete
intersection in $S$ since $\mP_{q}^{2}\cap S=q$, and then $\sO_{q}$
has the following Koszul resolution as a $\sO_{S}$-module: \[
0\to\sO_{S}(-3)\to\sO_{S}(-2)^{\oplus3}\to\sO_{S}(-1)^{\oplus3}\to\sO_{S}\to\sO_{q}\to0.\]
 Tensoring this exact sequence with ${\eQ}^{*}|_{S}$ and using Theorem
\ref{thm:Bott}, it is easy to derive $H^{\bullet}({\eQ}^{*}|_{q})=0$. 

Now consider the case that $q$ is a $\rho$-conic. By Lemma \ref{cla:rhoplane},
it holds that $\eQ|_{\mP_{q}^{2}}\simeq T_{\mP_{q}^{2}}(-1)$. Tensoring
${\eQ}^{*}|_{\mP_{q}^{2}}$ with the exact sequence $0\to\sO_{\mP_{q}^{2}}(-2)\to\sO_{\mP_{q}^{2}}\to\sO_{q}\to0$,
we have \[
0\to\Omega_{\mP_{q}^{2}}^{1}(-1)\to\Omega_{\mP_{q}^{2}}^{1}(1)\to{\eQ}^{*}|_{q}\to0.\]
 By Theorem \ref{thm:Bott}, it holds that all the cohomology groups
of $\Omega_{\mP_{q}^{2}}^{1}(-1)$ and $\Omega_{\mP_{q}^{2}}^{1}(1)$
vanish. Thus we have $H^{\bullet}({\eQ}^{*}|_{q})=0$.

The remaining property follows from the vanishing property of the
cohomology. \end{proof}

The following argument to obtain (\ref{eqnarray:I'}) is inspired
by \cite[Lemma 8.2]{Ku2}.

$\;$

\noindent \textbf{Proof of (\ref{eqnarray:I'}).} Let us split (\ref{eqnarray:I''})
into the short exact sequences: \begin{equation}
0\to\Lwedge^{4}(\mu_{2}^{*}{\eQ}^{*}\boxtimes g^{*}\sF)\to\Lwedge^{3}(\mu_{2}^{*}{\eQ}^{*}\boxtimes g^{*}\sF)\to\sK_{1}\to0.\label{eq:sp1}\end{equation}
 \begin{equation}
0\to\sK_{1}\to\Lwedge^{2}(\mu_{2}^{*}{\eQ}^{*}\boxtimes g^{*}\sF)\to\sK_{2}\to0.\label{eq:sp2}\end{equation}
 \begin{equation}
0\to\sK_{2}\to\mu_{2}^{*}{\eQ}^{*}\boxtimes g^{*}\sF\to\sI_{\Delta_{2}}\to0.\label{eq:sp3}\end{equation}
 By (\ref{eq:sp3}) and Lemma \ref{cla:conic}, we have \begin{equation}
\cLpi_{2'*}^{o}(\mu_{2}^{*}{\eQ}^{*}\boxtimes g^{*}\sF)=R^{1}\cLpi_{2'*}^{o}(\mu_{2}^{*}{\eQ}^{*}\boxtimes g^{*}\sF)=0.\label{eq:R1}\end{equation}
 Hence we have $\cLpi_{2'*}^{o}\sI_{\Delta_{2}}=R^{1}\cLpi_{2'*}^{o}\sK_{2}$,
and $\cLpi_{2'*}^{o}\sK_{2}=0$. Thus by (\ref{eq:sp2}), \[
0\to R^{1}\cLpi_{2'*}^{o}\sK_{1}\to R^{1}\cLpi_{2'*}^{o}\Lwedge^{2}(\mu_{2}^{*}{\eQ}^{*}\boxtimes g^{*}\sF)\to R^{1}\cLpi_{2'*}^{o}\sK_{2}=\cLpi_{2'*}^{o}\sI_{\Delta_{2}}\to0.\]
 Since $\Lwedge^{2}(\mu_{2}^{*}{\eQ}^{*}\boxtimes g^{*}\sF)|_{q}\simeq\sO_{\mP^{1}}(-2)^{\oplus6}$
on a smooth fiber $q$ by Lemma \ref{cla:conic} and $\cLpi_{2'*}^{o}\Lwedge^{2}(\mu_{2}^{*}{\eQ}^{*}\boxtimes g^{*}\sF)$
is torsion free, we have $\cLpi_{2'*}^{o}\Lwedge^{2}(\mu_{2}^{*}{\eQ}^{*}\boxtimes g^{*}\sF)=0$.
This implies that $\cLpi_{2'*}^{o}\sK_{1}=0$ by (\ref{eq:sp2}).
Thus by (\ref{eq:sp1}), we have \[
0\to R^{1}\cLpi_{2'*}^{o}\Lwedge^{4}(\mu_{2}^{*}{\eQ}^{*}\boxtimes g^{*}\sF)\to R^{1}\cLpi_{2'*}^{o}\Lwedge^{3}(\mu_{2}^{*}{\eQ}^{*}\boxtimes g^{*}\sF)\to R^{1}\cLpi_{2'*}^{o}\sK_{1}\to0.\]
 Therefore we obtain the exact sequence (\ref{eqnarray:I'}). \hfill
$\square$

$\;$

Let $\Delta_{2}^{o}$ and $\Delta_{\hcoY_{2}}^{o}$ be the closed
subschemes of $\Zpq_{2}^{o}\times\hchow$ and $\hcoY_{2}^{o}\times\hchow$
defined by $\sI_{\Delta_{2}}^{o}$ and $\sI_{2}^{o}$ respectively.
We set $\cLpi_{\Delta_{2}^{o}}:=\cLpi_{2'}|_{\Delta_{2}^{o}}$. The
following lemma will be used in Subsection \ref{subsection:CM}.

\begin{lem} \label{cla:CM1} ${\cLpi}_{\Delta_{2}^{o}*}\sO_{\Delta_{2}^{o}}=\sO_{\Delta_{\hcoY_{2}}^{o}}$
and $R^{1}\cLpi_{\Delta_{2}^{o}*}\sO_{\Delta_{2}^{o}}=0$. \end{lem}

\begin{proof} By (\ref{eq:sp3}) and (\ref{eq:R1}), we have $R^{1}\cLpi_{2'*}^{o}\sI_{\Delta_{2}}^{o}=0$.
Taking the higher direct image of the exact sequence $0\to\sI_{\Delta_{2}}^{o}\to\sO_{\Zpq_{2}^{o}\times\hchow}\to\sO_{\Delta_{2}^{o}}\to0$,
we obtain the exact sequence $0\to\sI_{2}^{o}\to\sO_{\hcoY_{2}^{o}\times\hchow}\to\cLpi_{\Delta_{2}^{o}*}\sO_{\Delta_{2}^{o}}\to0$
and $R^{1}\cLpi_{\Delta_{2}^{o}*}\sO_{\Delta_{2}^{o}}=0$ since $R^{1}\cLpi_{2'*}^{o}\sI_{\Delta_{2}^{o}}=0$
and $R^{1}\cLpi_{2'*}^{o}\sO_{\Zpq_{2}^{o}\times\hchow}=0$. Hence
$\cLpi_{\Delta_{2}^{o}*}\sO_{\Delta_{2}^{o}}=\sO_{\Delta_{\hcoY_{2}}^{o}}$.
\end{proof}

$\;$

\subsection{Step 1: Evaluating (\ref{eqnarray:I'}) \label{sub:Evaluating5.3}}

In this step, we rewrite each term of the resolution (\ref{eqnarray:I'})
by the Grothendieck-Verdier duality. 

\noindent \begin{lem} \label{cla:conic2} Let $q$ be a $\tau$-
or $\rho$-conic on $\mathrm{G}(3,V)$. Then $H^{1}(q,\Lwedge^{i}{\eQ}^{\oplus2}\otimes\sO_{q}(-1))=0$
$(1\leq i\leq4)$. \end{lem}

\begin{proof} Assume that $q$ is a $\tau$-conic. As in the proof
of Lemma \ref{cla:conic}, we take $S\simeq\mathrm{G}(2,4)$ in $\rG(3,V)$
such that $q\subset S$ and consider the Koszul resolution of $\sO_{q}$
on $S$. Tensoring this exact sequence with $\Lwedge^{i}{\eQ}|_{S}^{\oplus2}\otimes\sO_{S}(-1)$,
we obtain \begin{eqnarray*}
0\to(\Lwedge^{i}{\eQ}|_{S}^{\oplus2})(-4)\to(\Lwedge^{i}{\eQ}|_{S}^{\oplus2})(-3)^{\oplus3}\to(\Lwedge^{i}{\eQ}|_{S}^{\oplus2})(-2)^{\oplus3}\to\\
\Lwedge^{i}{\eQ}|_{S}^{\oplus2}(-1)\to\Lwedge^{i}{\eQ}^{\oplus2}\otimes\sO_{q}(-1)\to0.\end{eqnarray*}
 Using Theorem \ref{thm:Bott}, it is easy to derive the assertion. 

When $q$ is a $\rho$-conic, as in the proof of Lemma \ref{cla:conic},
tensoring $(\Lwedge^{i}{\eQ}^{\oplus2}|_{\mP_{q}^{2}})(-1)\simeq(\Lwedge^{i}{T_{\mP_{q}^{2}}(-1)}^{\oplus2})(-1)$
with the exact sequence $0\to\sO_{\mP_{q}^{2}}(-2)\to\sO_{\mP_{q}^{2}}\to\sO_{q}\to0$,
we have \[
0\to(\Lwedge^{i}{T_{\mP_{q}^{2}}(-1)}^{\oplus2})(-3)\to(\Lwedge^{i}{T_{\mP_{q}^{2}}(-1)}^{\oplus2})(-1)\to\Lwedge^{i}{\eQ}^{\oplus2}\otimes\sO_{q}(-1)\to0.\]
 Computing all the cohomology groups of $(\Lwedge^{i}{T_{\mP_{q}^{2}}(-1)}^{\oplus2})(-3)$
and $(\Lwedge^{i}{T_{\mP_{q}^{2}}(-1)}^{\oplus2})(-1)$ by Theorem
\ref{thm:Bott}, we have the assertion. \end{proof}

$\;$

By Lemma \ref{cla:conic2}, we may apply the second part of the Grothendieck-Verdier
duality \ref{cla:duality} to the morphism $\cLpi_{2'}^{o}:{\Zpq}_{2}^{o}\times{\hchow}\to{\hcoY}_{2}^{o}\times{\hchow}$,
and then we have \[
R^{1}\cLpi_{2'*}^{o}\Lwedge^{i}(\mu_{2}^{*}{\eQ}^{*}\boxtimes g^{*}\sF)\simeq\Big(\cLpi_{2'*}^{o}\big\{\Lwedge^{i}(\mu_{2}^{*}{\eQ}\boxtimes g^{*}\sF^{*})\otimes\omega_{\Zpq_{2}\times{\hchow}/\hcoY_{2}\times{\hchow}}\big\}\Big)^{*}.\]
 Note that $\omega_{\Zpq_{2}\times{\hchow}/\hcoY_{2}\times{\hchow}}=\mathrm{pr}_{2}^{*}\omega_{\Zpq_{2}/\hcoY_{2}}=\omega_{\Zpq_{2}/\hcoY_{2}}\boxtimes\sO_{{\hchow}}$.
By Proposition \ref{cla:cano}, we have $K_{\Zpq_{2}/\hcoY_{2}}=M_{\Zpq_{2}}-N_{\Zpq_{2}}$.
Thus we have \begin{equation}
\begin{aligned} & R^{1}\cLpi_{2'*}^{o}\Lwedge^{i}(\mu_{2}^{*}{\eQ}^{*}\boxtimes g^{*}\sF)\simeq\\
 & \big(\cLpi_{2'*}^{o}(\Lwedge^{i}(\mu_{2}^{*}{\eQ}\boxtimes g^{*}\sF^{*})\otimes(\sO_{\Zpq_{2}^{o}}(-{N}_{\Zpq_{2}^{o}})\boxtimes\sO_{{\hchow}})\big)\otimes(\sO_{\hcoY_{2}^{o}}(M_{\hcoY_{2}^{o}})\boxtimes\sO_{{\hchow}})\big)^{*}.\end{aligned}
\label{eq:duality}\end{equation}
 To simplify this, we use the following formula (see \cite[Exercise 6.11]{FH}):
\begin{equation}
\Lwedge^{i}(\mu_{2}^{*}{\eQ}\boxtimes g^{*}\sF^{*})\simeq\bigoplus_{\lambda}\ft{\Sigma}^{\lambda}\mu_{2}^{*}{\eQ}\boxtimes\ft{\Sigma}^{\lambda'}g^{*}\sF^{*},\label{eqn:decompL}\end{equation}
 where $\lambda$ are partitions of $i$ with at most $2$ rows and
column, and $\lambda'$ is the dual partition to $\lambda$.

\vspace{0.2cm}
 \begin{prop} \label{prop:Io2} The exact sequence $(\ref{eqnarray:I'})$
twisted by $\sO_{{\hcoY}_{2}^{o}}(M_{{\hcoY}_{2}^{o}})\boxtimes\sO_{\hchow}(2L_{{\hchow}})$
is evaluated as \begin{equation}
\begin{aligned}0\to & \{\Lpi_{2'*}^{o}\sO_{\Zpq_{2}^{o}}(N_{\Zpq_{2}^{o}})\}^{*}\boxtimes\sO_{{\hchow}}\to\{\cLpi_{2'*}^{o}(\mu_{2}^{*}{\eQ})\}^{*}\boxtimes g^{*}\sF^{*}\to\\
 & \sO_{\hcoY_{2}^{o}}\boxtimes g^{*}\ft{S}^{2}\sF^{*}\oplus\{\Lpi_{2'*}^{o}(\{\mu_{2}^{*}\ft{S}^{2}{\eQ}\}(-{N}_{\Zpq_{2}^{o}}))\}^{*}\boxtimes\sO_{{\hchow}}(L_{\hchow})\to\\
 & \hspace{4cm}\sI_{2}^{o}\otimes\sO_{{\hcoY}_{2}^{o}}(M_{{\hcoY}_{2}^{o}})\boxtimes\sO_{\hchow}(2L_{{\hchow}})\to0\end{aligned}
\label{eqnarray:fin}\end{equation}
 \end{prop}

\begin{proof} By using (\ref{eqn:decompL}), we calculate $\Lwedge^{4}(\mu_{2}^{*}{\eQ}\boxtimes g^{*}\sF^{*})\simeq\ft{\Sigma}^{(2,2)}\mu_{2}^{*}\eQ\boxtimes\ft{\Sigma}^{(2,2)}g^{*}\eF^{*}=\sO_{\Zpq_{2}}(2N_{\Zpq_{2}})\boxtimes\sO_{\hchow}(2L_{{\hchow}})$.
Then we have \[
R^{1}\cLpi_{2'*}^{o}\Lwedge^{4}(\mu_{2}^{*}{\eQ}^{*}\boxtimes g^{*}\sF^{*})\simeq(\{\Lpi_{2'*}^{o}\sO_{\Zpq_{2}^{o}}(N_{\Zpq_{2}^{o}})\}^{*}\otimes\sO_{\hcoY_{2}^{o}}(-M_{\hcoY_{2}^{o}}))\boxtimes\sO_{\hchow}(-2L_{{\hchow}}).\]
 Similarly, we evaluate $\Lwedge^{3}(\mu_{2}^{*}{\eQ}\boxtimes g^{*}\sF^{*})\simeq\mu_{2}^{*}{\eQ}(N_{\Zpq_{2}})\boxtimes g^{*}\sF^{*}(L_{{\hchow}})$
with $\lambda=\lambda'=(2,1)$ and have \[
R^{1}\cLpi_{2'*}^{o}\Lwedge^{3}(\mu_{2}^{*}{\eQ}^{*}\boxtimes g^{*}\sF)\simeq\{\Lpi_{2'*}^{o}(\mu_{2}^{*}{\eQ})^{*}\otimes\sO_{\hcoY_{2}}(-M_{\hcoY_{2}})\}\boxtimes(g^{*}\sF^{*}(-2L_{{\hchow}})),\]
 where we use $\eF=\eF^{*}(-L_{\hchow}).$ Finally, we have \[
\begin{aligned}\Lwedge^{2}(\mu_{2}^{*}{\eQ}\boxtimes g^{*}\sF^{*}) & \simeq\Lwedge^{2}\mu_{2}^{*}{\eQ}\boxtimes\ft{S}^{2}(g^{*}\sF^{*})\oplus\ft{S}^{2}(\mu_{2}^{*}{\eQ})\boxtimes\Lwedge^{2}g^{*}\sF^{*}\\
 & \simeq\sO_{\Zpq_{2}}(N_{\Zpq_{2}})\boxtimes g^{*}\ft{S}^{2}\sF^{*}\oplus\mu_{2}^{*}\ft{S}^{2}{\eQ}\boxtimes\sO_{\hchow}(L_{{\hchow}}).\end{aligned}
\]
 Using this we evaluate $R^{1}\cLpi_{2'*}^{o}\Lwedge^{2}(\mu_{2}^{*}{\eQ}^{*}\boxtimes g^{*}\sF)$
as \[
\sO_{\hcoY_{2}^{o}}(-M_{\hcoY_{2}^{o}})\boxtimes g^{*}\ft{S}^{2}\sF\oplus\,\{\Lpi_{2'*}^{o}(\{\mu_{2}^{*}\ft{S}^{2}\eQ\}(-{N}_{\Zpq_{2}^{o}}))^{*}\otimes\sO_{\hcoY_{2}^{o}}(-M_{\hcoY_{2}^{o}})\}\boxtimes\sO_{\hchow}(-L_{{\hchow}}),\]
and use $\eF=\eF^{*}(-L_{\hchow})$ again to obtain (\ref{eqnarray:fin}).\end{proof}

$\;$

\subsection{Step 2: A locally free resolution of $\sI_{2}^{o}$ on $\hcoY_{2}^{o}\times\hchow$
\label{sub:Simplify-I2}}

We will characterize the following sheaves: \begin{equation}
\Lpi_{2'*}^{o}\sO_{\Zpq_{2}^{o}}(N_{\Zpq_{2}^{o}}),\;\;\Lpi_{2'*}^{o}(\mu_{2}^{*}{\eQ}),\;\;\Lpi_{2'*}^{o}(\{\mu_{2}^{*}\ft{S}^{2}\eQ\}(-{N}_{\Zpq_{2}^{o}})),\label{eq:sheaves}\end{equation}
 which have appeared in the resolution (\ref{eqnarray:fin}). Below
is a preliminary result.

\begin{lem} \label{cla:locfree} All the sheaves in $(\ref{eq:sheaves})$
are locally free on ${\hcoY}_{2}^{o}$. \end{lem}

\begin{proof} By the Grauert theorem, it suffices to check that the
dimensions of $H^{0}$-terms of the restrictions to fibers are constant
on ${\hcoY}_{2}^{o}$. Let $q$ be the fiber of $\Zpq_{2}^{o}\to\hcoY_{2}^{o}$
over a point of ${\hcoY}_{2}^{o}$. Let $[V_{1}]\in\mP(V)$ be the
image of $q$ by $\Zpq_{2}^{o}\to\mP(V)$. We may consider $q$ as
a conic on $\mathrm{G}(2,V/V_{1})$. Let $\mP_{q}^{2}$ be the plane
spanned by $q$.

For $\Lpi_{2'*}^{o}\sO_{\Zpq_{2}^{o}}(N_{\Zpq_{2}^{o}})$, we have
\[
H^{0}(q,N_{\Zpq_{2}^{o}}|_{q})\simeq H^{0}(q,\sO_{q}(1))\simeq H^{0}(\mP_{q}^{2},\sO_{\mP_{q}^{2}}(1))\simeq\mC^{3}.\]
For the sheaf $\Lpi_{2'*}^{o}(\mu_{2}^{*}{\eQ})$, the sheaf $\mu_{2}^{*}{\eQ}|_{q}$
is generated by global sections and its degree is two. Therefore,
by the Riemann-Roch theorem, $H^{0}(q,\mu_{2}^{*}{\eQ}|_{q})\simeq\mC^{4}.$
Finally, for the sheaf $\Lpi_{2'*}^{o}(\{\mu_{2}^{*}\ft{S}^{2}\eQ\}(-{N}_{\Zpq_{2}^{o}}))$,
we can show \[
H^{0}(q,(\{\mu_{2}^{*}\ft{S}^{2}\eQ\}(-{N}_{\Zpq_{2}^{o}}))|_{q})\simeq\mC^{3}\]
 by similar computations to those in the proof of Lemma \ref{cla:conic2}.
Here we present the calculations only for the case where $q$ is a
$\rho$-conic: Tensoring $(\{\mu_{2}^{*}\ft{S}^{2}\eQ\}(-{N}_{\Zpq_{2}^{o}}))|_{\mP_{q}^{2}}\simeq(\ft{S}^{2}T_{\mP_{q}^{2}}(-1))(-1)$
with the exact sequence $0\to\sO_{\mP_{q}^{2}}(-2)\to\sO_{\mP_{q}^{2}}\to\sO_{q}\to0$,
we have \[
0\to(\ft{S}^{2}T_{\mP_{q}^{2}}(-1))(-3)\to(\ft{S}^{2}T_{\mP_{q}^{2}}(-1))(-1)\to(\ft{S}^{2}T_{\mP_{q}^{2}}(-1))(-1)|_{q}\to0.\]
 We compute the cohomology groups by Theorem \ref{thm:Bott}. It turns
out that all the cohomology groups of $(\ft{S}^{2}T_{\mP_{q}^{2}}(-1))(-1)$
vanish, and the only nonvanishing cohomology group of $(\ft{S}^{2}T_{\mP_{q}^{2}}(-1))(-3)$
is $H^{1}$ with \begin{equation}
H^{1}(\mP_{q}^{2},(\ft{S}^{2}T_{\mP_{q}^{2}}(-1))(-3))\simeq\bar{U}^{*}\otimes\Lwedge^{3}\bar{U}^{*},\label{eq:H1}\end{equation}
 where $\bar{U}$ is the three dimensional subspace of $\wedge^{2}(V/V_{1})$
such that $\mP_{q}^{2}=\mP(\bar{U})$. Consequently, we have \[
H^{0}(q,(\{\mu_{2}^{*}\ft{S}^{2}\eQ\}(-{N}_{\Zpq_{2}^{o}}))|_{q})\simeq H^{1}(\mP_{q}^{2},(\ft{S}^{2}T_{\mP_{q}^{2}}(-1))(-3))\simeq\mC^{3}.\]
 \end{proof}

\subsubsection{Part 1: }

Here we consider sheaves on $\Zpq_{2}^{t}$ (see Definition \ref{def:Z2t})
which have similar forms to each of the sheaves in (\ref{eq:sheaves}).
Recall that we have defined $\Zpq_{2}^{t}=B(2,4,\hcoY_{2})\cap\mP(\Lrho_{2}^{*}\sS$)
in Definition \ref{def:Z2t} with the Grassmann bundle $B(2,4,\hcoY_{2})$.
An advantage of working on $\Zpq_{2}^{t}$ is that its structure sheaf
$\sO_{\Zpq_{2}^{t}}$ has a nice Koszul resolution as $\sO_{B(2,4)}$-module,
where and hereafter we abbreviate $B(2,4,\hcoY_{2})$ in subscripts
to $B(2,4)$.

By Proposition \ref{cla:KoszulZ} and $\Zpq_{2}^{u}:=\mP(\Lrho_{2}^{*}\sS$),
the variety $\Zpq_{2}^{u}$ is a complete intersection in $\mP(T(-1)^{2})\times_{\mP(V)}\hcoY_{2}$
with respect to a section of $\sO_{\mP(T(-1)^{2})}(1)\boxtimes\Lrho_{2}^{*}\sQ$.
Therefore $\Zpq_{2}^{t}=B(2,4,\hcoY_{2})\cap\Zpq_{2}^{u}$ is the
complete intersection in $B(2,4,\hcoY_{2})$ by a section of $(\sO_{\mP(T(-1)^{2})}(1)\boxtimes\Lrho_{2}^{*}\sQ)|_{B(2,4)}$.
By (\ref{eq:pullback}) and Proposition \ref{cla:compare}, we have
an isomorphism \[
(\sO_{\mP(T(-1)^{2})}(1)\boxtimes\Lrho_{2}^{*}\sQ)|_{B(2,4)}\simeq\sO_{\mathrm{G}(2,T(-1))}({N}_{\mathrm{G}(2,T(-1))}-L_{\mathrm{G}(2,T(-1))})\boxtimes\Lrho_{2}^{*}\sQ.\]
 From Proposition \ref{cla:KoszulZ}, we see that the sheaf $\sO_{\Zpq_{2}^{t}}$
has the following Koszul resolution as a $\sO_{B(2,4)}$-module: \begin{equation}
0\to\sA_{3}\to\sA_{2}\to\sA_{1}\to\sO_{B(2,4)}\to\sO_{\Zpq_{2}^{t}}\to0,\label{eqnarray:KoszulZ2}\end{equation}
 where we set \[
\sA_{i}:=\sO_{\mathrm{G}(2,T(-1))}(-i{N}_{\mathrm{G}(2,T(-1))}+iL_{\mathrm{G}(2,T(-1))})\boxtimes\Lwedge^{i}\Lrho_{2}^{*}\sQ^{*}\text{ for }\ i=0,1,2,3.\]

Let $\Lpi_{2^{t}}\colon\Zpq_{2}^{t}\to\hcoY_{2}$ and $\Lrho_{2^{t}}\colon\Zpq_{2}^{t}\to\mathrm{G}(3,V)$
be the natural morphisms, and $N_{\Zpq_{2}^{t}}$ the pull-back of
$N_{\mathrm{G}(2,T(-1))}$. Then, using the Koszul resolution (\ref{eqnarray:KoszulZ2}),
we have

\begin{lem} \label{cla:est} \begin{myitem2}

\item[\rm{(i)}]$\Lpi_{2^{t}*}\sO_{\Zpq_{2}^{t}}(N_{\Zpq_{2}^{t}})\simeq\Lrho_{2}^{*}\sS^{*}(L_{\hcoY_{2}})$, 

\item[\rm{(ii)}]$\Lpi_{2^{t}*}(\Lrho_{2^{t}}^{*}{\eQ})\simeq\Lpi_{2}^{*}T(-1)$,
and 

\item[\rm{(iii)}]$\Lpi_{2^{t}*}(\{\Lrho_{2^{t}}^{*}\ft{S}^{2}{\eQ}\}(-N_{\Zpq_{2}^{t}}))\simeq\Lrho_{2}^{*}\sQ(-M_{\hcoY_{2}}-F_{\rho}).$ 

\end{myitem2}

\end{lem}

\begin{proof} We show (i)--(iii) using Theorem \ref{thm:Bott} and
noting that $\pr_{2}\colon B(2,4,\hcoY_{2})\to\hcoY_{2}$ is a $\mathrm{G}(2,4)$-bundle.
Let $\Gamma$ be a fiber of $\pr_{2}$.

\noindent(i) We tensor (\ref{eqnarray:KoszulZ2}) with $\sO_{B(2,4)}(N_{B(2,4)})$,
which is the pull-back to $B(2,4,\hcoY_{2})$ of $\sO_{\mP(\Omega(1)^{2})}(1)|_{\mathrm{G}(2,T(-1))}$
by (\ref{eq:pullback}) and Lemma \ref{cla:compare}. It is easy to
see that\[
\begin{aligned} & R^{\bullet}\pr_{2*}(\sO_{B(2,4)}({N}_{B(2,4)})\otimes\sA_{1})=0\text{ for }\bullet>0,\\
 & R^{\bullet}\pr_{2*}(\sO_{B(2,4)}({N}_{B(2,4)})\otimes\sA_{i})=0\text{ for }\bullet\geq0\text{ for }i=2,3,\end{aligned}
\]
since $(\sO_{B(2,4)}({N}_{B(2,4)})\otimes\sA_{i})|_{\Gamma}\simeq\sO_{\Gamma}(-i+1)$.
Moreover, we have \[
\pr_{2*}(\sO_{B(2,4)}({N}_{B(2,4)})\otimes\sA_{1})=\Lrho_{2}^{*}\sQ^{*}(L_{\hcoY_{2}}),\;\;\pr_{2*}\sO_{B(2,4)}({N}_{B(2,4)})=\Lpi_{2}^{*}T(-1)^{2}.\]
 Therefore we obtain the short exact sequence \[
0\to\Lrho_{2}^{*}\sQ^{*}(L_{\hcoY_{2}})\to\Lpi_{2}^{*}T(-1)^{2}\to\Lpi_{2^{t}*}\sO_{\Zpq_{2}^{t}}(N_{\Zpq_{2}^{t}})\to0,\]
 which coincides with the pull-back of the dual of the universal exact
sequence (\ref{eq:univ}) twisted by $L_{\hcoY_{2}}$ by the proof
of Lemma \ref{lem:KoszulZ}. Thus $\Lpi_{2^{t}*}\sO_{\Zpq_{2}^{t}}(N_{\Zpq_{2}^{t}})\simeq\Lrho_{2}^{*}\sS(L_{\hcoY_{2}})$
as claimed.

\noindent(ii) We tensor (\ref{eqnarray:KoszulZ2}) with the pull
back $\eQ_{B(2,4)}$ on $B(2,4,\hcoY_{2})$ of $\eQ$. We see that
$R^{\bullet}\pr_{2*}(\eQ_{B(2,4)}\otimes\sA_{i})=0$ for $\bullet\geq0$
and $i=1,2,3$ by Theorem \ref{thm:Bott} since $(\eQ_{B(2,4)}\otimes\sA_{i})|_{\Gamma}\simeq\eQ_{\Gamma}(-i)$,
where $\eQ_{\Gamma}$ is the universal quotient bundle of rank $2$
on $\Gamma\simeq\mathrm{G}(2,4)$. Moreover, $\pr_{2*}\eQ_{B(2,4)}\simeq\Lpi_{2}^{*}T(-1)$.
Therefore we have $\Lpi_{2^{t}*}((\Lrho_{2^{t}}^{*}{\eQ})\simeq\Lpi_{2}^{*}T(-1)$
as claimed.

\noindent(iii) We tensor (\ref{eqnarray:KoszulZ2}) with $\ft{S}^{2}\eQ_{B(2,4)}(-{N}_{B(2,4)})$.
We see that \[
\begin{aligned} & R^{\bullet}\pr_{2*}(\ft{S}^{2}\eQ_{B(2,4)}(-{N}_{B(2,4)})\otimes\sA_{i})=0\text{ for }\bullet\geq0\text{ and }i=0,1,3,\\
 & R^{\bullet}\pr_{2*}(\ft{S}^{2}\eQ_{B(2,4)}(-{N}_{B(2,4)})\otimes\sA_{2})=0\text{ for }\bullet\not=2\end{aligned}
\]
 by Theorem \ref{thm:Bott} since $(\ft{S}^{2}\eQ_{B(2,4)}(-{N}_{B(2,4)})\otimes\sA_{i})|_{\Gamma}\simeq\ft{S}^{2}\eQ_{\Gamma}(-i-1)$.
Moreover, \[
R^{2}\pr_{2*}(\ft{S}^{2}\eQ_{B(2,4)}(-{N}_{B(2,4)})\otimes\sA_{2})\simeq\{\Lwedge^{2}\Lrho_{2}^{*}\sQ^{*}\}(L_{\hcoY_{2}})\simeq\Lrho_{2}^{*}(\sQ\otimes\det\sQ^{*})(L_{\hcoY_{2}}).\]
 By Proposition \ref{prop:div-relations} (5), we also have \[
\Lrho_{2}^{*}(\sQ\otimes\det\sQ^{*})(L_{\hcoY_{2}})\simeq\Lrho_{2}^{*}\sQ(-M_{\hcoY_{2}}-F_{\rho}).\]
 Therefore we have \[
\Lpi_{2^{t}*}(\{\Lrho_{2^{t}}^{*}\ft{S}^{2}\eQ\}(-{N}_{\Zpq_{2}^{t}}))\simeq R^{2}\pr_{2*}(\ft{S}^{2}\eQ_{B(2,4)}(-{N}_{B(2,4)})\otimes\sA_{2}),\]
 which is isomorphic to $\Lrho_{2}^{*}\sQ(-M_{\hcoY_{2}}-F_{\rho})$
as claimed. \end{proof}

\begin{prop} \label{prop:injections}There exist the following injective
morphisms:

\begin{myitem2}

\item[\rm{(i)}] $\Lpi_{2^{t}*}\sO_{\Zpq_{2}^{t}}(N_{\Zpq_{2}^{t}})\hookrightarrow\Lpi_{2'*}\sO_{\Zpq_{2}}(N_{\Zpq_{2}}),$

\item[\rm{(ii)}] $\Lpi_{2^{t}*}(\Lrho_{2^{t}}^{*}\eQ)\hookrightarrow\Lpi_{2'*}(\mu_{2}^{*}\eQ)$,

\item[\rm{(iii)}] $\Lpi_{2^{t}*}(\left\{ \Lrho_{2^{t}}^{*}\ft{S}^{2}\eQ\right\} (-N_{\Zpq_{2}^{t}})\hookrightarrow\Lpi_{2'*}(\left\{ \mu_{2}^{*}\ft{S}^{2}\eQ\right\} (-N_{\Zpq_{2}})$.

\end{myitem2}

\end{prop}

\begin{proof} Since $\Zpq_{2}^{t}$ is the total transform of the
blow-up $\Zpq_{3}\times_{\hcoY_{3}}\hcoY_{2}\to\Zpq_{3}$ along $\Lpi_{3'}^{-1}(\Prt_{\rho})$
and $\Zpq_{2}$ is the strict transform of $\Zpq_{3}$, we have a
natural morphism $\Lpi_{2^{t}*}(\sB\vert_{\Zpq_{2}^{t}})$ $\to\Lpi_{2'*}(\sB\vert_{\Zpq_{2}})$
for the sheaves on $\hcoY_{2}$ associated to any sheaf $\sB$ on
$\Zpq_{2}^{u}=\mP(\Lrho_{2}^{*}\sS$). We note that the morphism is
injective if the sheaf $\Lpi_{2^{t}*}(\sB\vert_{\Zpq_{2}^{t}})$ is
locally free. Lemma \ref{cla:est} indicates that this is the case
for sheaves $\sB$ such that $\sB\vert_{\Zpq_{2}^{t}}=\Lpi_{2^{t}*}\sO_{\Zpq_{2}^{t}}(N_{\Zpq_{2}^{t}}),\Lpi_{2^{t}*}(\Lrho_{2^{t}}^{*}\eQ)$
and $\Lpi_{2^{t}*}\big(\Lrho_{2^{t}}^{*}\ft{S}^{2}\eQ(-N_{\Zpq_{2}^{t}})\big)$.
\end{proof}

$\;$

\subsubsection{Part 2:}

Here we determine the sheaves in (\ref{eq:sheaves}) and complete
our construction of the locally free resolution of $\sI_{2}^{o}\otimes\{\sO_{\hcoY_{2}^{o}}(M_{\hcoY_{2}^{o}})\boxtimes\sO_{\hchow}(2L_{{\hchow}})\}$
as follows: \begin{equation}
\begin{aligned}0\;\to\; & \iota_{2}^{*}\Lrho_{2}^{*}\sS(-L_{\hcoY_{2}^{o}})\boxtimes\sO_{{\hchow}}\;\to\;\iota_{2}^{*}\sT_{2}^{*}\boxtimes g^{*}\sF^{*}\;\to\\
 & \qquad\sO_{\hcoY_{2}^{o}}\boxtimes g^{*}\ft{S}^{2}\sF^{*}\oplus\iota_{2}^{*}\Lrho_{2}^{*}\sQ^{*}(M_{{\hcoY}_{2}^{o}})\boxtimes\sO_{\hchow}(L_{\hchow})\;\to\\
 & \qquad\qquad\qquad\sI_{2}^{o}\otimes\{\sO_{\hcoY_{2}^{o}}(M_{\hcoY_{2}^{o}})\boxtimes\sO_{\hchow}(2L_{{\hchow}})\}\;\to\;0.\end{aligned}
\label{eqnarray:fin2}\end{equation}

\noindent \begin{prop} \label{prop:est} It holds that \begin{myitem2}

\item[\rm{(1)}]$\Lpi_{2*}\sO_{\Zpq_{2}}(N_{\Zpq_{2}})\simeq\Lrho_{2}^{*}\sS^{*}(L_{\hcoY_{2}})$. 

\item[\rm{(2)}]$\Lpi_{2*}(\mu_{2}^{*}{\eQ})\simeq\sT_{2}$ $($cf.~$(\ref{eq:T*}))$,
and 

\item[\rm{(3)}]$\Lpi_{2*}(\{\mu_{2}^{*}\ft{S}^{2}\eQ\}(-{N}_{\Zpq_{2}}))\simeq\Lrho_{2}^{*}\sQ(-M_{\hcoY_{2}}).$ 

\end{myitem2}\end{prop} 

\begin{proof} From Lemma \ref{prop:injections}, we have a natural
injection $\Lpi_{\Zpq_{2}^{t}*}(\sB|_{\Zpq_{2}^{t}})\hookrightarrow\Lpi_{2'*}(\sB|_{\Zpq_{2}})$
for the sheaves $\sB$ described there. Note that the injection is
isomorphic outside $F_{\rho}$. Let $y$ be a point of $F_{\rho}$
and $q$ the fiber of $\Zpq_{2}\to\hcoY_{2}$ over $y$. Let $[V_{1}]$
be the image of $y$ on $\mP(V)$. We write $\mP_{q}^{2}=\mP(\bar{U})$,
where the plane $\mP_{q}^{2}$ in $\mathrm{G}(2,V/V_{1})$ is spanned
by the conic $q$, and $\bar{U}$ is a three-dimensional subspace
of $\Lwedge^{2}(V/V_{1})$. Since $q$ is a $\rho$-conic, there exists
a $2$-dimensional subspace $V_{2}$ such that $V_{1}\subset V_{2}$
and $\bar{U}=V/V_{2}\otimes V_{2}/V_{1}\simeq V/V_{2}$.

To see (1), we note the injection in Proposition \ref{prop:injections}
(i), \[
\Lrho_{2}^{*}\sS^{*}(L_{\hcoY_{2}})\simeq\Lpi_{2^{t}*}\sO_{\Zpq_{2}^{t}}(N_{\Zpq_{2}^{t}})\hookrightarrow\Lpi_{2'*}\sO_{\Zpq_{2}}(N_{\Zpq_{2}}).\]
 with the first isomorphism in Lemma \ref{cla:est} (i). This injection
must be an isomorphism since both the fibers of $\Lrho_{2}^{*}\sS(L_{\hcoY_{2}})$
and $\Lpi_{2'*}\sO_{\Zpq_{2}}(N_{\Zpq_{2}})$ at $y$ are isomorphic
to $H^{0}(\mP_{q}^{2},\sO_{\mP_{q}^{2}}(1))$.

We show (2). In the proof of Lemma \ref{cla:est} (ii), we have shown
that $\pr_{2*}(\Lpi_{B(2,4)}^{*}T(-1))\simeq\Lpi_{2^{t}*}(\Lrho_{2^{t}}^{*}{\eQ})$,
where $\Lpi_{B(2,4)}=\Lpi_{2}\circ\pr_{2}\colon B(2,4,\hcoY_{2})\to\mP(V)$
is the natural morphism. We now compute the map \begin{equation}
\pr_{2*}(\Lpi_{B(2,4)}T(-1))\otimes k(y)\simeq\Lpi_{2^{t}*}(\Lrho_{2^{t}}^{*}{\eQ})\otimes k(y)\to\Lpi_{2'*}(\mu_{2}^{*}{\eQ})\otimes k(y).\label{eq:(2)}\end{equation}
 Since $\pr_{2*}$ and $\Lpi_{2}$ are flat near $y$, we can compute
$\pr_{2*}(\Lpi_{B(2,4)}^{*}T(-1))\otimes k(y)\simeq H^{0}(\mathrm{G}(2,V/V_{1}),V/V_{1}\otimes\sO_{\mathrm{G}(2,V/V_{1})})\simeq V/V_{1}$
and $\Lpi_{2*}(\mu_{2}^{*}{\eQ})\otimes k(y)\simeq H^{0}(q,\mu_{2}^{*}{\eQ}|_{q})$
by the Grauert theorem. Then, by Lemma \ref{cla:rhoplane}, the map
$V/V_{1}\to H^{0}(q,\mu_{2}^{*}{\eQ}|_{q})$ factors through $H^{0}({\mP_{q}^{2}},T_{\mP_{q}^{2}}(-1))\simeq V/V_{2}$.
Therefore the cokernel of the dual of the map (\ref{eq:(2)}) is isomorphic
to $(V_{2}/V_{1})^{*}$, which can be identified with the fiber of
$\Lrho_{2}\vert_{F_{\rho}}^{*}\sO_{\mP(T(-1))}(1)$ at $y$ (where
$\Lrho_{2}\vert_{F_{\rho}}:F_{\rho}\to\Prt_{\rho}\simeq\mP(T(-1))$
). Hence it holds that $\{\Lpi_{2'*}(\mu_{2}^{*}{\eQ})\}^{*}\subset\sT_{2}^{*}$
by (\ref{eq:T*}). Since both the kernel of \[
(\{\Lpi_{2'*}(\mu_{2}^{*}{\eQ})\}^{*})\otimes k(y)\to(\{\Lpi_{2^{t}*}(\Lrho_{2^{t}}^{*}{\eQ})\}^{*})\otimes k(y)\]
 and that of \[
(\sT_{2}^{*})\otimes k(y)\to(\{\Lpi_{2^{t}*}(\Lrho_{2^{t}}^{*}{\eQ})\}^{*})\otimes k(y)\]
 are one-dimensional, the images of them must coincide. Combined with
the property $\{\Lpi_{2'*}(\mu_{2}^{*}{\eQ})\}^{*}\subset\sT_{2}^{*}$,
we obtain $\{\Lpi_{2'*}(\mu_{2}^{*}{\eQ})\}^{*}=\sT_{2}^{*}$ as claimed. 

Finally we show (3). Let $E_{\rho}^{u}:=\mP(\Lrho_{2}^{*}\sS|_{F_{\rho}})$
which is the restrictions of $\Zpq_{2}^{u}$ over ${F_{\rho}}$, and
recall our definition $E_{\rho}=\Lpi_{2'}^{-1}(F_{\rho})$ in Subsection
\ref{Z2Y2}. Note that $E_{\rho}^{u}$ is the exceptional divisor
of the blow-up $\Zpq_{2}^{u}\to\Zpq_{3}^{u}$ discussed after Definition
\ref{def:Z2t}. We summarize the geometry of blow-up and morphisms
in the following diagram: \begin{equation}
\begin{matrix}\xydiagIIIp\end{matrix}\label{eq:diag3'}\end{equation}
 Let $\eQ_{E_{\rho}^{u}}$ be the pull-back of $\eQ_{\rho}$ (introduced
in Proposition \ref{cla:WZp}) by $E_{\rho}^{u}\to\Zpq_{\rho}^{u}=\Zpq_{\rho}$.
Then, by the proof of Lemma \ref{cla:locfree}, we have \begin{equation}
\Lpi_{2'*}\big(\{\mu_{2}^{*}\ft{S}^{2}\eQ\}(-{N}_{\Zpq_{2}})|_{E_{\rho}}\big)\simeq R^{1}\Lpi_{2^{u}*}\big(\ft{S}^{2}\eQ_{E_{\rho}^{u}}(-N_{\Zpq_{2}^{u}}\vert_{E_{\rho}^{u}}-\Zpq_{2}|_{E_{\rho}^{u}})\big).\label{eq:finaleq}\end{equation}
In Lemma \ref{cla:calc} below, we rewrite the r.h.s.~of (\ref{eq:finaleq})
in terms of the pull-back $\sR_{E_{\rho}^{u}}$ of $\sR_{\rho}$ (introduced
in (\ref{eq:EulerS})) to $E_{\rho}^{u}$ and some divisors on $F_{\rho}$.
The first factor $R^{1}\Lpi_{2^{u}*}(\ft{S}^{2}\sR_{E_{\rho}^{u}}(-3\det\sR_{E_{\rho}^{u}}))$
in (\ref{eq:onemore}) can be evaluated by the Bott theorem \ref{thm:Bott}
applied to the projective bundle $E_{\rho}^{u}\to F_{\rho}$, and
turns out to be isomorphic to \[
\Lrho_{F_{\rho}}^{\;*}(\sS^{*}\otimes\det\sS^{*}|_{\Prt_{\rho}})\simeq\Lrho_{F_{\rho}}^{\;*}(\sQ\otimes\det\sQ|_{\Prt_{\rho}})(-4L_{F_{\rho}}),\]
where we use Proposition \ref{prop:div-relations} (1) and (3). Therefore,
we finally obtain$\Lrho_{F_{\rho}}^{\;*}\sQ(-M_{\hcoY_{2}}|_{F_{\rho}})$
for the r.h.s. of (\ref{eq:finaleq}). \end{proof}

\begin{lem} \label{cla:calc} Let $\sR_{E_{\rho}^{u}}$ be the pull-back
of the sheaf $\sR_{\rho}=T_{\Zpq_{\rho}/\Prt_{\rho}}\otimes\sO_{\Zpq_{\rho}}(-1)$
defined in $(\ref{eq:EulerS})$. The r.h.s. of $(\ref{eq:finaleq})$
is isomorphic to \begin{equation}
R^{1}\Lpi_{2^{u}*}(\ft{S}^{2}\sR_{E_{\rho}^{u}}(-3\det\sR_{E_{\rho}^{u}}))\otimes\sO_{F_{\rho}}(-\Lrho_{F_{\rho}}^{\;*}\det\sQ|_{\Prt_{\rho}}+4L_{F_{\rho}}-M_{\hcoY_{2}}|_{F_{\rho}}).\label{eq:onemore}\end{equation}
 \end{lem} 

\begin{proof} Let us first note the following relations which follows
from Lemma \ref{cla:WZp}, \[
\ft{S}^{2}\eQ_{\rho}\simeq\ft{S}^{2}\sR_{\rho}\otimes\Lpi_{\rho}^{*}\sO_{\mP(T(-1))}(2)\simeq\ft{S}^{2}\sR_{\rho}(\Lpi_{\rho}^{*}(\det\sQ|_{\Prt_{\rho}})-2L_{\Zpq_{\rho}}),\]
where for the the second isomorphism we use Proposition \ref{prop:div-relations}
(4) and the equality $H_{\Prt_{\rho}}=\sO_{\mP(T(-1))}(1)$. Then,
pulling this back to $E_{\rho}^{u}$, we obtain \begin{equation}
\ft{S}^{2}\eQ_{E_{\rho}^{u}}\simeq\ft{S}^{2}\sR_{E_{\rho}^{u}}(\Lpi_{2^{u}\vert_{\rho}}^{*}\Lrho_{F_{\rho}}^{\;*}(\det\sQ|_{\Prt_{\rho}})-2L_{E_{\rho}^{u}}).\label{eq:pull2}\end{equation}
Second, we take the determinants of (\ref{eq:EulerS}) to have \begin{equation}
H_{\mP(\sS|_{\Prt_{\rho}})}=\det\sR_{\rho}+\Lpi_{\rho}^{*}(\det\sS^{*}|_{\Prt_{\rho}})=\det\sR_{\rho}+\Lpi_{\rho}^{*}(\det\sQ|_{\Prt_{\rho}})-3L_{\Zpq_{\rho}^{u}},\label{eq:detEulerS}\end{equation}
 where Proposition \ref{prop:div-relations} (1) is used for the second
equality. Therefore, using Proposition \ref{cla:cano} (1), we obtain
\[
{N}_{\Zpq_{3}^{u}}|_{{\Zpq}_{\rho}^{u}}=\det\sR_{\rho}+\Lpi_{3^{u}\vert_{\rho}}^{*}(\det\sQ|_{\Prt_{\rho}})-2L_{\Zpq_{\rho}^{u}},\]
 and then pulling this back to $E_{\rho}^{u}$, we have \begin{equation}
{N}_{\Zpq_{2}^{u}}|_{E_{\rho}^{u}}=\det\sR_{E_{\rho}^{u}}+\Lpi_{2^{u}\vert_{\rho}}^{*}(\Lrho_{F_{\rho}}^{\;*}\det\sQ|_{\Prt_{\rho}})-2L_{E_{\rho}^{u}}.\label{eq:pull1}\end{equation}
Now, from (\ref{eq:pull2}) and (\ref{eq:pull1}), we have \begin{equation}
\ft{S}^{2}\eQ_{E_{\rho}^{u}}(-N_{\Zpq_{2}^{u}}|_{E_{\rho}^{u}})\simeq\ft{S}^{2}\sR_{E_{\rho}^{u}}(-\det\sR_{E_{\rho}^{u}}).\label{eq:pull3}\end{equation}

Let us compute the class of the divisor $\Zpq_{2}|_{E_{\rho}^{u}}(=E_{\rho})$
in $E_{\rho}^{u}$. We see that $\Zpq_{2}\in|2H_{\mP(\Lrho_{2}^{*}\sS)}+L_{\Zpq_{2}^{u}}-\Lpi_{2^{u}}^{*}F_{\rho}|$
since $\Zpq_{2}$ is the strict transform of $\Zpq_{3}$ by the blow-up
$\Zpq_{2}^{u}\to\Zpq_{3}^{u}\simeq\mP(\sS)$ and $\Zpq_{3}\in|2H_{\mP(\sS)}+L_{\Zpq_{3}^{u}}|$
by the proof of Proposition \ref{cla:cano}. Using Proposition \ref{prop:div-relations}
(5), we now have \[
2H_{\mP(\Lrho_{2}^{*}\sS)}+L_{\Zpq_{2}^{u}}-\Lpi_{2^{u}}^{*}F_{\rho}=2H_{\mP(\Lrho_{2}^{*}\sS)}+2L_{\Zpq_{2}^{u}}-\Lpi_{2^{u}}^{*}(\Lrho_{F_{\rho}}^{\;*}\det\sQ|_{\Prt_{\rho}}-M_{\hcoY_{2}}|_{F_{\rho}}).\]
 Therefore, by (\ref{eq:detEulerS}), we obtain \begin{equation}
\Zpq_{2}\vert_{E_{\rho}^{u}}\in|2\det\sR_{E_{\rho}^{u}}+\Lpi_{2^{u}}^{*}(\Lrho_{F_{\rho}}^{\;*}\det\sQ|_{\Prt_{\rho}})-4L_{E_{\rho}^{u}}+\Lpi_{2^{u}}^{*}(M_{\hcoY_{2}}|_{F_{\rho}})|.\label{eq:pull4}\end{equation}
 From (\ref{eq:pull3}) and (\ref{eq:pull4}), we obtain the claimed
form for $(\ref{eq:finaleq})$. \end{proof}

$\;$

\subsection{Step 3: A locally free resolution of $\sI^{o}$ on $\widetilde{\hcoY}^{o}\times{\hchow}$}

\label{subsection:LocY}~

Set $\tLrho\,'_{2}:=\tLrho_{2}^{o}\times{\mathrm{id}}$. We calculate
the pushforward $\tLrho\,'_{2*}$ of the exact sequence (\ref{eqnarray:fin2}).
To do this, we split (\ref{eqnarray:fin2}) as follows: \begin{equation}
0\to\iota_{2}^{*}\Lrho_{2}^{*}\sS(-L_{\hcoY_{2}^{o}})\boxtimes\sO_{{\hchow}}\to\iota_{2}^{*}\sT_{2}^{*}\boxtimes g^{*}\sF^{*}\to\sC\to0\label{eq:sp1'}\end{equation}
 and \begin{equation}
\begin{aligned}0\;\to\;\sC\;\to\;\sO_{\hcoY_{2}^{o}}\boxtimes & g^{*}\ft{S}^{2}\sF^{*}\oplus\iota_{2}^{*}\Lrho_{2}^{*}\sQ^{*}(M_{\hcoY_{2}^{o}})\boxtimes\sO_{\hchow}(L_{\hchow})\;\to\\
 & \;\;\quad\sI_{2}^{o}\otimes\{\sO_{\hcoY_{2}^{o}}(M_{\hcoY_{2}^{o}})\boxtimes\sO_{\hchow}(2L_{{\hchow}})\}\;\to\;0.\end{aligned}
\label{eq:sp2'}\end{equation}
 Since $\Lrho_{2}^{*}\sS(-L_{\hcoY_{2}})$, $\sT_{2}^{*}$ and $\Lrho_{2}^{*}\sQ^{*}(M_{\hcoY_{2}})$
are the pull-backs of locally free sheaves $\widetilde{\sS}_{L}$,
$\widetilde{\sT}^{*}$, and $\widetilde{\sQ}^{*}(M_{{\widetilde{\hcoY}}})$
on $\widetilde{\hcoY}$, the higher direct images of $\iota_{2}^{*}\Lrho_{2}^{*}\sS(-L_{\hcoY_{2}})\boxtimes\sO_{\hchow}$
and $\iota_{2}^{*}\sT_{2}^{*}\boxtimes g^{*}\sF$ vanish. Therefore
the pushforward of (\ref{eq:sp1'}) is still exact and the higher
direct images of $\sC$ vanish. Then the pushforward of (\ref{eq:sp1'})
is also exact. Therefore we obtain the following exact sequence on
$\widetilde{\hcoY}^{o}\times\hchow$: \begin{equation}
\begin{aligned}0\to\tilde{\iota}^{*}\widetilde{\sS}_{L}\boxtimes\sO_{{\hchow}}\to\tilde{\iota}^{*}\widetilde{\sT}^{*}\boxtimes g^{*}\sF^{*}\to & \sO_{\widetilde{\hcoY}^{o}}\boxtimes g^{*}\ft{S}^{2}\sF^{*}\oplus\tilde{\iota}^{*}\widetilde{\sQ}^{*}(M_{\widetilde{\hcoY}^{o}})\boxtimes\sO_{\hchow}(L_{\hchow})\\
 & \to\sI^{o}\otimes\{\sO_{\widetilde{\hcoY}^{o}}(M_{\widetilde{\hcoY}^{o}})\boxtimes\sO_{\hchow}(2L_{{\hchow}})\}\to0,\end{aligned}
\label{eqnarray:fin3}\end{equation}
 where $\sI^{o}:=\tLrho\,'_{2*}\sI_{2}^{o}$. 

\begin{lem} \label{cla:CM2} Let $\Delta^{o}$ be the closed subscheme
of $\widetilde{\hcoY}^{o}\times\hchow$ defined by $\sI_{2}^{o}$.
Define $\tLrho_{\Delta_{\hcoY_{2}}^{o}}:=\tLrho\,'_{2}|_{\Delta_{\hcoY_{2}}^{o}}$,
then $\tLrho_{\Delta_{\hcoY_{2}}^{o}*}\sO_{\Delta_{\hcoY_{2}}^{o}}=\sO_{\Delta^{o}}$
and $R^{1}\tLrho_{\Delta_{\hcoY_{2}}^{o}*}\sO_{{\Delta_{2}}^{o}}=0$.
\end{lem}

\begin{proof} By (\ref{eq:sp1'}) and (\ref{eq:sp2'}), we have $R^{1}\tLrho\,'_{2*}{\sI_{2}}^{o}=0$
since $\Lrho_{2}^{*}\sS(-L_{\hcoY_{2}})$, $\sT^{*}$ and $\Lrho_{2}^{*}\sQ^{*}(M_{{\hcoY_{2}}})$
are the pull-backs of locally free sheaves on $\widetilde{\hcoY}$.
Taking the higher direct image of the exact sequence $0\to{\sI_{2}}^{o}\to\sO_{\hcoY_{2}^{o}\times\hchow}\to\sO_{\Delta_{\hcoY_{2}}^{o}}\to0$,
we obtain the exact sequence $0\to{\sI}^{o}\to\sO_{\widetilde{\hcoY}^{o}\times\hchow}\to\tLrho_{\Delta_{\hcoY_{2}}^{o}*}\sO_{\Delta_{\hcoY_{2}}^{o}}\to0$
and also $R^{1}\tLrho_{\Delta_{\hcoY_{2}}^{o}*}\sO_{\Delta_{\hcoY_{2}}^{o}}=0$
since $R^{1}\tLrho\,'_{2*}\sI_{2}^{o}=0$ and $R^{1}\tLrho\,'_{2*}\sO_{\hcoY_{2}^{o}\times\hchow}=0$.
Hence $\tLrho_{\Delta_{\hcoY_{2}}^{o}*}\sO_{\Delta_{\hcoY_{2}}^{o}}=\sO_{\Delta^{o}}$.

\end{proof}

$\;$

\subsection{Step 4: Ideal sheaf $\sI$ and its resolution (\ref{eqn:fin5}) \label{sub:Step-4}}

Let $\sI:=\tilde{\iota}_{*}\sI^{o}$ and set $\Gamma_{\widetilde{\hcoY}}:=\widetilde{\hcoY}\times{\hchow}\setminus\widetilde{\hcoY}^{o}\times{\hchow}$.
We note that $\codim\Gamma_{\widetilde{\hcoY}}=6$ since the codimension
of $\sP_{\sigma}$ in $\widetilde{\hcoY}$ is $6$. Define a sheaf
$\sA$ on $\widetilde{\hcoY}^{o}$ by \[
0\to\tilde{\iota}^{*}\widetilde{\sS}_{L}\boxtimes\sO_{{\hchow}}\to\tilde{\iota}^{*}\widetilde{\sT}^{*}\boxtimes g^{*}\sF^{*}\to\sA\to0,\]
then, we have the following exact sequence \begin{equation}
\begin{aligned} & 0\to\widetilde{\sS}_{L}\boxtimes\sO_{{\hchow}}\to\widetilde{\sT}^{*}\boxtimes g^{*}\sF^{*}\to\tilde{\iota}_{*}\sA\to R^{1}\tilde{\iota}_{*}(\tilde{\iota}^{*}\widetilde{\sS}_{L}\boxtimes\sO_{{\hchow}})\\
 & \to R^{1}\tilde{\iota}_{*}(\tilde{\iota}^{*}\widetilde{\sT}^{*}\boxtimes g^{*}\sF^{*})\to R^{1}\tilde{\iota}_{*}\sA\to R^{2}\tilde{\iota}_{*}(\tilde{\iota}^{*}\widetilde{\sS}_{L}\boxtimes\sO_{{\hchow}}),\end{aligned}
\label{eqnarray:split1}\end{equation}
 and also from (\ref{eqnarray:fin3}), \begin{equation}
\begin{aligned}0\to\tilde{\iota}_{*}\sA\to & \sO_{\widetilde{\hcoY}}\boxtimes g^{*}\ft{S}^{2}\sF^{*}\oplus\widetilde{\sQ}^{*}(M_{\widetilde{\hcoY}})\boxtimes\sO_{\hchow}(L_{\hchow})\to\\
 & \sI\otimes\{\sO_{\widetilde{\hcoY}}(M_{\widetilde{\hcoY}})\boxtimes\sO_{\hchow}(2L_{{\hchow}})\}\to R^{1}\tilde{\iota}_{*}\sA,\end{aligned}
\label{eqnarray:split2}\end{equation}
 where we note that $\tilde{\iota}_{*}(\sE|_{{\widetilde{\hcoY}}^{o}\times{\hchow}})=\sE$
for a locally free sheaf $\sE$ on ${\widetilde{\hcoY}}\times{\hchow}$
since $\codim\Gamma_{\widetilde{\hcoY}}\geq2$, and $\iota_{\tilde{\iota}_{*}}(\sI^{o}\otimes\{\sO_{\widetilde{\hcoY}^{o}}(M_{\widetilde{\hcoY}^{o}})\boxtimes\sO_{\hchow}(2L_{{\hchow}})\})=\sI\otimes\{\sO_{\widetilde{\hcoY}}(M_{\widetilde{\hcoY}})\boxtimes\sO_{\hchow}(2L_{{\hchow}})\}$
by definition. For a sheaf $\sE$ on ${\widetilde{\hcoY}}\times{\hchow}$,
it holds that $R^{i}\tilde{\iota}_{*}(\sE|_{{\widetilde{\hcoY}^{o}}\times{\hchow}})=\sH_{\Gamma_{\widetilde{\hcoY}}}^{i+1}(\sE)$
for $i>0$ by \cite[p.9, Corollary 1.9]{H}. Moreover, if $\sE$ is
locally free, then ${\sH}_{\Gamma_{\widetilde{\hcoY}}}^{i+1}(\sE)=0$
for $i+1<4$ by \cite[p.44, Theorem 3.8]{H} since $\codim\Gamma_{\widetilde{\hcoY}}=6$.
Therefore, \[
R^{1}\tilde{\iota}_{*}(\tilde{\iota}^{*}\widetilde{\sS}_{L}\boxtimes\sO_{\hchow})\simeq R^{1}\tilde{\iota}_{*}(\tilde{\iota}^{*}\widetilde{\sT}^{*}\boxtimes g^{*}\sF^{*})\simeq R^{2}\tilde{\iota}_{*}(\tilde{\iota}^{*}\widetilde{\sS}_{L}\boxtimes\sO_{{\hchow}})=0.\]
 From (\ref{eqnarray:split1}), we see that $R^{1}\tilde{\iota}_{*}\sA=0$.
Consequently, we obtain the claim that the sequence (\ref{eqn:fin5})
in Theorem \ref{thm:resolY} is exact on $\widetilde{\hcoY}\times\hchow$
by (\ref{eqnarray:split1}) and (\ref{eqnarray:split2}). \hfill
$\square$

$\;$

\subsection{$\Delta$ is normal and Cohen-Macaulay}

\label{subsection:CM}

Let $\Delta\subset\widetilde{\hcoY}\times\hchow$ be the variety determined
by the ideal sheaf $\sI$. We show that $\Delta$ is normal and is
Cohen-Macaulay to complete our proof of Theorem \ref{thm:resolY}.

We see that $\Delta_{2}$ is smooth by Proposition \ref{cla:delta2},
and also $\Delta_{2}\to\Delta_{\hcoY_{2}}$ is birational since $\Delta_{3}\to\Delta_{\hcoY_{3}}$
is birational by Proposition \ref{cla:descrdel3} and so are $\Delta_{2}\to\Delta_{3}$
and $\Delta_{\hcoY_{2}}\to\Delta_{\hcoY_{3}}.$ Recall that we have
set $\cLpi_{\Delta_{2}^{o}}:=\cLpi_{2'}|_{\Delta_{2}^{o}}$ and $\tLrho_{\Delta_{\hcoY_{2}}^{o}}:=\tLrho\,'_{2}|_{\Delta_{\hcoY_{2}}^{o}}$.
We check the claimed properties above separately on $\widetilde{\hcoY}^{o}\times\hchow$
and $\Prt_{\sigma}\times\hchow$.

Let us first consider the properties on $\widetilde{\hcoY}^{o}\times\hchow$.
We show that \begin{equation}
\text{\ensuremath{R^{i}(\tLrho_{\Delta_{\hcoY_{2}}^{o}}\circ\cLpi_{\Delta_{2}^{o}})_{*}\sO_{\Delta_{2}^{o}}=0}\; for \ensuremath{i>0}\;\text{ and }\;(\ensuremath{\tLrho_{\Delta_{\hcoY_{2}}^{o}}\circ\cLpi_{\Delta_{2}^{o}})_{*}\sO_{\Delta_{2}^{o}}\simeq\sO_{\Delta^{o}}}.}\label{eq:normality}\end{equation}
 The latter shows that $\Delta^{o}=\tilde{\iota}^{*}\Delta$ is normal
since $\Delta_{2}^{o}$ is smooth, and the former show that $\Delta^{o}$
has only rational singularities, and then is Cohen-Macaulay. By using
the Leray spectral sequence, the relations in (\ref{eq:normality})
follow from Lemma \ref{cla:CM1} for $\cLpi_{\Delta_{2}^{o}}$ and
Lemma \ref{cla:CM2} for $\tLrho_{\Delta_{\hcoY_{2}}^{o}}$.

Now we consider the properties on $\Prt_{\sigma}\times\hchow$. By
the result above, we see that $\Delta$ is regular in codimension
one since the codimension of $\Delta\cap(\Prt_{\sigma}\times\hchow)$
is greater than two. Therefore it suffices to show that $\Delta$
is Cohen-Macaulay at any point of $\Delta\cap(\Prt_{\sigma}\times\hchow)$.
This follows from taking the local cohomology sequence of the locally
free resolution (\ref{eqn:fin5}) since $\widetilde{\hcoY}\times\hchow$
is smooth and the length of the locally free part of (\ref{eqn:fin5})
is three.

\begin{rem} Since we have shown $\Delta$ is reduced, $\Delta$ is
the closure of $\Delta^{o}$. \end{rem}

$\;$\vspace{1cm}

\section{The universal family of hyperplane sections}

\label{subsection:Univ} Let $\Vs\subset\widetilde{\hcoY}\times{\hchow}$
be the pull-back of the universal family of hyperplane sections in
$\mP({\ft S}^{2}V^{*})\times\mP({\ft S}^{2}V)$. We can regard $\Vs$
as the family of the pull-backs of hyperplanes in $\chow$ parameterized
by $\widetilde{\hcoY}$, and also as the family of the pull-backs
of hyperplanes in $\Hes$ parameterized by ${\hchow}$. We simply
say that $\Vs$ is the family of hyperplane sections of $\widetilde{\hcoY}$
and $\hchow$. Note that the fiber of $\Vs$ over a point $x\in X\subset\hchow$
is the pull-back of the hyperplane section $w_{\bm{x}\bm{y}}^{\perp}\cap\Hes$
of $\Hes$, where $x=w_{\bm{x}\bm{y}}$ as a point of ${\ft S}^{2}\mP(V)$.

In Proposition \ref{cla:Vs}, we show that $\Delta$ is a closed subscheme
of $\Vs$. Moreover, we give a locally free resolution of the ideal
sheaf of $\Delta$ in $\Vs$ as an $\sO_{\widetilde{\hcoY}\times\hchow}$-module.
It should be noted that the locally free sheaves in this resolution
are exactly those used to construct the (dual) Lefschetz collections
in $\sD^{b}(\widetilde{\hcoY})$ and $\sD^{b}(\hchow)$ in \cite{HoTa3}.
In Proposition \ref{cla:hyp}, we show that any hyperplane section
of $\widetilde{\hcoY}$ corresponding to a point of $\hchow$ has
only canonical singularities. Although technical, this is important
to apply the Kawamata-Viehweg vanishing theorem in our proof of the
derived equivalence (Lemma \ref{cla:vanishing}).

\vspace{0.3cm}
 Let us begin with a preliminary discussion. The ideal sheaf $\sI_{\Vs}$
of $\Vs$ on $\widetilde{\hcoY}\times{\hchow}$ is isomorphic to $\sO_{\widetilde{\hcoY}}(-M_{\widetilde{\hcoY}})\boxtimes\sO_{{\hchow}}(-H_{{\hchow}})$.
Note that $\Vs$ has a natural $\mathrm{SL}(V)$-action. Therefore
the injection $\sO_{\widetilde{\hcoY}}(-M_{\widetilde{\hcoY}})\boxtimes\sO_{{\hchow}}(-H_{{\hchow}})\to\sO_{\widetilde{\hcoY}\times{\hchow}}$
is $\mathrm{SL}(V)$-equivariant. We can apply the construction as
in Subsection \ref{subsection:State} to this map since $\Hom(\sO_{\widetilde{\hcoY}}(-M_{\widetilde{\hcoY}}),\sO_{\widetilde{\hcoY}}))\simeq\ft{S}^{2}V$
and $\Hom(\sO_{{\hchow}}(-H_{{\hchow}}),\sO_{{\hchow}})\simeq\ft{S}^{2}V^{*}$.
Since the above injection is $\mathrm{SL}(V)$-equivariant, and $\ft{S}^{2}V\otimes\ft{S}^{2}V^{*}$
contains a unique one-dimensional representation, which is generated
by the identity element, we see that the above injection is induced
from the identity element as in Subsection \ref{subsection:State}.

\begin{prop} \label{cla:Vs} $\sI$ contains $\sI_{\Vs}$, equivalently,
the subvariety $\Delta$ is contained in $\Vs$. Set $\sI_{\Delta/\Vs}:=\sI/\sI_{\Vs}$,
the ideal sheaf of $\Delta$ in $\Vs$. Denote by $\iota_{\Vs}$ the
closed immersion $\Vs\hookrightarrow\widetilde{\hcoY}\times\hchow$.
Then $\iota_{\Vs*}\sI_{\Delta/\Vs}$ has the following locally free
resolution on $\widetilde{\hcoY}\times\hchow:$ \begin{equation}
\begin{aligned}0\;\to\; & \widetilde{\sS}_{L}\boxtimes\sO_{{\hchow}}\;\to\;\widetilde{\sT}^{*}\boxtimes g^{*}\sF^{*}\;\to\;\\
 & \sO_{\widetilde{\hcoY}}\boxtimes\{g^{*}\ft{S}^{2}\sF^{*}/\sO_{\hchow}(-H_{\hchow}+2L_{\hchow})\}\oplus\widetilde{\sQ}^{*}(M_{\widetilde{\hcoY}})\boxtimes\sO_{\hchow}(L_{\hchow})\;\to\\
 & \qquad\qquad\qquad\qquad\quad\;\;\iota_{{\Vs}*}\sI_{\Delta/\Vs}\otimes\{\sO_{\widetilde{\hcoY}}(M_{\widetilde{\hcoY}})\boxtimes\sO_{\hchow}(2L_{{\hchow}})\}\;\to\;0.\end{aligned}
\label{eqnarray:onVs}\end{equation}
 where the inclusion $\sO_{\hchow}(-H_{\hchow}+2L_{\hchow})\subset g^{*}\ft{S}^{2}\sF^{*}$
is defined by the Euler sequence $0\to\sO_{\mP(\ft{S}^{2}\sF)}(-1)\to g^{*}\ft{S}^{2}\sF\to T_{\mP(\ft{S}^{2}\sF)/\mathrm{G}(2,V)}(-1)\to0$
for $\hchow=\mP({\ft S}^{2}\sF)$ and the relation $\eF=\eF^{*}(L_{\hchow})$
with $L_{\hchow}=g^{*}\sO_{\rG(2,V)}(1)$. \end{prop}

\begin{proof} We have an $\mathrm{SL}(V)$-equivariant map \[
\sO_{\widetilde{\hcoY}}(-M_{\widetilde{\hcoY}})\boxtimes\sO_{{\hchow}}(-H_{{\hchow}})\to\sO_{\widetilde{\hcoY}}(-M_{\widetilde{\hcoY}})\boxtimes\ft{S}^{2}(g^{*}\sF),\]
 which is induced from the inclusion $\sO_{\mP(\ft{S}^{2}\sF)}(-1)\to g^{*}\ft{S}^{2}\sF$.
Therefore we have a $\mathrm{SL}(V)$-equivariant map \[
\sO_{\widetilde{\hcoY}}(-M_{\widetilde{\hcoY}})\boxtimes\sO_{{\hchow}}(-H_{{\hchow}})\to\sO_{\widetilde{\hcoY}}(-M_{\widetilde{\hcoY}})\boxtimes\ft{S}^{2}(g^{*}\sF)\oplus\widetilde{\sQ}^{*}\boxtimes\sO_{\hchow}(-L_{\hchow})\to\sI\hookrightarrow\sO_{\widetilde{\hcoY}\times{\hchow}}.\]
 By the uniqueness of such a map, its image coincides with $\sI_{\Vs}$.
Therefore $\sI_{\Vs}\subset\sI$.

The proof of the remaining assertion follows from the above discussion
and Theorem \ref{thm:resolY}. \end{proof}

$\;$ 

Recall that $f\colon\hchow\to\chow$ is the Hilbert-Chow morphism
and $E_{f}$ is the $f$-exceptional divisor as in Subsection \ref{HilbChow}.
Let $x$ be a point of $\mP(V)$ and $e$ any point of $E_{f}$ such
$f(e)=[2x]$. Then the fiber of $\Vs\to\hchow$ over $e$ is the pull-back
of the hyperplane section of $\Hes$ parameterizing singular quadrics
which contain the point $x$ by the duality between $\mP({\ft S}^{2}V)$
and $\mP({\ft S}^{2}V^{*})$. In particular, the fiber is independent
of a choice of $e$ once we fix a point $x$, thus we denote it by
$V_{x}$.

$\;$

\begin{prop} \label{cla:hyp} Any fiber of $\Vs\to\hchow$ is normal
and has only canonical singularities. \end{prop}

\begin{proof} The proof is similar to the argument in Subsection
\ref{subsection:CM}.

It suffices to show the assertion for $V_{x}$ ($x\in\mP(V)$) since
it is a special fiber of $\Vs\to\hchow$. Let $V_{x}^{t}$ is the
strict transform of $V_{x}$ on $\hcoY_{2}$, which is also the total
transform since $V_{x}$ does not contains the center $G_{\widetilde{\hcoY}}$
of the birational morphism $\hcoY_{2}\to\widetilde{\hcoY}$. Hence
$V_{x}^{t}\in|M_{\hcoY_{2}}|$. We see that $-K_{V_{x}^{t}}$ is $(\widetilde{\Lrho}_{2}|_{V_{x}^{t}})$-ample
since $V_{x}^{t}\in|M_{\hcoY_{2}}|$ and $-K_{\hcoY_{2}}=10M_{\hcoY_{2}}$
is $\widetilde{\Lrho}_{2}$-ample. Therefore it suffices to show similar
assertions for $V_{x}^{t}$. Let $W_{G}$, $W_{3}$ and $W_{2}$ be
the pull-backs on $\mathrm{G}(2,T(-1))$, $\Zpq_{3}$, and $\Zpq_{2}$,
respectively, of the subvariety $W_{0}:=\{\Pi\mid x\in\Pi\}\simeq\mathrm{G}(2,4)$
in $\mathrm{G}(3,V)$. It is easy to see that $V_{x}^{t}$ is the
image of $W_{2}$ set-theoretically and $W_{2}\to V_{x}^{t}$ is birational
(note that, once we fix a general quadric $Q$ and a $\mP^{1}$-family
$q$ of planes in $Q$, there is only one plane in $q$ which contains
$x$). We show that $W_{2}$ is smooth, therefore $W_{2}\to V_{x}^{t}$
is a resolution of singularities. Indeed, $W_{3}$ is smooth since
$W_{0}\simeq\mathrm{G}(2,4)$, $W_{G}\to W_{0}$ is a $\mP^{2}$-bundle,
and $W_{3}\to W_{G}$ is a $\mathrm{G}(2,5)$-bundle by Proposition
\ref{cla:G25}. Then $W_{2}$ is also smooth since $W_{2}\to W_{3}$
is the blow-up along $\Lpi_{3'}^{-1}(\Prt_{\rho})\cap W_{3}$, which
is a $\mP^{1}$-bundle over $W_{G}$ by Proposition \ref{cla:P1bdl},
and hence is smooth.

Similarly to the proof of Proposition \ref{cla:Delta}, we see that
the ideal sheaf $\sI_{\Gamma_{x}}$ of $\Gamma_{x}$ has the following
Koszul resolution: \begin{equation}
0\to\Lwedge^{2}{\eQ}^{*}\to{\eQ}^{*}\to\sI_{\Gamma_{x}}\to0.\label{eq:resolW}\end{equation}

First we check the assertions on ${\hcoY}_{2}^{o}$. For simplicity
of notation, we abbreviate the symbols for the restrictions. Since
$V_{x}^{t}$ is Gorenstein, we have only to show $V_{x}^{t}$ is normal
and has only rational singularities. In the same way to show (\ref{eq:R1}),
we obtain $\Lpi_{2'*}^{o}\iota_{2}^{*}\mu_{2}^{*}{\eQ}^{*}=R^{1}\Lpi_{2'*}^{o}\iota_{2}^{*}\mu_{2}^{*}{\eQ}^{*}=0$.
Therefore we have \begin{equation}
\Lpi_{2'*}^{o}\sI_{W_{2}}^{o}\simeq R^{1}\Lpi_{2'*}^{o}\Lwedge^{2}\iota_{2}^{*}\mu_{2}^{*}{\eQ}^{*},\label{eq:I'W}\end{equation}
 where $\sI_{W_{2}}^{o}$ is the ideal sheaf of $W_{2}$ in $\hcoY_{2}^{o}$.
In the same way to show (\ref{eq:duality}), we obtain by the Grothendieck-Verdier
duality \ref{cla:duality} \[
R^{1}\Lpi_{2'*}^{o}\Lwedge^{2}\iota_{2}^{*}\mu_{2}^{*}{\eQ}^{*}\simeq\big({\Lpi}_{2'*}^{o}\{\Lwedge^{2}\iota_{2}^{*}\mu_{2}^{*}{\eQ}\otimes\sO_{\Zpq_{2}^{o}}(-{N}_{\Zpq_{2}^{o}})\}\otimes\sO_{\hcoY_{2}^{o}}(M_{\hcoY_{2}^{o}})\big)^{*}\simeq\sO_{\hcoY_{2}^{o}}(-M_{\hcoY_{2}^{o}}).\]
 Therefore $\Lpi_{2'*}^{o}\sI_{W_{2}}^{o}$ is nothing but the ideal
sheaf $\sI_{V_{x}^{t}}^{o}$ of $V_{x}^{t}$ since $V_{x}^{t}$ is
the image of $W_{2}$ set-theoretically. In other words, $V_{x}^{t}$
is the scheme-theoretic pushforward of $W_{2}$. We set $\Lpi_{W}^{0}:=\Lpi_{2'}^{o}|_{W_{2}}\colon W_{2}\to V_{x}^{t}$.
In the same way to show Lemma \ref{cla:CM2}, we also have \begin{equation}
\text{\ensuremath{R^{k}\Lpi_{W*}^{o}\sO_{W_{2}}=0}for \ensuremath{k>0}and \ensuremath{\Lpi_{W*}^{o}\sO_{W_{2}}\simeq\sO_{V_{x}^{t}}}}.\label{eq:ratsingW1}\end{equation}
 Since $W_{2}$ is smooth, the latter shows that $V_{x}^{t}$ is normal
and the former shows that $V_{x}^{t}$ has only rational singularities.

Second we show the assertions on the whole $\hcoY_{2}$. By the above
argument on $\hcoY_{2}^{o}$, we see that $V_{x}^{t}$ is regular
in codimension one since the codimension of $V_{x}^{t}\cap\Prt_{\sigma}$
in $V_{x}^{t}$ is greater than two. Therefore $V_{x}^{t}$ is normal
since $V_{x}^{t}$ is Gorenstein. To check $V_{x}^{t}$ has only canonical
singularities, we have only to show that $W_{2}\to V_{x}^{t}$ is
crepant since $W_{2}$ is smooth. This follows by calculating the
canonical divisor of $W_{2}$. Note that $\sN_{W_{2}/\Zpq_{2}}\simeq\mu_{2}^{*}{\eQ}|_{W_{2}}$
by (\ref{eq:resolW}). Therefore $\det\sN_{W_{2}/\Zpq_{2}}\simeq\sO_{\Zpq_{2}}(-N_{\Zpq_{2}})|_{W_{2}}$
since $\det\eQ=\sO_{\mathrm{G}(3,V)}(1)$. Thus we have \[
\begin{aligned}K_{W_{2}} & =K_{\Zpq_{2}}|_{W_{2}}+\det\sN_{W_{2}/\Zpq_{2}}\\
 & =\{\Lpi_{2'}^{*}(K_{\hcoY_{2}}+M_{\hcoY_{2}})-N_{\Zpq_{2}}\}|_{W_{2}}+N_{\Zpq_{2}}|_{W_{2}}=\Lpi_{2'}^{*}(K_{\hcoY_{2}}+M_{\hcoY_{2}})|_{W_{2}},\end{aligned}
\]
 where we use Proposition \ref{cla:cano} (3) for the second equality.\end{proof}

$\;$

\section{The family of curves on $Y$ parameterized by $X$ revisited}

\label{section:flat}

Let $X$ and $Y$ be smooth Calabi-Yau threefolds which are mutually
orthogonal linear sections of $\chow$ and $\hcoY$ respectively,
as described in Subsection \ref{section:XY} and Section \ref{section:family}.
In this section, we show that a family of curves on $Y$ parameterized
by $X$ comes out from $\Delta\to\hchow$ and this family coincides
with one constructed in Section \ref{section:family}. We prove also
the family is flat.

\subsection{Flatness of $\Delta\to\hchow$ and the locally free resolution of
$\sI_{x}$}

\label{subsection:ideal}

Below, we denote by $(-1)$ the tensor product of $\sO_{\widetilde{\hcoY}}(-M_{\widetilde{\hcoY}})$.

\begin{prop} \label{cor:flat} \rm{(1)} The scheme $\Delta$ is
flat over $\hchow$. 

\noindent \rm{(2)} Let $\Delta_{x}$ be the fiber of $\Delta\to\hchow$
over a point $x\in\hchow$. Then the ideal sheaf $\sI_{x}$ of $\Delta_{x}$
in $\widetilde{\hcoY}$ is $\sI\otimes_{\sO_{\widetilde{\hcoY}\times\hchow}}\sO_{\widetilde{\hcoY}_{x}}$,
where $\widetilde{\hcoY}_{x}\simeq\widetilde{\hcoY}$ is the fiber
of $\widetilde{\hcoY}\times\hchow\to\hchow$ over $x$. Moreover,
the exact sequence $(\ref{eqn:fin5})$ remains to be exact after restricting
on $\widetilde{\hcoY}_{x}$ and gives the following locally free resolution
of $\sI_{x}:$ \begin{eqnarray}
\ 0\to\widetilde{\sS}_{L}(-1)\to{\widetilde{\sT}^{*}}(-1)^{\oplus2}\to\sO_{\widetilde{\hcoY}}(-1)^{\oplus3}\oplus\widetilde{\sQ}^{*}\to\sI_{x}\to0.\label{eqnarray:Cx}\end{eqnarray}
 \end{prop}

\begin{proof} By Proposition \ref{cla:descrdel3}, $\Delta\to\hchow$
is equi-dimensional (actually all fibers are isomorphic). Therefore
$\Delta$ is flat over $\hchow$ since $\Delta$ is Cohen-Macaulay
and $\hchow$ is smooth.

Tensoring the exact sequence $0\to\sI\to\sO_{\widetilde{\hcoY}\times\hchow}\to\sO_{\Delta}\to0$
with $k(x)$, 
we obtain the exact sequence $\sI\otimes k(x)\to\sO_{\widetilde{\hcoY}_{x}}\to\sO_{\Delta_{x}}\to0$.
Then, applying \cite[Theorem 22.5 (1)$\Rightarrow$(2)]{M}, we see
that $\sI\otimes k(x)\to\sO_{\widetilde{\hcoY}_{x}}$ is injective
since $\Delta$ is flat over $\hchow$. Note that, for any coherent
$\sO_{\widetilde{\hcoY}\times\hchow}$-module $\sA$ and a point $x\in\hchow$,
it holds that $\sA\otimes_{\sO_{\widetilde{\hcoY}\times\hchow}}\sO_{\widetilde{\hcoY}_{x}}\simeq\sA\otimes_{\sO_{\hchow}}k(x)$
since $\sO_{\widetilde{\hcoY}_{x}}\simeq{\sO_{\widetilde{\hcoY}\times\hchow}}\otimes_{\sO_{\hchow}}k(x)$.
Therefore $\sI_{x}\simeq\sI\otimes k(x)\simeq\sI\otimes_{\sO_{\widetilde{\hcoY}\times\hchow}}\sO_{\widetilde{\hcoY}_{x}}$.

Proof of the exactness of (\ref{eqnarray:Cx}) is similar. \end{proof}

\vspace{0.3cm}

\subsection{Cutting a family of curves $\Cd\to X$ from $\Delta$ }

\label{subsection:Cutting}

The generically conic bundle $\Zpq_{2}\to\hcoY_{2}$ in Proposition
\ref{prop:gen-conic-bundle2} defines a conic bundle $\Zpq_{2}^{o}\to\hcoY_{2}^{o}$
over $\hcoY_{2}^{o}=\hcoY_{2}\setminus\Prt_{\sigma}$, which we can
write explicitly as\[
\Zpq_{2}^{o}=\left\{ ([\bar{c}],(\bar{q},[\bar{U}],[V_{1}]))\,\Big|\,[\bar{c}]\in\bar{q},{\,\bar{q}\text{ is a conic in }\mP(\bar{U})\cap\rG(2,V/V_{1})\atop \text{ and }([\bar{U}],[V_{1}])\in\hcoY_{3}\setminus\Prt_{\sigma}}\right\} .\]
Let us recall $\overline{\hcoY}=\left\{ [U]\in\rG(3,\wedge^{3}V)\mid[U]=[\bar{U}\wedge V_{1}]\text{ for some }V_{1}\right\} $
and define \[
\widetilde{\Zpq}^{o}=\left\{ ([c],(q,[U]))\,\Big|\,[c]\in q,{q\text{ is a conic in }\mP(U)\cap\rG(3,V)\atop \text{ and }[U]\in\overline{\hcoY}\setminus\overline{\Prt}_{\sigma}}\,\right\} .\]
Then there is a conic bundle $\widetilde{\Zpq}^{o}\to\widetilde{\hcoY}^{o}$
with the following commutative diagram:\[
\xymatrix{\Zpq_{2}^{o}\ar[d]_{\Lpi_{2'}^{o}}\ar[r]_{\tLrho_{2'}^{o}} & \widetilde{\Zpq}^{o}\ar[d]^{\Lpi_{\widetilde{\Zpq}}^{o}}\\
\hcoY_{2}^{o}\ar[r]^{\tLrho_{2}^{o}} & \widetilde{\hcoY}^{o}}
\]
We note that the conic bundle $\widetilde{\Zpq}^{o}\to\widetilde{\hcoY}^{o}$
over $\widetilde{\hcoY}^{o}\setminus F_{\widetilde{\hcoY}}$ is isomorphic
to the conic bundle $\Zpq\to\hcoY$ over $\hcoY\setminus G_{\hcoY}$
in Subsection \ref{section:Quintic} (see \cite[Prop.~4.2.5,~5.2.1]{HoTa3}).

As we noted in Subsection \ref{section:XY}, we can assume that $X\subset\hchow$
and $Y$ is disjoint from $G_{\hcoY}=\Sing\hcoY$. Since $\widetilde{\hcoY}\to\hcoY$
is isomorphic outside $G_{\hcoY}$, we can also assume $Y\subset\widetilde{\hcoY}$.
Note that the subvariety $Y\times X$ of $\widetilde{\hcoY}\times{\hchow}$
is contained in $\Vs$ since $X$ and $Y$ are mutually orthogonal.
Now we consider the following restriction of $\Delta\subset\widetilde{\hcoY}\times\hchow$:
\[
\Cd={\Delta}|_{X\times Y}.\]
 We denote by $I$ the ideal sheaf of $\Cd$ in $Y\times X$ and by
$I_{x}$ the ideal sheaf of $C_{x}$ in $Y_{x}$, where $Y_{x}=Y$
is the fiber of $Y\times X\to X$ over $x$.

\vspace{0.2cm}
 \begin{prop} \label{prop:Cflat} The scheme $\Cd$ is flat over
$X$ and its fiber over $x\in X$ coincides with $C_{x}$ defined
in Subsection $\ref{subsection:curves}$. Moreover, it holds that
$I\simeq\sI_{\Delta/\Vs}\otimes_{\sO_{\Vs}}\sO_{Y\times X}$ and $I_{x}\simeq\sI_{\Delta_{x}/\Vs_{x}}\otimes_{\sO_{\widetilde{\hcoY}_{x}}}\sO_{Y_{x}}$.
\end{prop}

\begin{proof} Recall that the curve $C_{x}\subset Y$ in Subsection
$\ref{subsection:curves}$ is defined through $\gamma_{x}=G_{x}\cap\Lpi_{\Zpq}^{-1}(Y)$,
which consists of $([\Pi],[Q])\in\Zpq$ satisfying $l_{x}\subset\mP(\Pi)$
and $[Q]\in P$. Let us define $\Delta_{\widetilde{\Zpq}^{o}}\subset\widetilde{\Zpq}^{o}\times\hchow$
to be the image of $\Delta_{2}^{o}\subset\widetilde{\Zpq_{2}}^{o}\times\hchow$
under $\tLrho_{2'}^{o}\times\id_{\hchow}$, and denote by $\Delta_{\widetilde{\Zpq}^{o},x}$
the fiber over $x\in X$. Then, by definition, the intersection $\Delta_{\widetilde{\Zpq}^{o},x}\cap\Lpi_{\widetilde{\Zpq}}^{o-1}(Y)$
consists of points $([c],(q,[U]))\in\widetilde{\Zpq}^{o}$ such that
$l_{x}\subset\mP(V_{c})$ and $\Lpi_{\widetilde{\Zpq}}(([c],(q,[U])))\in Y$
with the three dimensional subspace $V_{c}\subset V$ representing
$c=[\wedge^{3}V_{c}].$ Here we note that $\Lpi_{\widetilde{\hcoY}}(([c],(q,[U])))\in Y$
is equivalent to $[Q_{y}]\in P$ for the quadric $Q_{y}$ which corresponds
to $y=\Lpi_{\widetilde{\hcoY}}((q,[U]))$. Therefore $\gamma_{x}$
can identified with $\Delta_{\widetilde{\Zpq}^{o},x}\cap\Lpi_{\widetilde{\Zpq}}^{o-1}(Y)$.
Now, due to the isomorphism $\widetilde{\Zpq}^{o}\to\widetilde{\hcoY}^{o}$
over $\widetilde{\hcoY}^{o}\setminus F_{\widetilde{\hcoY}}$ to the
conic bundle $\Zpq\to\hcoY$ over $\hcoY\setminus G_{\hcoY},$ the
curve $C_{x}=\Lpi_{\hcoY}(\gamma_{x})$ can be identified with the
curve that is defined by the fiber of the scheme $\Cd$ over $x\in X$.
Hereafter we denote by $C_{x}$ the fiber of $\Cd\to X$ over $x\in X$.

Recall that the fiber $\Vs_{x}$ of $\Vs\to\hchow$ over $x$ contains
$\Delta_{x}$ by Proposition \ref{cla:Vs}. Since $\Vs_{x}$ contains
also $Y$, we see that $Y$ is the complete intersection in $\Vs_{x}$
of $9$ members $M_{1},\dots,M_{9}$ of $|M_{\widetilde{\hcoY}}|_{\Vs_{x}}|$.
Therefore $C_{x}$ is cut out from $\Delta_{x}$ in $\Vs_{x}$ by
$M_{1},\dots,M_{9}$. By \cite[Corollary to Theorem 23.3]{M}, ${\Delta}_{x}$
is Cohen-Macaulay for any $x\in\hchow$ since $\Delta$ is Cohen-Macaulay
by Theorem \ref{thm:resolY} and is flat over $\hchow$ by Proposition
\ref{cor:flat}. Since $C_{x}$ is one-dimensional for any $x\in X$,
the divisors $M_{1},\dots,M_{9}$ form a regular sequence \cite[Theorem 17.4 iii)]{M}.
Note that $\Delta|_{\widetilde{\hcoY}\times X}\to X$ is flat by Proposition
\ref{cor:flat} and its ideal sheaf is $\sI\otimes\sO_{\widetilde{\hcoY}\times X}$
by \cite[Theorem 22.5 (2)$\Rightarrow$(1)]{M}. Therefore, by \cite[Corollary to Theorem 22.5 (2)$\Rightarrow$(1)]{M},
$\Cd$ is flat over $X$ and is cut out in $\Vs$ from $\Delta|_{\widetilde{\hcoY}\times X}$
by a regular sequence. The latter implies that $\sI_{\Delta/\Vs}\otimes_{\sO_{\widetilde{\hcoY}\times\hchow}}\sO_{Y\times X}$
is the ideal sheaf of $\Cd$ in $X\times Y$. Similarly, we have $I_{x}=\sI_{\Delta_{x}/\Vs_{x}}\otimes_{\sO_{\widetilde{\hcoY}_{x}}}\sO_{Y_{x}}$
since the scheme $C_{x}$ is cut out in $\Vs_{x}$ from ${\Delta}_{x}$
by a regular sequence. \end{proof}

$\;$\pagebreak{}

\section{Derived equivalence}

\label{section:BC} In this section, we derive the main result of
this article: 

\begin{thm} \label{thm:main} Let $X$ and $Y$ be smooth Calabi-Yau
threefolds which are mutually orthogonal linear sections of $\chow$
and $\hcoY$ respectively. Let $I$ be the ideal sheaf as in Proposition
$\ref{prop:Cflat}$. Then the Fourier-Mukai functor $\Phi_{I}$ with
$I$ as its kernel is an equivalence between $\sD^{b}(X)$ and $\sD^{b}(Y)$.
\end{thm}

Since the cohomology groups related to the locally free resolution
of $\sI$ have been computed in \cite[Thm.8.1.1]{HoTa3}, the rest
of our proof of the derived equivalence between $X$ and $Y$ proceeds
in the same way as that of \cite{BC}. Namely, we show that the functor
$\Phi_{I}:\sD^{b}(X)\to\sD^{b}(Y)$ is an equivalence by verifying
the conditions (i) and (ii) of Theorem \ref{thm:D}. The condition
(i) may be verified by the following general lemma \cite[Proposition 4.5]{BC}.
We include the proof for completeness.

\begin{lem} \label{cla:Hom} Let $Y$ be a smooth projective threefold
and $I$ an ideal sheaf of $\sO_{Y}$ such that the closed subscheme
$C$ defined by $I$ is of $($not necessarily pure$)$ dimension
less than or equal to one. Then $\Hom(I,I)\simeq\mC$. \end{lem}

\begin{proof} Taking $\Hom(I,-)$ of the exact sequence \begin{equation}
0\to I\to\sO_{Y}\to\sO_{C}\to0,\label{eq:OC}\end{equation}
 we obtain an injection $\Hom(I,I)\to\Hom(I,\sO_{Y})$. We have only
to show $\Hom(I,\sO_{Y})\simeq\mC$ since $\Hom(I,I)$ contains at
least constant maps. To compute $\Hom(I,\sO_{Y})$, we take $\Hom(-,\sO_{Y})$
of (\ref{eq:OC}). Then we obtain the exact sequence $0\to\Hom(\sO_{Y},\sO_{Y})\to\Hom(I,\sO_{Y})\to\Ext^{1}(\sO_{C},\sO_{Y})$.
By the Serre duality, we have $\Ext^{1}(\sO_{C},\sO_{Y})\simeq H^{2}(Y,\sO_{C}\otimes\omega_{Y})$,
where the r.h.s. is $0$ since $\dim C\leq1$. Therefore we have $\Hom(I,\sO_{Y})\to\Hom(\sO_{Y},\sO_{Y})\simeq\mC$.
\end{proof}

In what follows, we show the property (ii) of Theorem \ref{thm:D},
i.e., the vanishing: \begin{equation}
\Ext^{\bullet}(I_{x_{1}},I_{x_{2}})=0\text{ for any two distinct points }\ensuremath{x_{1}}\text{ and }\ensuremath{x_{2}}\text{ of }\ensuremath{X}.\label{eq:aim}\end{equation}

We denote by $(-t)$ the tensor product of $\sO_{\widetilde{\hcoY}}(-tM_{\widetilde{\hcoY}})$.
Then the result (3) of \cite[Thm.8.1.1]{HoTa3} may be read as:

\begin{thm} \label{thm:digGvan} $H^{\bullet}({\sA}^{*}\otimes{\sB}(-t))=0\,(1\leq t\leq9)$,
where ${\sA}$ and ${\sB}$, respectively, are one of the sheaves
$\widetilde{\sS}_{L}$, $\widetilde{\sT}^{*}$, $\sO_{\widetilde{\hcoY}}$,
$\widetilde{\sQ}^{*}(M_{\widetilde{\hcoY}})$. \end{thm}

\noindent It is standard to derive the following proposition from
Proposition \ref{cor:flat} and Theorem \ref{thm:digGvan}. \begin{prop}
\label{prop:Extvan} For any two points $x_{1}$ and $x_{2}$ of $\hchow$,
it holds \[
\Ext^{\bullet}(\sI_{x_{1}},\sI_{x_{2}}(-t))=0\,(1\leq t\leq9).\]
 \end{prop}

\vspace{0.3cm}
 We derive the vanishing (\ref{eq:aim}) from Proposition \ref{prop:Extvan}
following \cite[Subsection 5.6]{BC}. For the derivation, it is important
to a choose suitable sequence of the complete intersections of the
members of $|M_{\widetilde{\hcoY}}|$ containing $Y$.

\begin{lem} \label{cla:tower} There exists a tower of complete intersections
of $\widetilde{\hcoY}$ \[
Y_{0}\subset Y_{1}\subset\cdots\subset Y_{9}\subset Y_{10}\;(\dim Y_{i}=3+i),\]
by the members of $|M_{\widetilde{\hcoY}}|$, which satisfies the
following conditions: Set ${\Delta}_{x_{i};j}:={{\Delta}_{x_{i}}}|_{Y_{j}}$
and denote by $\sI_{x_{i};j}$ the ideal sheaf of ${{\Delta}_{x_{i};j}}$
in $Y_{j}$ and by $\iota_{Y_{j}}$ the embedding $Y_{j}\hookrightarrow Y_{j+1}.$
Then 

\begin{myitem2}

\item[\rm{(1)}]$Y_{0}=Y$ and $Y_{10}=\widetilde{\hcoY}$. 

\item[\rm{(2)}]$Y_{9}=\Vs_{x_{1}}$, where $\Vs_{x_{1}}$ is the
fiber of $\Vs\to\hchow$ over $x_{1}$ $($cf.~Section $\ref{subsection:Univ})$.
In particular, $Y_{9}$ contains ${\Delta}_{x_{1}}$ $($ Proposition
$\ref{cla:Vs})$, and hence $\iota_{Y_{9}*}\sI_{x_{1};9}=\sI_{x_{1}}/\sO_{\widetilde{\hcoY}}(-Y_{9})\simeq\sI_{x_{1}}/\sO_{\widetilde{\hcoY}}(-1)$.
The ideal sheaf $\sI_{x_{2};9}$ is equal to $\sI_{x_{2}}\otimes\sO_{Y_{9}}\simeq\iota_{Y_{9}}^{*}\sI_{x_{2}}$. 

\item[\rm{(3)}]$Y_{8}=\Vs_{x_{1}}\cap\Vs_{x_{2}}$, where the intersection
is taken in $\widetilde{\hcoY}$. In particular, $Y_{8}$ contains
${{\Delta}_{x_{2}}}|_{Y_{9}}$, and hence $\iota_{Y_{8}*}\sI_{x_{2};8}=\sI_{x_{2};9}/\sO_{Y_{9}}(-Y_{8})\simeq\sI_{x_{2};9}/\sO_{Y_{9}}(-1)$.
The ideal sheaf $\sI_{x_{1};8}$ is equal to $\sI_{x_{1};9}\otimes\sO_{Y_{8}}\simeq\iota_{Y_{8}}^{*}\sI_{x_{1};9}$. 

\item[\rm{(4)}]For any $j\leq7$, the ideal sheaf $\sI_{x_{i};j}$
is equal to $\sI_{x_{i};j+1}\otimes\sO_{Y_{j}}\simeq\iota_{Y_{j}}^{*}\sI_{x_{i};j+1}$. 

\end{myitem2}\end{lem}

\begin{proof} We take $Y_{9}=\Vs_{x_{1}}$ and $Y_{8}=\Vs_{x_{1}}\cap\Vs_{x_{2}}$
as in the statement and let $Y_{7},\dots,Y_{0}$ be general complete
intersections containing $Y$ (recall that $Y$ is contained in any
fiber of $\Vs\to\hchow$ over a point of $X$). Note that $\Vs_{x_{1}}$
is irreducible by Proposition \ref{cla:hyp}, and also $\Vs_{x_{1}}\not=\Vs_{x_{2}}$
for $x_{1}\not=x_{2}$. The last property follows from the duality
between $\mP({\ft S}^{2}V)$ and $\mP({\ft S}^{2}V^{*})$ and the
fact that $X$ is embedded in $\chow\subset\mP({\ft S}^{2}V)$.

The descriptions of $\sI_{x_{2};9}$, $\sI_{x_{1};8}$ and $\sI_{x_{i};j}$
($i=1,2$, $0\leq j\leq7$) follow from the last part of the proof
of Proposition \ref{prop:Cflat}. \end{proof}

\vspace{0.2cm}
 The choices of $Y_{9}$ and $Y_{8}$ in the lemma turns out to be
crucial in Steps 1 and 2 of the arguments below. 

\textbf{$\;$}

\textbf{Step 1 (from $\widetilde{\hcoY}$ to $Y_{9}$).}

\noindent In this step, we show \begin{equation}
\Ext_{Y_{9}}^{\bullet-1}(\sI_{x_{1};9},\sI_{x_{2};9}(-t+1))=0\,(1\leq t\leq9).\label{eq:step1}\end{equation}

By $\sI_{x_{2};9}\simeq\iota_{Y_{9}}^{*}\sI_{x_{2}}$, we have \begin{eqnarray}
\Ext_{Y_{9}}^{\bullet-1}(\sI_{x_{1};9},\sI_{x_{2};9}(-t+1))\simeq\Ext_{Y_{9}}^{\bullet-1}(\sI_{x_{1};9},\iota_{Y_{9}}^{*}\sI_{x_{2}}(-t+1))\label{eqnarray:A}\end{eqnarray}
 By applying the Grothendieck-Verdier duality \ref{cla:duality} to
the embedding $\iota_{Y_{9}}\colon Y_{9}\hookrightarrow\widetilde{\hcoY}$,
we have \[
\Ext_{Y_{9}}^{\bullet-1}(\sI_{x_{1};9},\iota_{Y_{9}}^{*}\sI_{x_{2}}(-t+1))\simeq\Ext_{\widetilde{\hcoY}}^{\bullet}(\iota_{Y_{9}*}\sI_{x_{1};9},\sI_{x_{2}}(-t)).\]
 Therefore we have \begin{equation}
\Ext_{Y_{9}}^{\bullet-1}(\sI_{x_{1};9},\sI_{x_{2};9}(-t+1))\simeq\Ext_{\widetilde{\hcoY}}^{\bullet}(\iota_{Y_{9}*}\sI_{x_{1};9},\sI_{x_{2}}(-t)).\label{eq:B}\end{equation}
 Taking $\Hom(-,\sI_{x_{2}}(-t))$ of the exact sequence \[
0\to\sO_{\widetilde{\hcoY}}(-1)\to\sI_{x_{1}}\to\iota_{Y_{9}*}\sI_{x_{1};9}\to0\]
 (cf.~Lemma \ref{cla:tower} (2)), we obtain the exact sequence \begin{equation}
\begin{aligned} & H^{\bullet-1}(\sI_{x_{2}}((-t+1))\to\Ext_{\widetilde{\hcoY}}^{\bullet}(\iota_{Y_{9}*}\sI_{x_{1};9},\sI_{x_{2}}(-t))\to\Ext_{\widetilde{\hcoY}}^{\bullet}(\sI_{x_{1}},\sI_{x_{2}}(-t)),\end{aligned}
\label{eqnarray:C}\end{equation}
 where the last term vanishes by Proposition \ref{prop:Extvan}. 

\begin{cla} $H^{\bullet-1}(\sI_{x_{2}}((-t+1))=0$. \end{cla} 

\begin{proof} The assertion follows from Proposition \ref{cor:flat}
and Theorem \ref{thm:digGvan}. \end{proof} Therefore we have (\ref{eq:step1})
from (\ref{eq:B}) and (\ref{eqnarray:C}).\\

\vspace{0.3cm}
 \textbf{Step 2 (from $Y_{9}$ to $Y_{8}$).}

\noindent In this step, we show \begin{equation}
\Ext_{Y_{8}}^{\bullet-1}(\sI_{x_{1};8},\sI_{x_{2};8}(-t+1))=0\,(1\leq t\leq9).\label{eq:step2}\end{equation}

Since $\sI_{x_{1};8}\simeq\iota_{Y_{8}}^{*}\sI_{x_{1};9}$, we have
\begin{equation}
\begin{aligned}\Ext_{Y_{8}}^{\bullet-1}(\sI_{x_{1};8},\sI_{x_{2};8}(-t+1)) & \simeq\Ext_{Y_{8}}^{\bullet-1}(\iota_{Y_{8}}^{*}\sI_{x_{1};9},\sI_{x_{2};8}(-t+1))\\
 & \simeq\Ext_{Y_{9}}^{\bullet-1}(\sI_{x_{1};9},\iota_{Y_{8}*}\sI_{x_{2};8}(-t+1)).\end{aligned}
\label{eqnarray:D}\end{equation}
 From (\ref{eqnarray:D}) and $\Hom(\sI_{x_{1};9}(t-1),-)$ of the
exact sequence \[
0\to\sO_{Y_{9}}(-1)\to\sI_{x_{2};9}\to\iota_{Y_{8}*}\sI_{x_{2};8}\to0\]
 (cf.~Lemma \ref{cla:tower} (3)) we obtain the exact sequence \begin{equation}
\begin{aligned}\Ext_{Y_{9}}^{\bullet-1}(\sI_{x_{1};9},\sI_{x_{2};9}(-t+1)) & \to\Ext_{Y_{8}}^{\bullet-1}(\sI_{x_{1};8},\sI_{x_{2};8}(-t+1))\\
 & \to\Ext_{Y_{9}}^{\bullet}(\sI_{x_{1};9},\sO_{Y_{9}}(-t)),\end{aligned}
\label{eqnarray:E}\end{equation}
 where the first term vanishes from (\ref{eq:step1}).

\begin{cla} \label{cla:vanishing} $\Ext_{Y_{9}}^{\bullet}(\sI_{x_{1};9},\sO_{Y_{9}}(-t))=0$.
\end{cla}

\begin{proof} Set $F_{Y_{9}}:=F_{\widetilde{\hcoY}}|_{Y_{9}}$. By
the Serre-Grothendieck duality, it holds \begin{equation}
\Ext_{Y_{9}}^{\bullet}(\sI_{x_{1};9},\sO_{Y_{9}}(-t))\simeq H^{12-\bullet}(Y_{9},\sI_{x_{1};9}((t-9)M_{Y_{9}}+2F_{Y_{9}}))^{*}\label{eq:step2dual}\end{equation}
 since $K_{Y_{9}}=-9M_{Y_{9}}+2F_{Y_{9}}$. We show that \begin{equation}
H^{12-\bullet}(Y_{9},\sI_{x_{1};9}((t-9)M_{Y_{9}}+2F_{Y_{9}}))\simeq H^{12-\bullet}(\widetilde{\hcoY},\sI_{x_{1}}((t-9)M_{\widetilde{\hcoY}}+2F_{\widetilde{\hcoY}})).\label{eq:tY-Y9}\end{equation}
 Since ${\Delta}_{x_{1};9}={{\Delta}_{x_{1}}}$, we have the following
two exact sequences on $\widetilde{\hcoY}$ and $Y_{9}$ respectively:
\begin{equation}
\begin{aligned}0\to\sI_{x_{1}}((t-9)M_{\widetilde{\hcoY}}+2F_{\widetilde{\hcoY}}) & \to\sO_{\widetilde{\hcoY}}((t-9)M_{\widetilde{\hcoY}}+2F_{\widetilde{\hcoY}})\\
 & \to\sO_{{\Delta}_{x_{1}}}((t-9)M_{\widetilde{\hcoY}}+2F_{\widetilde{\hcoY}})\to0,\end{aligned}
\label{eqnarray:tY}\end{equation}
 and \begin{equation}
\begin{aligned}0\to\sI_{x_{1};9}((t-9)M_{Y_{9}}+2F_{Y_{9}}) & \to\sO_{Y_{9}}((t-9)M_{Y_{9}}+2F_{Y_{9}})\\
 & \to\sO_{{\Delta}_{x_{1}}}((t-9)M_{\widetilde{\hcoY}}+2F_{\widetilde{\hcoY}})\to0.\end{aligned}
\label{eqnarray:Y9}\end{equation}
 Since $(t-9)M_{\widetilde{\hcoY}}+2F_{\widetilde{\hcoY}}=(t+1)M_{\widetilde{\hcoY}}+K_{\widetilde{\hcoY}}$
and $(t+1)M_{\widetilde{\hcoY}}$ is nef and big, the cohomology groups
$H^{12-\bullet}(\widetilde{\hcoY},\sO_{\widetilde{\hcoY}}((t-9)M_{\widetilde{\hcoY}}+2F_{\widetilde{\hcoY}}))$
vanish for $12-\bullet>0$ by the Kawamata-Viewheg vanishing theorem.
Moreover, it holds that $H^{0}(\widetilde{\hcoY},\sO_{\widetilde{\hcoY}}((t-9)M_{\widetilde{\hcoY}}+2F_{\widetilde{\hcoY}}))\simeq H^{0}({\hcoY},\sO_{{\hcoY}}((t-9)M_{{\hcoY}}))$,
and the latter vanishes if $t\leq8$, and is isomorphic to $\mC$
if $t=9$. If $t=9$, then $H^{0}(\widetilde{\hcoY},\sO_{\widetilde{\hcoY}}(2F_{\widetilde{\hcoY}}))\simeq H^{0}({\Delta}_{x_{1}},\sO_{{\Delta}_{x_{1}}}(2F_{\widetilde{\hcoY}}))$
and hence $H^{0}(\widetilde{\hcoY},\sI_{x_{1}}(2F_{\widetilde{\hcoY}}))=0$
by (\ref{eqnarray:tY}). In the remaining cases, we have $H^{12-\bullet}(\widetilde{\hcoY},\sI_{x_{1}}((t-9)M_{\widetilde{\hcoY}}+2F_{\widetilde{\hcoY}}))\simeq H^{11-\bullet}({\Delta}_{x_{1}},\sO_{{\Delta}_{x_{1}}}((t-9)M_{\widetilde{\hcoY}}+2F_{\widetilde{\hcoY}}))$
by (\ref{eqnarray:tY}). Similarly, by (\ref{eqnarray:Y9}) and the
Kawamata-Viewheg vanishing theorem, we have $H^{0}(Y_{9},\sI_{x_{1};9}(2F_{Y_{9}}))=0$,
and in the remaining cases, $H^{12-\bullet}(Y_{9},\sI_{x_{1};9}((t-9)M_{Y_{9}}+2F_{Y_{9}}))\simeq H^{11-\bullet}({\Delta}_{x_{1}},\sO_{{\Delta}_{x_{1}}}((t-9)M_{\widetilde{\hcoY}}+2F_{\widetilde{\hcoY}}))$,
where we need $Y_{9}=\Vs_{x_{1}}$ has only canonical singularities
(see Proposition \ref{cla:hyp}).

Therefore we have the isomorphism (\ref{eq:tY-Y9}).

Now we have only to show the vanishings of $H^{12-\bullet}(\widetilde{\hcoY},\sI_{x_{1}}((t-9)M_{\widetilde{\hcoY}}+2F_{\widetilde{\hcoY}}))$
by (\ref{eq:step2dual}) and (\ref{eq:tY-Y9}). These follow from
the vanishings of \[
\begin{aligned}H^{12-\bullet}(\widetilde{\hcoY},\widetilde{\sS}_{L}((t-10)M_{\widetilde{\hcoY}}+2F_{\widetilde{\hcoY}})),\, & \;\; H^{12-\bullet}(\widetilde{\hcoY},\widetilde{\sT}^{*}((t-10)M_{\widetilde{\hcoY}}+2F_{\widetilde{\hcoY}})),\\
H^{12-\bullet}(\widetilde{\hcoY},\sO_{\widetilde{\hcoY}},((t-10)M_{\widetilde{\hcoY}}+2F_{\widetilde{\hcoY}})),\, & \;\; H^{12-\bullet}(\widetilde{\hcoY},\widetilde{\sQ}^{*}((t-9)M_{\widetilde{\hcoY}}+2F_{\widetilde{\hcoY}}))\end{aligned}
\]
 by (\ref{eqnarray:Cx}). Those cohomology groups are Serre-dual to
\[
\begin{aligned}H^{\bullet+1}(\widetilde{\hcoY},\widetilde{\sS}_{L}^{*}(-t)),\, & \quad H^{\bullet+1}(\widetilde{\hcoY},\widetilde{\sT}(-t)),\\
H^{\bullet+1}(\widetilde{\hcoY},\sO_{\widetilde{\hcoY}}(-t)),\, & \quad H^{\bullet+1}(\widetilde{\hcoY},\widetilde{\sQ}(-t-1)),\end{aligned}
\]
 which vanish by Theorem \ref{thm:digGvan}. Therefore, we have $\Ext_{Y_{9}}^{\bullet}(\sI_{x_{1};9},\sO_{Y_{9}}(-t))=0$
by (\ref{eq:step2dual}). \end{proof}


Now (\ref{eq:step2}) follows from (\ref{eqnarray:E}).

\vspace{0.8cm}
\textbf{Step 3 (from $Y_{8}$ to $Y_{7}$).}

\noindent In this step, we show \begin{equation}
\Ext_{Y_{7}}^{\bullet-1}(\sI_{x_{1};7},\sI_{x_{2};7}(-t+1))=0\,(1\leq t\leq8).\label{eq:step3}\end{equation}

Since $\sI_{x_{1};7}\simeq\iota_{Y_{7}}^{*}\sI_{x_{1};8}$, we have
\begin{equation}
\begin{aligned}\Ext_{Y_{7}}^{\bullet-1}(\sI_{x_{1};7},\sI_{x_{2};7}(-t+1)) & \simeq\Ext_{Y_{7}}^{\bullet-1}(\iota_{Y_{7}}^{*}\sI_{x_{1};8},\sI_{x_{2};7}(-t+1))\\
 & \simeq\Ext_{Y_{8}}^{\bullet-1}(\sI_{x_{1};8},\iota_{Y_{7}*}\sI_{x_{2};7}(-t+1)).\end{aligned}
\label{eqnarray:D'}\end{equation}
 Note that a defining equation of $Y_{7}$ restricts to a regular
element in each local ring of ${\Delta}_{x_{2};8}$. Hence the sequence
\begin{equation}
0\to\sI_{x_{2};8}(-1)\to\sI_{x_{2};8}\to\iota_{Y_{7}*}\sI_{x_{2};7}\to0\label{eq:87}\end{equation}
 is exact. Therefore, by (\ref{eqnarray:D'}), and $\Hom(\sI_{x_{1};8}(t-1),-)$
of (\ref{eq:87}), we have the following exact sequence: \begin{equation}
\begin{aligned}\Ext_{Y_{8}}^{\bullet-1}(\sI_{x_{1};8},\sI_{x_{2};8}(-t+1)) & \to\Ext_{Y_{7}}^{\bullet-1}(\sI_{x_{1};7},\sI_{x_{2};7}(-t+1))\\
 & \to\Ext_{Y_{8}}^{\bullet}(\sI_{x_{1};8},\sI_{x_{2};8}(-t)),\end{aligned}
\label{eqnarray:C'}\end{equation}
 where $\Ext_{Y_{8}}^{\bullet-1}(\sI_{x_{1};8},\sI_{x_{2};8}(-t+1))$
and $\Ext_{Y_{8}}^{\bullet}(\sI_{x_{1};8},\sI_{x_{2};8}(-t))$ vanish
by (\ref{eq:step2}) if $1\leq t\leq8$, and then we have (\ref{eq:step3}).\\

\vspace{0.3cm}
 \textbf{Step 4 (from $Y_{7}$ to $Y_{6}$, \dots, $Y_{1}$ to $Y$).}

\noindent In this step, we finish the proof of (\ref{eq:aim}). Since
a defining equation of $Y_{i}$ restricts to a regular element in
each local ring of both ${\Delta}_{x_{1};i+1}$ and ${\Delta}_{x_{2};i+1}$,
we can show inductively the following vanishing for any $i\in[0,6]$
in a similar way to the argument of Step 3: \begin{equation}
\Ext_{Y_{i}}^{\bullet-1}(\sI_{x_{1};i},\sI_{x_{2};i}(-t+1))=0\,(1\leq t\leq i+1).\label{eq:step4}\end{equation}

In particular, we have $\Ext_{Y}^{\bullet-1}(I_{x_{1}},I_{x_{2}})=0$,
which is (\ref{eq:aim}). \hfill{}$\square$

\vspace{1cm}
 \pagebreak{}

\section{Further discussions}

In our proof of the derived equivalence, we have used (the ideal sheaf
of) a family of curves $\{C_{x}\}_{x\in X}$ which arises from the
restriction $\sC=\Delta|_{X\times Y}$. Obviously, the other choice
of a family $\{C_{y}\}_{y\in Y}$ should be possible for that purpose.
In this section, we obtain a flat family of curves on $X$ from $\Delta$
for the latter choice, and remark, however, that a technical problem
prevent us to complete a proof by using this family. We also make
a comment on non-invariances of the fundamental groups and the Brauer
groups under the derived equivalence.

\vspace{0.3cm}

\subsection{A family of curves on $X$}

\label{subsection:onX}

For a point $y\in\widetilde{\hcoY}$, we denote by $Q_{y}$ the quadric
in $\mP^{4}$ corresponding to the image of $y$ on $\Hes$. We also
denote by ${\Delta}_{y}$ the fiber of ${\Delta}\to\widetilde{\hcoY}$
over $y$, which is a closed subscheme of ${\hchow}$.

Let $V_{\widetilde{\hcoY}}$ be the open subset of $\widetilde{\hcoY}$
consisting of points $y$ such that $\rank Q_{y}=3$ or $4$. For
a point $y\in V_{\widetilde{\hcoY}}$, let $q_{y}\subset\mathrm{G}(3,V)$
be the conic corresponding to $y$ ($q_{y}$ is one of the connected
components of the families of planes in $Q_{y}$). We describe $\Delta_{y}$
for $y\in V_{\widetilde{\hcoY}}$.

\begin{prop} \label{cla:fiber} ${\Delta}_{y}$ is the restriction
of $\hchow=\mP({\ft S}^{2}\sF)\to\mathrm{G}(2,V)$ over the subset
\[
G_{y}:=\{l\mid l\text{ is contained in a plane belonging to \ensuremath{q_{y}}}\}\subset\mathrm{G}(2,V).\]
 If $\rank Q_{y}=4$, then $G_{y}$ is isomorphic to $\mP(\sO_{\mP^{1}}(-1)^{\oplus2}\oplus\sO_{\mP^{1}}(-2))$
and $\sO_{\mathrm{G}(2,V)}(1)|_{G_{y}}$ is the tautological divisor.
If $\rank Q_{y}=3$, then $G_{y}$ is isomorphic to the image of $\mP(\sO_{\mP^{1}}\oplus\sO_{\mP^{1}}(-2)^{\oplus2})$
by the tautological linear system and ${\sO_{\mathrm{G}(2,V)}(1)}|_{G_{y}}$
is the tautological divisor.

In particular, $\Delta_{y}$ is smooth if $\rank Q_{y}=4$. \end{prop}

\begin{proof} The first part easily follows from the definition of
$\Delta$.

We describe $G_{y}$. By definition, $G_{y}$ is nothing but the restriction
of $\Delta_{0}\to\mathrm{G}(3,V)$ over $q_{y}$. Since $\Delta_{0}=\mathrm{F}(2,3,V)$,
we have $\Delta_{0}=\mP(\eS^{*})$ as a $\mP^{2}$-bundle over $\mathrm{G}(3,V)$.
We denote by $p_{1}$ the projection $\Delta_{0}\to\mathrm{G}(3,V)$
and by $p_{2}$ the projection $\Delta_{0}\to\mathrm{G}(2,V)$. We
show \begin{equation}
\sO_{\mP(\eS^{*})}(1)\simeq p_{2}^{*}\sO_{\mathrm{G}(2,V)}(1)\otimes p_{1}^{*}\sO_{\mathrm{G}(3,V)}(-1).\label{eq:eStaut}\end{equation}
 Indeed, by $\eS^{*}\simeq\wedge^{2}\eS\otimes\sO_{\mathrm{G}(3,V)}(1)$,
we have $\mP(\eS^{*})\simeq\mP(\wedge^{2}\eS)$ and $\sO_{\mP(\eS^{*})}(1)\simeq\sO_{\mP(\wedge^{2}\eS)}(1)\otimes p_{1}^{*}\sO_{\mathrm{G}(3,V)}(-1)$.
Moreover, by the universal exact sequence $0\to\eS\to V\otimes\sO_{\mathrm{G}(3,V)}\to\eQ\to0$,
we obtain the injection $\wedge^{2}\eS\to\wedge^{2}V\otimes\sO_{\mathrm{G}(3,V)}$.
Therefore, $\sO_{\mP(\wedge^{2}\eS)}(1)=p_{2}^{*}\sO_{\mathrm{G}(2,V)}(1)$,
which implies (\ref{eq:eStaut}). Note that (\ref{eq:eStaut}) is
equivalent to $\sO_{\mP(\eS^{*}(-1))}(1)\simeq p_{2}^{*}\sO_{\mathrm{G}(2,V)}(1)$.

By the discussion in the proof of \cite[Prop.~5.2.1]{HoTa3}, $\eS^{*}(-1)|_{q_{y}}\simeq\sO_{\mP^{1}}(-1)^{\oplus2}\oplus\sO_{\mP^{1}}(-2)$
if $\rank Q_{y}=4$, and $\eS^{*}(-1)|_{q_{y}}\simeq\sO_{\mP^{1}}\oplus\sO_{\mP^{1}}(-2)^{\oplus2}$
if $\rank Q_{y}=3$. 

Therefore, $G_{y}$ is as in the statement. \end{proof}

\vspace{0.2cm}
 Similarly to the proof of Proposition \ref{cor:flat}, we can prove
the following:

\begin{prop} \label{cla:flat2} The scheme $\Delta$ is flat over
$V_{\widetilde{\hcoY}}$. The ideal sheaf $\sI_{y}$ of $\Delta_{y}$
in $\hchow$ is $\sI\otimes\sO_{\hchow_{y}}$ for $y\in V_{\widetilde{\hcoY}}$,
where $\hchow_{y}\simeq\hchow$ is the fiber of $\widetilde{\hcoY}\times\hchow\to\widetilde{\hcoY}$
over $y$. Moreover, the exact sequence $(\ref{eqn:fin5})$ remains
to be exact after restricting on $\hchow_{y}$ and gives the following
locally free resolution of $\sI_{y}:$ \begin{equation}
\begin{aligned}0\to & g^{*}\sO_{\mathrm{G}(2,V)}(-2)^{\oplus3}\to g^{*}(\sF^{*}(-2))^{\oplus4}\to\\
 & g^{*}{\ft S}^{2}\sF\oplus g^{*}\sO_{\mathrm{G}(2,V)}(-1)^{\oplus3}\to\sI_{y}\to0.\end{aligned}
\label{eqnarray:Iy}\end{equation}
 \end{prop} 

\begin{proof} By Proposition \ref{cla:fiber}, $\Delta\to\widetilde{\hcoY}$
is equi-dimensional over $V_{\widetilde{\hcoY}}$. Therefore all the
assertions can be proved in the same way of the proof of Proposition
\ref{cor:flat}. \end{proof}

The following computation is standard: 

\begin{lem} \label{cla:compS2F} 

$c_{1}(\ft{S}^{2}\sF^{*})=c_{1}(\sO_{\mathrm{G}(2,V)}(3))$, $c_{2}(\ft{S}^{2}\sF^{*})=2c_{1}(\sO_{\mathrm{G}(2,V)}(1))^{2}+4c_{2}(\sF^{*})$,
and $c_{3}(\ft{S}^{2}\sF^{*})=4c_{1}(\sO_{\mathrm{G}(2,V)}(1))c_{2}(\sF^{*})$.
\end{lem}

\vspace{0.3cm}
 Now we can derive the following:

\begin{prop} \label{prop:Cflat2} The scheme $\Cd$ defined as in
Subsection $\ref{subsection:Cutting}$ is flat over $Y$. Let $C_{y}$
be the fiber of $\Cd$ over a point $y\in Y$. Then $C_{y}$ is a
curve of arithmetic genus $14$ and degree $20$. Moreover, if $X$
and $Y$ are general, then a general $C_{y}$ is smooth. \end{prop}

\vspace{0.1cm}
 \begin{proof} Note that $Y\subset V_{\widetilde{\hcoY}}$. Take
a point $y\in Y$. We only consider the case of $\rank Q_{y}=4$ since
the other case can be studied similarly.

First we show that $\deg\Delta_{y}=20$ with respect to $H_{\hchow}=H_{\mP({\ft S}^{2}\sF)}$.
The degree of $\Delta_{y}$ is evaluated by the Segre class of ${\ft S}^{2}\sF^{*}$
as $s_{3}({\ft S}^{2}\sF^{*})G_{y}$. By the formula $s_{3}({\ft S}^{2}\sF^{*})=c_{3}({\ft S}^{2}\sF^{*})-2c_{1}({\ft S}^{2}\sF^{*})c_{2}({\ft S}^{2}\sF^{*})+c_{1}({\ft S}^{2}\sF^{*})^{3}$
and Lemma \ref{cla:compS2F}, we have $\deg\Delta_{y}=(15c_{1}(\sO_{\mathrm{G}(2,V)}(1))^{3}-20c_{1}(\sO_{\mathrm{G}(2,V)}(1))c_{2}(\sF^{*}))G_{y}$.
By Proposition \ref{cla:fiber}, we have $\deg G_{y}=c_{1}(\sO_{\mathrm{G}(2,V)}(1))^{3}G_{y}=4$.
Note that $c_{2}(\sF^{*})=[\mathrm{G}(2,V_{4})]$ as codimension 2
cycle with some $4$-dimensional space $V_{4}\subset V$. Then we
see that $c_{2}(\sF^{*})G_{y}$ is represented by a conic, which parameterizes
a $\mP^{1}$-family of rulings in the smooth quadric surface $Q_{y}\cap\mP(V_{4})$.
Therefore we have $c_{1}(\sO_{\mathrm{G}(2,V)}(1))c_{2}(\sF^{*})G_{y}=2$.
Consequently, we have $\deg\Delta_{y}=20$.

We show that $C_{y}$ is a curve. Similarly to the proof of Proposition
\ref{prop:Cflat}, $C_{y}$ is a complete intersection in $\Delta_{y}$
by $4$ members of $|H_{\hchow}|$. Therefore $\dim C_{y}\geq1$.
Assume that $\dim C_{y}=2$. Then $\deg C_{y}\leq20$ since $\deg\Delta_{y}=20$.
This is, however, impossible since ${H_{{\hchow}}}|_{X}$ generates
$\Pic X$ modulo torsion and $({H_{{\hchow}}}|_{X})^{3}=35$. Therefore
$C_{y}$ is a curve, and then $\Cd$ is flat over $Y$ as in the proof
of Proposition \ref{prop:Cflat}.

Now we compute the canonical divisor of $C_{y}$. Since $\Delta_{y}\to G_{y}$
is a projective bundle, we have $K_{\Delta_{y}}=-3H_{\hchow}|_{\Delta_{y}}+(g|_{\Delta_{y}})^{*}(c_{1}({\ft S}^{2}\sF^{*})|_{G_{y}}+K_{G_{y}})$.
Since we have seen that $C_{y}$ is a complete intersection in $\Delta_{y}$
by $4$ members of $|H_{\hchow}|$, we have $K_{C_{y}}=H_{\hchow}|_{C_{y}}+(g|_{\Delta_{y}})^{*}(c_{1}({\ft S}^{2}\sF^{*})|_{G_{y}}+K_{G_{y}})|_{C_{y}}$.
Therefore $\deg K_{C_{y}}=H_{\hchow}^{5}\Delta_{y}+H_{\hchow}^{4}(g|_{\Delta_{y}})^{*}(c_{1}({\ft S}^{2}\sF^{*})|_{G_{y}}+K_{G_{y}})$.
We have already computed $H_{\hchow}^{5}\Delta_{y}=20$. By using
the Segre class of ${\ft S}^{2}\sF^{*}$, we have \[
H_{\hchow}^{4}(g|_{\Delta_{y}})^{*}(c_{1}({\ft S}^{2}\sF^{*})|_{G_{y}}+K_{G_{y}})=s_{2}({\ft S}^{2}\sF^{*})|_{G_{y}}(c_{1}({\ft S}^{2}\sF^{*})|_{G_{y}}+K_{G_{y}}).\]
 By Lemma \ref{cla:compS2F}, we have $s_{2}({\ft S}^{2}\sF^{*})=c_{1}({\ft S}^{2}\sF^{*})^{2}-c_{2}({\ft S}^{2}\sF^{*})=7c_{1}(\sO_{\mathrm{G}(2,V)}(1))^{2}-4c_{2}(\sF^{*})$.
By Proposition \ref{cla:fiber}, we have $\sO_{G_{y}}(K_{G_{y}})\simeq\sO_{\mathrm{G}(2,V)}(-3)|_{G_{y}}\otimes p^{*}\sO_{\mP^{1}}(2)$,
where $p$ is the natural morphism $p\colon G_{y}=\mP(\sO_{\mP^{1}}(-1)^{\oplus2}\oplus\sO_{\mP^{1}}(-2))\to q_{y}\simeq\mP^{1}$.
Thus $c_{1}({\ft S}^{2}\sF^{*})|_{G_{y}}+K_{G_{y}}=p^{*}\sO_{\mP^{1}}(2)$.
Hence \[
H_{\hchow}^{4}(g|_{\Delta_{y}})^{*}(c_{1}({\ft S}^{2}\sF^{*})|_{G_{y}}+K_{G_{y}})=(7c_{1}(\sO_{\mathrm{G}(2,V)}(1))^{2}-4c_{2}(\sF^{*}))|_{G_{y}}p^{*}\sO_{\mP^{1}}(2).\]
 Recall that $c_{2}(\sF^{*})|_{G_{y}}$ is represented as a conic
on $\mathrm{G}(2,V)$, which is a section of $p\colon G_{y}\to q_{y}$.
Therefore $(7c_{1}(\sO_{\mathrm{G}(2,V)}(1))^{2}-4c_{2}(\sF^{*}))|_{G_{y}}p^{*}\sO_{\mP^{1}}(2)=6$.
Consequently, we have $\deg K_{C_{y}}=26$, equivalently, $p_{a}(C_{y})=14$.

Let $y$ be a point of $\hcoY$ such that $\rank Q_{y}=4$. Take a
general $4$-plane $P$ in $\mP({\ft S}^{2}V^{*})$ containing $[Q_{y}]$
and define $X$ and $Y$ as before. Note that $y\in Y$. Write $P=\langle Q_{y},Q_{1},\dots,Q_{4}\rangle$.
Let $H_{i}$ be the member of $|H_{\hchow}|$ corresponding to $Q_{i}$.
Note that $\Delta_{y}\subset\Vs_{y}$ by Proposition \ref{cla:Vs}.
Since $P$ is general, $H_{i}$ are general members of $|H_{\hchow}|$.
Therefore, $C_{y}$ is smooth since so is $\Delta_{y}$ and $C_{y}$
is a complete intersection in $\Delta_{y}$ by $H_{1},\dots,H_{4}$.
\end{proof}

One might consider that we can show the derived equivalence by the
Fourier-Mukai functor with the ideal sheaves of the family $\{C_{y}\}_{y\in Y}$.
Actually, computations of Ext groups among the sheaves appearing in
the resolution (\ref{eqnarray:Iy}) are easier than Theorem \ref{thm:digGvan}
(see \cite[Thm.3.4.5]{HoTa3}). However, one technical problem is
involved as follows: When we follow the argument of Section \ref{section:BC},
it is crucial to have the property that, for \textit{any} distinct
two points $y_{1}$ and $y_{2}$ of $Y$, the hyperplane sections
$\Vs_{y_{1}}$ and $\Vs_{y_{2}}$ of $X$ are different (cf.~Lemma
\ref{cla:tower}). This, however, does not hold for $y_{1}$ and $y_{2}$
such that their images on $H$ coincides. This forced us to choice
$\{C_{x}\}_{x\in X}$ in our proof.

\vspace{0.2cm}
 \begin{rem} It should be interesting to find the curves $C_{y}\,(y\in Y)$
of genus 14 and degree 20 in the table of the BPS numbers of $X$
\cite[Table 2]{HoTa1}. In fact, we read from the table the counting
number of the curves of genus 14 and degree 20 as \[
n_{14}^{X}(20)=500.\]
 In a similar way to the BPS number $n_{3}^{Y}(5)=100$ discussed
in Introduction, we may arrange this number as $n_{14}^{X}(20)=(-1)^{\dim Y}e(Y)\times10$.
This time, however, it is not clear whether the factor $10$ has a
nice interpretation from the geometry of $X$. Nevertheless, we expect
that the number $n_{14}^{X}(20)=500$ is 'counting' the Euler numbers
of the parameter spaces of generically smooth family of curves on
$X$ by general properties of the BPS numbers \cite{GV}, since we
were able to verify $n_{14}^{X}(21)=0$ after rather heavy calculations
using mirror symmetry. \end{rem}

\vspace{0.2cm}
 \begin{rem} In the Grassmann-Pfaffian case due to \cite{BC,Ku2},
the constructions of curves $\{C_{x}\}_{x\in X}$ and $\{C_{y}\}_{y\in Y}$
are more straightforward. Assume $X$ is a smooth linear section Calabi-Yau
threefold of $\mathrm{G}(2,7)$, which is embedded in $\mP(\wedge^{2}\mC^{7})$,
and $Y$ is the corresponding (smooth) orthogonal linear section Calabi-Yau
threefolds of $\mathrm{Pf}(7)$ in $\mP(\wedge^{2}(\mC^{*})^{7})$.
Let us write by $[\xi_{x}]=[\xi_{x}^{(1)},\xi_{x}^{(2)}]\simeq\mP^{1}$
the line corresponding to a point $x\in\mathrm{G}(2,7)$, and by $[\eta_{y}]$
a skew symmetric matrix ${\rm rank}\,\eta_{y}\leq4$ corresponding
to $y\in\mathrm{Pf}(7)$. Then the incidence relation used in \cite{BC,Ku2}
is $\Delta=\{([\xi_{x}],[\eta_{y}])|\dim(\xi_{x}\cap{\rm Ker}(\eta_{y}))\geq1\}$.
In this case, the fiber $\Delta_{y}$ of $\Delta\to\mathrm{Pf}(7)$
over $y$ is given by a Schubert cycle $\sigma_{3}$ in $\mathrm{G}(2,7)$
of codimension 3 if ${\rm rank}(\eta_{y})=4$, and simplifies the
proof of the derived equivalence using the family of curves $\{C_{y}\}_{y\in Y}=\{\Delta\cap(X\times\{y\})\}_{y\in Y}$.
As discussed in Introduction, $C_{y}$ is generically a smooth curve
on $X$ of genus 6 and degree 14. The fiber $\Delta_{x}$ of the other
fibration $\Delta\to\mathrm{G}(2,7)$ is easy to be described. It
turns out that \[
\Delta_{x}=\{([\xi_{x}],[\eta_{y}])\mid(\eta_{y}\,\xi_{x}^{(1)})\wedge(\eta_{y}\,\xi_{x}^{(2)})=0\}.\]
 When $X$ and $Y$ are generic and smooth, we can verify by \textsl{Macaulay2}
that $C_{x}=\Delta\cap(\{x\}\times Y)\;(x\in X)$ is generically a
smooth curve on $Y$ of genus 11 and degree 14. The corresponding
BPS number is, unfortunately, outside of the tables available in literatures
(see \cite[Section 4]{HoTa1}). \end{rem}

\vspace{0.3cm}

\subsection{The fundamental groups and the Brauer groups of $X$ and $Y$}

\label{subsection:fund}~

Finally, it should be worth while discussing about non-invariance
of the fundamental groups and the Brauer groups by the derived equivalence
between a Reye congruence $X$ and the double symmetroid $Y$ orthogonal
to $X$.

As for the fundamental groups, we have $\pi_{1}(X)\simeq\mZ_{2}$
by \cite[Prop.3.5.3]{HoTa3}, and $\pi_{1}(Y)\simeq0$ by {[}ibid.~Prop.4.3.4{]}.
To our best knowledge, this seems to be the second example of pairs
of derived equivalent Calabi-Yau threefolds with different fundamental
groups (see \cite{S}).

As for the Brauer groups, we follow the argument of \cite{A}: By
\cite{BK}, the Atiyah-Hirzebruch spectral sequence gives a short
exact sequence for any Calabi-Yau threefold $\Sigma$: \begin{equation}
0\to H_{1}(\Sigma,\mZ)\to K_{\mathrm{top}}^{1}(\Sigma)_{\mathrm{tors}}\to\mathrm{Br}\,(\Sigma)\to0,\label{eq:AT}\end{equation}
 where $K_{\mathrm{top}}(\Sigma)=K_{\mathrm{top}}^{0}(\Sigma)\oplus K_{\mathrm{top}}^{1}(\Sigma)$
is the topological $K$-group and the subscript 'tors' means the torsion
part. As we mentioned above, we have $H_{1}(X,\mZ)\simeq\pi_{1}(X)\simeq\mZ_{2}$
and $H_{1}(Y,\mZ)\simeq0$. By \cite[\S 2.2]{AT}, $K_{\mathrm{top}}^{1}(X)\simeq K_{\mathrm{top}}^{1}(Y)$
since $X$ and $Y$ are derived equivalent. Therefore, by (\ref{eq:AT}),
we have the following relation between the Brauer groups of $X$ and
$Y$: \begin{equation}
0\to\mZ_{2}\to\mathrm{Br}\,(Y)\to\mathrm{Br}\,(X)\to0.\label{eq:BrBr}\end{equation}
 We have shown $\mathrm{Br}\,(Y)$ contains a nonzero $2$-torsion
element in Proposition \ref{prop:Brauer}. If $\mathrm{Br}\,(Y)\simeq\mZ_{2}$,
we will have $\mathrm{Br}\,(X)\simeq0$ by (\ref{eq:BrBr}), which
indicates that the Brauer groups are not invariant under the derived
equivalence.

\vspace{2cm}

\vspace{0.5cm}

{\footnotesize \noindent Graduate School of Mathematical Sciences,
University of Tokyo, Meguro-ku,~Tokyo 153-8914,~Japan }{\footnotesize \par}

{\footnotesize \noindent e-mail addresses: hosono@ms.u-tokyo.ac.jp,
takagi@ms.u-tokyo.ac.jp}{\footnotesize \par}

\end{document}